\documentclass[12pt]{amsart} 
\usepackage[hoffset=0in,left=1in,right=1in]{geometry} 
\usepackage{amssymb,amsmath,amsthm,mathrsfs}
\usepackage[nocompress,nospace,noadjust]{cite}

\usepackage{dsfont}
\usepackage{bbm}
\usepackage{graphicx}
\graphicspath{ {./figs/} }
\usepackage{subcaption}
\usepackage{hyperref}
\usepackage{comment}
\usepackage{url}
\usepackage[colorinlistoftodos]{todonotes}

\usepackage{enumitem}
\usepackage{xfrac}
\usepackage{tikz-cd}

\newcommand{\field}[1]{\mathbb{#1}}
\DeclareMathOperator{\C}{\field{C}}
\DeclareMathOperator{\R}{\field{R}}

\DeclareMathOperator{\N}{\field{N}}

\renewcommand{\P}{\field{P}}
\DeclareMathOperator{\E}{\field{E}}

\newcommand{\parens}[1]{\left(#1\right)}
\newcommand{\set}[1]{\left\{#1\right\}}

\newcommand{\given}{\,|\,}

\DeclareMathOperator{\var}{Var}

\DeclareMathOperator{\dist}{dist}

\DeclareMathOperator{\Real}{Re}
\DeclareMathOperator{\Imag}{Im}
\DeclareMathOperator{\tr}{Tr}

\DeclareMathOperator{\corr}{Corr}
\DeclareMathOperator{\cov}{Cov}
\DeclareMathOperator{\rank}{rank}

\DeclareMathOperator{\Sparse}{Sparse}
\DeclareMathOperator{\supp}{supp}
\DeclareMathOperator{\Incomp}{Incomp}
\DeclareMathOperator{\Comp}{Comp}
\DeclareMathOperator{\Col}{Col}
\DeclareMathOperator{\Span}{Span}

\newtheorem{thm}{Theorem}
\numberwithin{thm}{section}
\newtheorem{lem}[thm]{Lemma}
\newtheorem{defn}[thm]{Definition}
\newtheorem{prop}[thm]{Proposition}
\newtheorem{cor}[thm]{Corollary}
\newtheorem{ass}[thm]{Assumption}
\newtheorem{remark}[thm]{Remark}

\newcommand{\norm}[1]{\left\lVert#1\right\rVert}
\newcommand{\cal}[1]{\mathcal{#1}}
\newcommand{\abs}[1]{\left|#1\right|}
\newcommand{\ceil}[1]{\left\lceil#1\right\rceil}
\newcommand{\floor}[1]{\left\lfloor#1\right\rfloor}

\DeclareMathOperator{\eps}{\varepsilon}

\title[Sparse elliptic random matrices]{The limiting spectral distribution for sparse elliptic random matrices}

\author[J. Carpenter]{Jackson Carpenter}
\address{Department of Mathematics\\ University of Colorado\\ Campus Box 395\\ Boulder, CO 80309-0395\\USA}
\email{jackson.carpenter@colorado.edu}
\author[S. O'Rourke]{Sean O'Rourke}
\address{Department of Mathematics\\ University of Colorado\\ Campus Box 395\\ Boulder, CO 80309-0395\\USA}
\email{sean.d.orourke@colorado.edu}
\thanks{S. O'Rourke has been partially supported by NSF CAREER grant DMS-2143142. }

\date{} 

\begin{document}
	
	\begin{abstract}
		This paper studies sparse elliptic random matrix models which generalize both the classical elliptic ensembles and sparse i.i.d.\ matrix models by incorporating correlated entries and a tunable sparsity parameter $p_n$. Each $n\times n$ matrix $X_n$ is formed by entry-wise multiplication of an elliptic random matrix by an elliptic matrix of Bernoulli($p_n$) variables, where $np_n$ tends to infinity at a sufficient rate, allowing for interpolation between dense and sparse regimes. The main result establishes that after appropriate normalization, the empirical spectral measures of these matrices converge weakly in probability to the uniform measure on a rotated ellipsoid in the complex plane as the dimension $n$ tends to infinity. Interestingly, the shape of the limiting ellipsoid depends not only on the mirrored entry-wise correlation structure, but also non-trivially on the sparsity limit $p=\lim\limits_{n\to\infty}p_n$. 
	\end{abstract}
	
	\maketitle
	\tableofcontents
	\section{Introduction}	
	
	
	For an $n\times n$ matrix $M_n$ over $\C$, we consider the empirical spectral measure (ESM) of $M_n$:
	\[
	\mu_{M_n}=\frac{1}{n}\sum_{i=1}^n\delta_{\lambda_{i}(M_n)}.
	\]
	Here $\lambda_1(M_n),\ldots,\lambda_n(M_n)\in\C$ are the $n$ unordered eigenvalues of $M_n$ counted with algebraic multiplicity and $\delta_z$ is a Dirac mass centered at $z\in\C$. If $M_n$ is a \textit{random} matrix model, this necessarily random probability measure on $\C$ (or $\R$, if the eigenvalues of $M_n$ are guaranteed to be real, say) captures the empirical location of the eigenvalues of $M_n$. One of the foundational questions in random matrix theory concerns the limiting behavior of $\mu_{c_n^{-1}M_n}$, as $n\to\infty$, where $c_n$ is an appropriate normalization constant for the matrix model.
	
	One of the most foundational results in this field concerns a large family of matrix models: \textit{i.i.d.\ random matrices}. In such models, $M_n = (\xi_{ij})_{1\leq i,j\leq n}$, where $(\xi_{ij})_{i,j\geq 1}$ is an infinite double array of independent random variables that are identically distributed as a real or complex-valued random variable $\xi$ with mean zero and unit variance. For such matrix models $M_n$, we have \textit{the circular law}, which states that $\mu_{n^{-1/2}M_n}$ converges weakly almost surely as $n\to\infty$ to the uniform measure on the unit disk in the complex plane \cite{taouniversality}. In other words, for every bounded and continuous function $f:\C\to\R$,
	\[
	\frac{1}{n}\sum_{i=1}^nf(\lambda_i(n^{-1/2}M_n))=\int f\,d\mu_{n^{-1/2}M_n}\to\frac{1}{\pi}\int_{\abs{z}<1}f(z)\,dz
	\]
	almost surely as $n\to\infty$.
	This result culminates a long line of research \cite{edelmangaussian, ginibregaussian,mehtagaussian,girko1,girko2,girko3,girko4,girko5,girko6,BaiCircular,rudtiksingularity,tao2010random,gotzecircularlaw,zhoucirc,tao2008randommatricescircularlaw}. This list is by no means exhaustive and we refer the reader to \cite{aroundthecirc} for more history on the development of the circular law to the universal, almost sure result we have today.
	
	\subsection{Background and setup}
	
	\subsubsection{Elliptic random matrices}Elliptic random matrix models interpolate between Wigner matrix models \cite{introtoRMT} and i.i.d.\ models. More specifically, the $(i,j)$ entry of an elliptic random matrix $M_n$ is independent of all other entries except possibly the $(j,i)$ entry.
	The study of such elliptic matrix models was introduced by Girko \cite{girkoelliptic1,girkoelliptic2,girkoelliptic3,GirkoElliptic4,girkoelliptic5,girkoelliptic6,girkoelliptic7,girkoelliptic8}. We give the definition of an elliptic random matrix model.
	
	\begin{defn}[Elliptic random matrix model]\label{def:ellipticmodel}
		Let $(\psi_1,\psi_2)$ be a complex-valued random vector where $\psi_1$ and $\psi_2$ have mean 0 and variance 1. We say that the sequence of random matrices $(X_n)_{n\geq1}$ is an \textbf{elliptic random matrix model with atom variable $(\psi_1,\psi_2)$} if the entries $(X_n)_{ij} = x_{ij}$ of the $n\times n$ matrix $X_n$ satisfy the following:
		\begin{enumerate}[label = (\roman*)]
			\item (Independence and identical distributions) We assume that $
			\{(x_{ij},x_{ji}):1\leq i<j\}$ is an infinite double array of independent random elements and each pair $(x_{ij},x_{ji})$, $1\leq i<j$ is distributed as $(\psi_1,\psi_2)$.
			
			\item (Flexibility on the diagonal) We assume that $\{x_{ii}:i\geq1\}$ is a sequence of i.i.d.\ complex-valued random variables having mean 0 and finite variance that is independent of all off-diagonal entries of the matrix.
		\end{enumerate}
	\end{defn}
	
	Much work has been done on the study of real and complex elliptic matrix models. For sources dealing with non-Hermitian or elliptic random matrices coming from the physics literature, see \cite{physics2,physics3,physicspaper1,dutta2025spectraldistributionsparsegaussian,Forrester,Kruger}. For a non-exhaustive list of more work in elliptic random matrices, we refer the reader to \cite{Akemann_2023,Bender,newelliptic1,fyodorov,fyodorov1,nonhermitianlevy,Sommers1,seanping,O_Rourke_2015,largeelliptic,G_tze_2015,andrewelliptic,newelliptic2,Byun_2021} and references therein.
	
	In the context of limiting spectral distributions for elliptic random matrices, Nguyen and the second author prove a general result in \cite{ellipticlaw}, which builds on the work done in \cite{naumov} by Naumov for the real case:
	
	\begin{thm}\label{thm:ellipticlaw}
		Suppose $(X_n)_{n\geq 1}$ is an elliptic random matrix model coming from Definition \ref{def:ellipticmodel} where the real-valued atom variable $(\psi_1,\psi_2)$ has correlation $\E[\psi_1\psi_2]=\rho\in(-1,1)$. Then the ESM $\mu_{n^{-1/2}X_n}$ converges weakly almost surely to $\mu_\rho$, where $\mu_\rho$ is the probability measure on $\C$ with density $\frac{1}{\pi(1-\rho^2)}\mathbbm{1}_{E_\rho}$ and $E_\rho$ is the ellipsoid
		\begin{equation}\label{eq:ellipse}
			E_\rho = \set{z\in\C:\frac{\Real(z)^2}{(1+\rho)^2}+\frac{\Imag(z)^2}{(1-\rho)^2}\leq 1}.
		\end{equation}
	\end{thm}
	
	Under technical assumptions on a complex-valued atom variable $(\psi_1,\psi_2)$, \cite{ellipticlaw} also proves weak almost sure convergence of the ESMs $\mu_{n^{-1/2}X_n}\to\mu_{\rho}$ as $n\to\infty$.
	

	\subsubsection{Sparse random matrices}
	Beyond elliptic random matrix models are those that this paper is concerned with: elliptic random matrix models that are sparse as well. Sparsity in random matrices can be applied to any underlying model $M_n$ by multiplying each entry of $M_n$ by independent (or correlated) Bernoulli($p_n$) variables: $\delta^{(n)}\sim$ Bernoulli($p_n$) if
	\[
	\P(\delta^{(n)} = 1) = p_n \quad\text{and}\quad \P(\delta^{(n)}=0)=1-p_n
	\]
	for some $p_n\in[0,1]$.
	In the sparse i.i.d.\ setting, for a certain level of sparsity depending on the decay rate of $p_n$, we still have weak convergence of the normalized ESMs to the uniform measure on the unit disk in the complex plane. Results of this type can be seen for various levels of sparsity in various Hermitian or non-Hermitian settings in \cite{Sparsesymmetricrogers,tao2008randommatricescircularlaw,gotzecircularlaw,wood,invertsparse,circlawforsparse, Sparsesymmetricrogers,osmanSparse}. For minimal assumptions on sparsity in i.i.d.\ matrix models, these results can be seen for the real case in \cite{rudelsontik} due to Rudelson and Tikhomirov or in \cite{revisited} due to Sah, Sahasrabudhe, Sawhney,  with a slightly different flavor of proof than the classical Hermitization argument to deal with the complex case. For sparse i.i.d.\ random matrix models, we have the following, which is \cite[Theorem 1.4]{revisited}:
	
	\begin{thm}\label{thm:sparsecircularlaw}
		Let $(M_n)_{n\geq 1}$ be a sparse i.i.d.\ random matrix model. Namely, the entries of $M_n$ are $(M_n)_{ij} = \delta_{ij}^{(n)}x_{ij}$ satisfying the following:
		\begin{itemize}
			\item[(i)] (Identical distributions) Each $x_{ij}$ (for $i,j\geq1$) is an i.i.d.\ copy of a complex-valued random variable $\xi$ with mean zero and variance 1. For each $n\geq 1$ and $1\leq i,j\leq n$, each $\delta_{ij}^{(n)}$ is an i.i.d.\ copy of a Bernoulli($p_n$) variable. 
			
			\item[(ii)] (Independence)
			$\{x_{ij}:i,j\geq1\}\cup\{\delta_{ij}^{(n)}:n\geq 1, 1\leq i,j\leq n\}$ is a collection of independent random elements.
			
			\item[(iii)] (Sparsity) $p_n\to 0$ and $np_n\to\infty$ as $n\to\infty$.
		\end{itemize}
		
		Then the ESMs $\mu_{(np_n)^{-1/2} M_n}$ converge weakly in probability to the uniform measure on the unit disk in the complex plane.
	\end{thm}
	
	In addition to the level of sparsity as seen in Theorem \ref{thm:sparsecircularlaw}, there are also results concerning the ``supersparse" regime, when $p_n = d/n$, where $d$ is a fixed constant parameter. In \cite{supersparse}, the empirical spectral measure $\mu_{M_n}$ of the i.i.d.\ Bernoulli matrices $(M_n)_{ij} = \delta_{ij}^{(n)}$ are shown to converge weakly in probability to a measure $\mu_d$ on $\C$ depending on $d$.

	\subsection{Definitions and setup}

	For our sparse elliptic model, we define an elliptic matrix of Bernoulli variables and introduce technical assumptions that we need for our main results.
	
	\begin{defn}[Elliptic Bernoulli matrix]\label{def:Bernoullimatrix}
		Fix a sequence $(p_n)_{n\geq1}$ of sparsity parameters with $p_n\in(0,1]$ for all $n\geq 1$ and a  parameter $\tau\in[-1,1]$. For each $n\geq 1$, let $(\eta_1^{(n)},\eta_2^{(n)})$ be a pair of Bernoulli($p_n$) random variables. For any $n\geq 1$, we say that a matrix $B_n = (\delta^{(n)}_{ij})_{1\leq i,j\leq n}$ is an \textbf{elliptic Bernoulli matrix with sparsity parameters $(p_n)_{n\geq1}$ and $\tau$} if the following conditions hold:
		\begin{enumerate}[label = (\roman*)]
			\item\label{item:sparsitylevel} (Sparsity) $p_n\to p\in[0,1]$ and $np_n\to\infty$ as $n\to\infty$.
			\item The quantities $\tau_n = \frac{1}{p_n}\cov(\eta_1^{(n)},\eta_2^{(n)})$\footnote{$\cov(X,Y) = \E X\overline{Y} - \E X \E \overline{Y}$ denotes the covariance between two (complex-valued) random variables.} converge to $\tau$ as $n\to\infty$. 
			\item (Independence and identical distributions) We assume that
			\[
			\set{(\delta_{ij}^{(n)},\delta_{ji}^{(n)}):n\geq 1,1\leq i<j\leq n}\cup\set{\delta_{ii}^{(n)}:n\geq 1, 1\leq i\leq n}
			\]
			is a collection of independent random elements. For all $n\geq 1$ and all $1\leq i<j\leq n$, $(\delta_{ij}^{(n)},\delta_{ji}^{(n)})$ is distributed as $(\eta_1^{(n)},\eta_2^{(n)})$ and each $\delta_{ii}^{(n)}$ is a Bernoulli($p_n$) variable.
		\end{enumerate}
	\end{defn}
	
	\begin{remark}\label{rem:pnequalsone}
		The definition of $\tau_n$ above allows us to treat the degenerate case of $p_n=1$ for all $n\geq 1$. In such a degenerate case, $\tau_n=0$ for all $n\geq 1$ and so $\tau=0$. On the other hand if $p_n<1$ for each $n\geq 1$, then the correlations $\rho_{2,n} := \corr(\eta_1^{(n)},\eta_2^{(n)})$ exist and satisfy $\tau_n = \rho_{2,n}(1-p_n)$. If, furthermore, the $p_n$ do not converge to 1, then assumptions (i) and (ii) above together imply the convergence of the correlations $\rho_{2,n}$.
	\end{remark}
	
	Condition \textit{\ref{item:sparsitylevel}} above is the weakest assumption needed to prove many of the results in this paper. However, to control the least singular values as well as intermediate singular values, we need a more technical assumption on the growth rate of the quantity $np_n$:
	\begin{ass}[Growth rate of $np_n$]\label{ass:sparsitylevel}
		The sequence of sparsity parameters $(p_n)_{n\geq 1}$ satisfies Assumption \ref{ass:sparsitylevel} if there exists a sequence $(q_n)_{n\geq 1}$ such that $$q_n\log\log n\to\infty \quad\text{and}\quad np_n\geq n^{q_n}.$$
	\end{ass}
	
	\begin{remark}\label{rem:explicitsparsitylevels} We give some explicit sparsity levels for which the main results of this paper are guaranteed. To begin, our work generalizes the dense case: with $np_n = n$, we may take $q_n = 1$ so that $q_n\log\log n \to\infty$. An explicit sparse example is given by $p_n$ satisfying $np_n \geq n^\theta$ for a fixed $0<\theta \leq 1$.
		We can do a bit better than this sparsity level imposed in \cite{gotzecircularlaw} for the i.i.d.\ case with an explicit example, taking $np_n \geq n^{1/\sqrt{\log\log n}}$, for instance.
	\end{remark}
	
	Now for the main definition of this paper.
	
	\begin{defn}[Sparse elliptic random matrix model]\label{def:sparseellipticmodel}
		Let $(\xi_1,\xi_2)$ be a complex-valued random vector where $\xi_1$ and $\xi_2$ have mean 0, variance 1, satisfy $\abs{\E[\xi_1\overline{\xi_2}]}<1$, and have $\E[\xi_1\xi_2] = \rho_1e^{i\theta}$ with $0\leq\rho_1<1$ and $0\leq \theta<2\pi$.  Let $(x_{ij})_{1\leq i, j}$ be an infinite double array of complex-valued random variables and for each $n\geq 1$, let $B_n = (\delta^{(n)}_{ij})_{1\leq i,j\leq n}$ be an elliptic Bernoulli matrix with sparsity parameters $(p_n)_{n\geq1}$ and $\tau$ coming from and inheriting the notation from Definition \ref{def:Bernoullimatrix}.  
		For each $n\geq1$, we define the random $n\times n$ matrix $X_n = (\delta_{ij}^{(n)}x_{ij})_{1\leq i, j\leq n}$. We refer to the sequence of random matrices $(X_n)_{n\geq1}$ as a \textbf{sparse elliptic random matrix model with atom variable $(\xi_1,\xi_2)$ and sparsity parameters $(p_n)_{n\geq1}$ and $\tau$} if the following additional conditions hold:
		\begin{enumerate}[label = (\roman*)]
			\item (Independence and identical distributions) We assume that
			\begin{align*}
				\{(x_{ij},x_{ji}):1\leq i<j\}&\cup
				\{(\delta^{(n)}_{ij},\delta^{(n)}_{ji}):n\geq 1, 1\leq i< j\leq n\}\,\cup\{\delta_{ii}^{(n)}:n\geq 1, 1\leq i \leq n\}
			\end{align*}
			is a collection of independent random elements. Each pair $(x_{ij},x_{ji})$, $1\leq i<j$ is distributed as $(\xi_1,\xi_2)$.
			
			\item (Flexibility on the diagonal) We assume that $\{x_{ii}:i\geq1\}$ is a sequence of i.i.d.\ complex-valued random variables having mean 0 and finite variance that is independent of all other sources of randomness.
		\end{enumerate}
	\end{defn}
	
	For what is to come, we will also want to define a ``dense" elliptic random matrix model:
	
	\begin{remark}\label{rem:ellipticmodel}
		If $(X_n)_{n\geq1}$ is a sparse elliptic random matrix model coming from Definition \ref{def:sparseellipticmodel} with atom variable $(\psi_1,\psi_2)$ and sparsity parameters $p_n=1$ for all $n\geq 1$, then $(X_n)_{n\geq1}$ is a \textit{dense} elliptic random matrix model coming from Definition \ref{def:ellipticmodel}. In other words,  Definition \ref{def:sparseellipticmodel} generalizes Definition \ref{def:ellipticmodel}.
	\end{remark}
	
	\begin{remark}
		The assumption on the moments of the diagonal elements---namely, our mean 0 and finite variance assumptions---can be relaxed in Definition \ref{def:ellipticmodel} and Definition \ref{def:sparseellipticmodel}, but the authors keep this assumption for simplicity.
	\end{remark}
	
	
	We remark that Definition \ref{def:sparseellipticmodel} also captures the i.i.d.\ matrix  model and sparse i.i.d.\ matrix model when we take $\xi_1$ and $\xi_2$ to be i.i.d.\ copies of some complex-valued random variable $\xi$ and the mirrored off-diagonal Bernoulli elements to be independent.

	The main aim of this paper is to prove Theorem \ref{thm:maintheorem}, a result about the limiting behavior of ESMs coming from an elliptic matrix model that is also sparse. See Figures \ref{fig:sfig1} and \ref{fig:sfig2} for a visualization of how the eigenvalues of sparse and dense elliptic matrices behave similarly. Figure \ref{fig:sfig1} features a plot of eigenvalues of a dense elliptic random matrix coming from Definition \ref{def:ellipticmodel} with real standard Gaussian atom variables having correlation $\rho = 0.5$. Figure \ref{fig:sfig2} features a plot of eigenvalues of a sparse elliptic random matrix coming from Definition \ref{def:sparseellipticmodel}. Here the atom variables are again distributed as real standard Gaussians with correlation $\rho_1=0.5$ and the Bernoulli variables have mean $p_n = n^{-1/2}$ and correlation $\rho_{2,n} = 0.5$.
	
	\begin{figure}
		\begin{subfigure}{.5\textwidth}
			\centering
			\includegraphics[width=.8\linewidth]{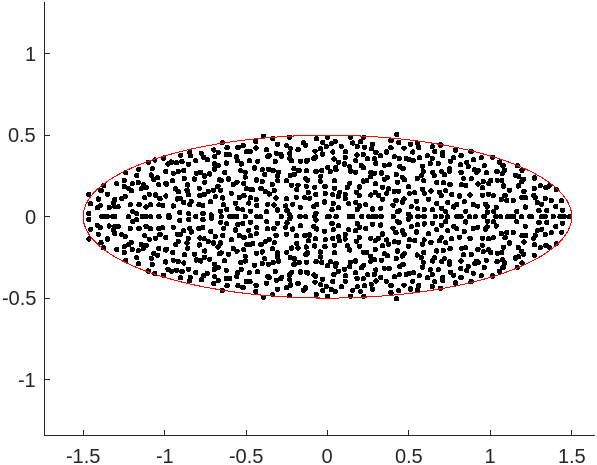}
			\caption{$\rho_1=0.5$, $\theta=0$, $p_n=1$,$\tau_n=0$}
			\label{fig:sfig1}
		\end{subfigure}%
		\begin{subfigure}{.5\textwidth}
			\centering
			\includegraphics[width=.8\linewidth]{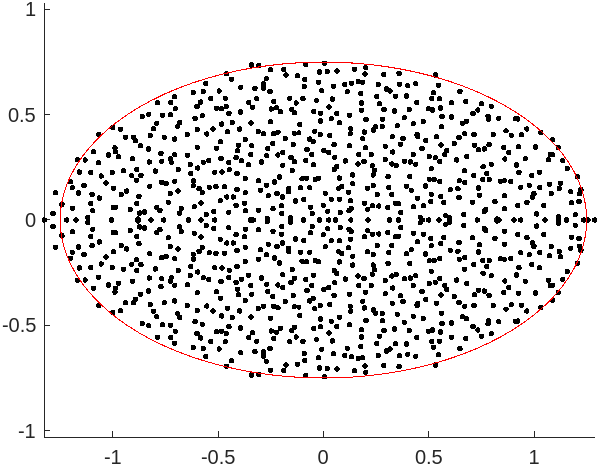}
			\caption{$\rho_1=0.5$, $\theta=0$, $p_n=n^{-1/2}$, $\rho_{2,n}=0.5$, $\tau_n = 0.5(1-n^{-1/2})$}
			\label{fig:sfig2}
		\end{subfigure}
		\caption{Plots of eigenvalues of $1000\times 1000$ real Gaussian elliptic matrices $(np_n)^{-1/2}X_n$ (so $n=1000$) with $\xi_1,\xi_2\sim\cal{N}_{\R}(0,1)$. The red ellipse plotted in Figure \ref{fig:sfig1} is $E_{0.5}$ and the red ellipse plotted in Figure \ref{fig:sfig2} is $E_{0.25}$.}
		\label{fig:fig}
	\end{figure}
	
	\subsection{Main results}
	
	The main theorem of this paper is concerned with the convergence of the ESMs $\mu_{M_n}$ of the normalized sparse elliptic random matrix $M_n:=(np_n)^{-1/2}(X_n+F_n)$. Here $F_n$ is a deterministic $n\times n$ matrix over $\C$.
	This $M_n$ notation will be used subsequently throughout the paper to denote adding the perturbation $F_n$ to our sparse elliptic random matrix $X_n$ and scaling by $(np_n)^{-1/2}$. Broadly, our result proves the weak, in probability convergence of the ESMs $\mu_{M_n}$ to the uniform measure supported on a rotated ellipsoid in the complex plane, as $n\to\infty$.
	
	\begin{thm}\label{thm:maintheorem}
		Let $(X_n)_{n\geq1}$ be a sparse elliptic random matrix model coming from Definition \ref{def:sparseellipticmodel} given by the complex-valued atom variable $(\xi_1,\xi_2)$ satisfying $\abs{\E[\xi_1\overline{\xi_2}]}<1$ and $\E[\xi_1\xi_2] = \rho_1e^{i\theta}$ ($\rho_1\in[0,1)$, $\theta\in[0,2\pi)$), sparsity parameters $(p_n)_{n\geq1}$ satisfying Assumption \ref{ass:sparsitylevel} with limit $p$, and parameter $\tau$. Furthermore, if $(F_n)_{n\geq 1}$ is a sequence of deterministic matrices over $\C$ satisfying $\rank(F_n) = o(n)$\footnote{In other words, $\frac{1}{n}\rank(F_n)\to0$ as $n\to\infty$.} and 
		\begin{equation}\label{eq:boundedhilbertschmidt}
			\sup_n\frac{1}{n^2p_n}\norm{F_n}_{HS}^2 = \sup_n\frac{1}{n^2p_n}\sum_{1\leq i,j\leq n}\abs{(F_n)_{ij}}^2<\infty,
		\end{equation} then with $M_n = (np_n)^{-1/2}(X_n+F_n)$,
		\[
		\mu_{M_n}\to \mu_{\rho_1,\theta,\tau,p} \text{ weakly in probability as }n\to\infty.
		\]
		Here $\mu_{\rho_1,\theta,\tau,p}$ is the uniform measure on the ellipsoid
		\[
		e^{i\theta/2}E_{\rho_1(\tau+p)} = \set{e^{i\theta/2}z:z\in E_{\rho_1(\tau+p)}},
		\]
		where $E_{\rho_1(\tau+p)}$ is as in \eqref{eq:ellipse}.
	\end{thm}
	
	\begin{remark}
		If the reader wishes to perturb the matrix $X_n$ \textit{after} scaling so that $M_n = (np_n)^{-1/2}X_n +F_n$ for some deterministic matrix $F_n$ over $\C$, then $F_n$ still needs $\rank(F_n) = o(n)$. On the other hand, Theorem \ref{thm:maintheorem} remains the same, but we must replace \eqref{eq:boundedhilbertschmidt} with 
		$$\sup_n\frac{1}{n}\norm{F_n}_{HS}^2 = \sup_n\frac{1}{n}\sum_{1\leq i,j\leq n}\abs{(F_n)_{ij}}^2<\infty.$$
	\end{remark}
	
	\begin{remark}
		In the event that $p_n<1$ for all $n\geq 1$ and the correlations $\rho_{2,n}$ coming from Definition \ref{def:Bernoullimatrix} converge to some limit $\rho_2$ as $n\to\infty$, then $\tau$ appearing in Theorem \ref{thm:maintheorem} may be replaced with $\tau=\rho_2(1-p)$.
	\end{remark}
	
	\begin{remark}
		Theorem \ref{thm:maintheorem} needs Assumption \ref{ass:sparsitylevel} to control both the least singular value of $M_n$ in Section \ref{sec:leastsingularvalue} and Section \ref{sec:lsvappendix} as well as intermediate singular values in Section \ref{sec:intermediatesingvals}. All other results of the paper require only the weaker assumption of $np_n\to\infty$.
	\end{remark}
	
	Aside from dealing with the complex version of $\E[\xi_1\xi_2] = \rho_1e^{i\theta}$ appearing in the hypothesis of Theorem \ref{thm:maintheorem}, one of the more unexpected and interesting  aspects of this result is the dependence on the sparsity limits $p$ and $\tau$. For a dense elliptic random matrix $\frac{1}{\sqrt{n}}X_n$ coming from Definition \ref{def:ellipticmodel}, the correlation of mirrored off-diagonal entries is $\E[(X_n)_{12}(X_n)_{21}] = \E[\psi_1\psi_2]$. In the case that this correlation $\rho = \E[\psi_1\psi_2]\in(-1,1)$ is real, this parameter $\rho$ defines the ellipsoid $E_\rho$ in \eqref{eq:ellipse}. In the sparse regime, we expect to recover this parameter of the limiting ellipsoid as well. If we view a sparse elliptic random matrix $\frac{1}{\sqrt{np_n}}X_n$ (coming from Definition \ref{def:sparseellipticmodel}) as $\frac{1}{\sqrt{np_n}}X_n=\frac{1}{\sqrt{n}}\left(\frac{1}{\sqrt{p_n}}X_n\right)$, then we compute, via independence between the Bernoulli and non-Bernoulli random variables:
	\begin{equation}\label{eq:mixedmoment}
		\E\left[\left(\frac{1}{\sqrt{p_n}}X_n\right)_{12}\left(\frac{1}{\sqrt{p_n}}X_n\right)_{21}\right]=\frac{1}{p_n}\E[\eta_{1}^{(n)}\xi_{1}\eta_2^{(n)}\xi_2] = \rho_1e^{i\theta}(\tau_n+p_n)\to\rho_1e^{i\theta}(\tau+p).
	\end{equation}
	Modulo the appearance of the phase shift $e^{i\theta}$ (which we will deal with in the next section), we may read off the dependence of $\rho_1$, $\tau$, and $p$ in the formula of the ellipsoid $E_{\rho_1(\tau+p)}$.
	See Figure \ref{fig:pfigs} for some simulations for concrete choices of the sparsity limit $p=\lim\limits_{n\to\infty} p_n$.

	\subsection{Corollaries and conjectures}
	
	The main result of this paper recovers many of the classical results for i.i.d.\ and elliptic random matrices in both dense and sparse regimes at high level of generality in the atom variables that make up our model. 
	We give some corollaries that are realizations of the main theorem with specific conditions on our parameters. 
	
	To recover a result akin to Theorem \ref{thm:sparsecircularlaw} but with our correlated entries, we set $p=\tau=0$ for our Bernoulli variables.
	
	\begin{cor}\label{cor:sparseiidrecovered}
		Suppose $(X_n)_{n\geq1}$ is a sparse elliptic random matrix model coming from Definition \ref{def:sparseellipticmodel} given by the complex-valued atom variable $(\xi_1,\xi_2)$ satisfying $\abs{\E[\xi_1\overline{\xi_2}]}<1$ and $\E[\xi_1\xi_2]=\rho_1e^{i\theta}$ with $\rho_1\in[0,1)$ and $\theta\in[0,2\pi)$, sparsity parameters $(p_n)_{n\geq1}$ satisfying Assumption \ref{ass:sparsitylevel} with limit $p=0$, and $\tau=0$. Then the ESMs $\mu_{(np_n)^{-1/2}X_n}$ converge weakly in probability to the uniform measure on the unit disk in the complex plane.
	\end{cor}
	\begin{proof}
		With $p = \tau=0$, and $\E[\xi_1\xi_2]=\rho_1e^{i\theta}$, Theorem \ref{thm:maintheorem} with $F_n=0$ gives that the ESMs $\mu_{(np_n)^{-1/2}X_n}$ converge weakly in probability to the uniform measure supported on
		\[
		e^{i\theta/2}E_0 = E_0 = \{z\in\C:\abs{z}\leq 1\},
		\]
		the unit disk in the complex plane.
	\end{proof}

	To recover a generalization to Theorem \ref{thm:ellipticlaw} with rotated spectrum and a perturbation, we merely take $p_n=1$ for all $n\geq1$.
	\begin{cor}\label{cor:denseellipticrecovered}
		Suppose $(X_n)_{n\geq1}$ is a dense elliptic random matrix model coming from Definition \ref{def:ellipticmodel} given by the complex-valued atom variable $(\xi_1,\xi_2)$ satisfying $\abs{\E[\xi_1\overline{\xi_2}]}<1$ and $\E[\xi_1\xi_2]=\rho_1e^{i\theta}$ with $\rho_1\in[0,1)$ and $\theta\in[0,2\pi)$. If $(F_n)_{n\geq 1}$ is a sequence of deterministic matrices over $\C$ satisfying $\rank(F_n)=o(n)$ and \eqref{eq:boundedhilbertschmidt} after setting $p_n=1$, then the ESMs $\mu_{n^{-1/2}(X_n+F_n)}$ converge weakly in probability to the uniform measure on the ellipsoid $e^{i\theta/2}E_{\rho_1}$, where $E_{\rho_1}$ is as given in \eqref{eq:ellipse}.
	\end{cor}
	\begin{proof}
		With $p_n=1$ for all $n\geq 1$, we also have $p=1$. From Remark \ref{rem:pnequalsone}, $\tau = 0$. Thus, by Theorem \ref{thm:maintheorem}, we have the weak in probability convergence of the ESMs $\mu_{(np_n)^{-1/2}(X_n+F_n)}$ to the uniform measure supported on $e^{i\theta/2}E_{\rho_1}$.
	\end{proof}
	
	Note here that we are not imposing the specialized covariance structure on the real vector $$(\Real(\xi_1),\Imag(\xi_1),\Real(\xi_2),\Imag(\xi_2))^T$$ as is done in the main results of \cite{ellipticlaw}. We do not know of any results as general as this Corollary \ref{cor:denseellipticrecovered}.
	
	In another interesting corollary, we can also examine a more general case when the Bernoulli variables form a symmetric matrix.
	
	\begin{cor}\label{cor:sparsesymmetricbernoullis}
		Suppose $(X_n)_{n\geq1}$ is a sparse elliptic random matrix model coming from Definition \ref{def:sparseellipticmodel} given by the complex-valued atom variable $(\xi_1,\xi_2)$ satisfying $\abs{\E[\xi_1\overline{\xi_2}]}<1$ and 
		$\E[\xi_1\xi_2]=\rho_1e^{i\theta}$ with $\rho_1\in[0,1)$ and $\theta\in[0,2\pi)$ and sparsity parameters $p_n<1$ satisfying Assumption \ref{ass:sparsitylevel}. In this case, if the correlations $\rho_{2,n} = \corr(\eta_1^{(n)},\eta_2^{(n)})$ satisfy $\rho_{2,n}=1$ for all $n\geq1$, then the ESMs $\mu_{(np_n)^{-1/2}X_n}$ converge weakly in probability to the uniform measure on the ellipsoid $e^{i\theta/2}E_{\rho_1}$, where $E_{\rho_1}$ is as given in \eqref{eq:ellipse}.
	\end{cor}
	\begin{proof}
		With $p_n<1$ for all $n\geq 1$, we have that $\tau_n = \rho_{2,n}(1-p_n)$. Thus, if $\rho_{2,n}=1$ for all $n\geq 1$, then $\tau_n\to1-p=\tau$ so that $\tau+p=1$. Thus, by Theorem \ref{thm:maintheorem}, we have the weak in probability convergence of the ESMs $\mu_{(np_n)^{-1/2}X_n}$ to the uniform measure supported on $e^{i\theta/2}E_{\rho_1}$.
	\end{proof}
	
	
	\begin{figure}
		\begin{subfigure}{.49\textwidth}
			\centering
			\includegraphics[width=.8\linewidth]{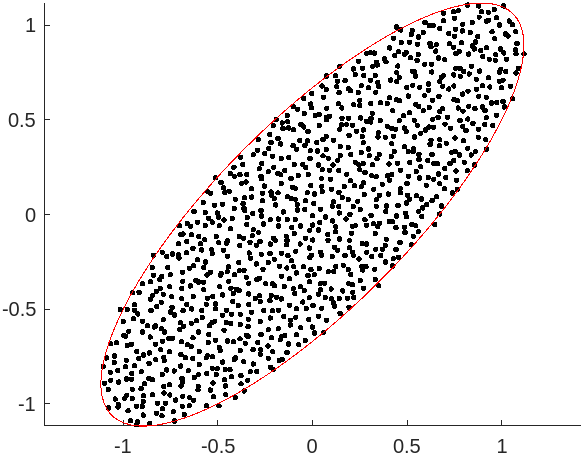}
			\caption{$p_n=\frac{n}{n+1}$}
			\label{fig:pfig1}
		\end{subfigure}
		\begin{subfigure}{.49\textwidth}
			\centering
			\includegraphics[width=.8\linewidth]{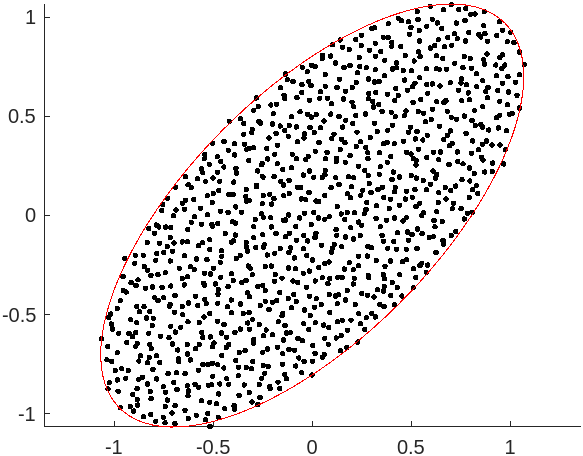}
			\caption{$p_n=\frac{n}{2n+1}$}
			\label{fig:pfig2}
		\end{subfigure}
		\caption{Plots of eigenvalues of $1000\times 1000$ complex Gaussian sparse elliptic random matrices $(np_n)^{-1/2}X_n$ (so $n=1000$). Throughout, $\rho_1 = \rho_{2,n} = 0.5$ and $\theta = \frac{\pi}{2}$. The two red ellipses above are \ref{fig:pfig1} $e^{i\pi/4}E_{0.5}$ and \ref{fig:pfig2} $e^{i\pi/4}E_{0.375}$.}
		\label{fig:pfigs}
	\end{figure}

	We conclude this subsection with a conjecture for future work, inspired by the work done in \cite{supersparse}. Inherent in the hypothesis of Theorem \ref{thm:maintheorem} is the assumption that the sparsity parameters $(p_n)_{n\geq 1}$ satisfy $np_n\to \infty$. The threshold for these results is when $p_n = \frac{c}{n}$ (for some constant $c>0$) so that $np_n=c$ no longer tends to infinity with $n$. In this regime of sparsity, we conjecture that, as in the i.i.d.\ regime, for every $c>0$, there still exists some deterministic measure $\mu_c$---depending at least on $c$---that is the weak limit of the ESMs of a sparse elliptic random matrix model with $p_n = \tfrac{c}{n}$. Figures \ref{fig:supersparse2}, \ref{fig:supersparse5}, and \ref{fig:supersparse10} show the behavior of eigenvalues of such ``supersparse" elliptic random matrices, with varying values of $c$.
	
	We also conjecture that our main result, Theorem \ref{thm:maintheorem}, holds without Assumption \ref{ass:sparsitylevel}. For results with minimal assumptions on the sparsity parameter $p_n$, we refer the reader to \cite{rudelsontik} for the treatment of a real i.i.d.\ model and to \cite{revisited} for the treatment of a complex i.i.d.\ model. The tool found in these works do not seem to translate well to our elliptic model.
	
	For tools to treat the dependence in our elliptic model, we refer the reader to \cite{andrewelliptic, ellipticlaw,naumov} for results and techniques to use with dense elliptic matrices. Furthermore, we refer the reader to \cite{gotzecircularlaw} for the arguments modified in Section \ref{sec:lsvappendix} that worked in tandem with those from the previous citations to treat the sparsity and dependence of our model.

	\begin{figure}
		\begin{subfigure}{.49\textwidth}
			\centering
			\includegraphics[width=.8\linewidth]{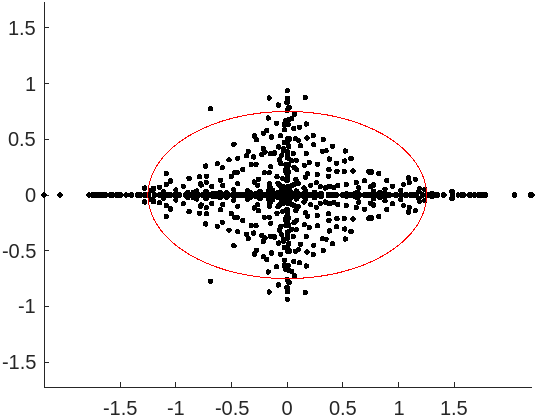}
			\caption{$p_n=\frac{2}{n}$}
			\label{fig:supersparse2}
		\end{subfigure}
		\begin{subfigure}{.49\textwidth}
			\centering
			\includegraphics[width=.8\linewidth]{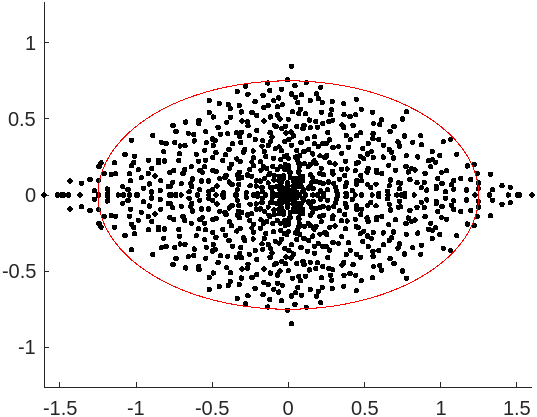}
			\caption{$p_n=\frac{5}{n}$}
			\label{fig:supersparse5}
		\end{subfigure}
		\begin{subfigure}{.49\textwidth}
			\centering
			\includegraphics[width=.8\linewidth]{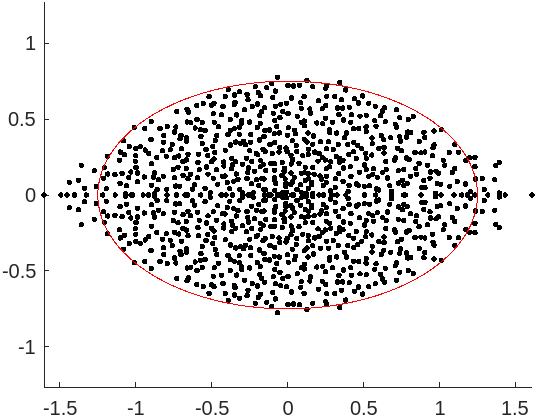}
			\caption{$p_n=\frac{10}{n}$}
			\label{fig:supersparse10}
		\end{subfigure}
		\caption{Plots of eigenvalues of $1000\times 1000$ (so $n = 1000$) real Gaussian ``supersparse" elliptic matrices $X_n$ with no normalization. Throughout, $\rho_1=\rho_{2,n}=0.5$ and $\theta=0$. The red ellipses plotted are each $E_{0.25}$}
		\label{fig:supersparse}
	\end{figure}
	
	\section{Preliminaries and outline of proof}\label{sec:defs}
	
	Before we give some preliminaries for the context of this paper and the outline of proof, we begin this section by introducing some notational conventions and defining relevant concepts.
	\subsection{Definitions and notation}\label{sec:definitions}
	To discuss the spaces of square matrices, we will be using the notation $\cal{M}_n(\cdot)$. For example, $\cal{M}_n(\C)$ will denote the space of $n\times n$ matrices over $\C$ and $\cal{M}_n(\cal{M}_2(\C))$ will denote the space of $n\times n$ matrices whose entries are themselves $2\times 2$ matrices over $\C$. We will denote the space of rectangular $m\times n$ matrices over $\C$ by $\cal{M}_{m,n}(\C)$
	
	Let $A_n\in \cal{M}_n(\C)$. As above, we denote the $n$ unordered eigenvalues of $A_n$ counted with algebraic multiplicity by $\lambda_1(A_n),\ldots,\lambda_n(A_n)$. Further, the $n$ ordered singular values of $A_n$ will be given by $0\leq \sigma_n(A_n)\leq \cdots\leq \sigma_1(A_n)$. The empirical spectral measure of $A_n$ and the analogous empirical singular value measure of $A_n$ are given by
	\[
	\mu_{A_n} = \frac{1}{n}\sum_{i=1}^n \delta_{\lambda_i(A_n)}
	\quad\text{and}\quad
	\nu_{A_n} = \frac{1}{n}\sum_{i=1}^n \delta_{\sigma_i(A_n)},
	\]
	respectively.
	Here, $\delta_z$ is the Dirac point mass centered at $z\in\C$. We denote the Hilbert--Schmidt norm of $A_n\in\cal{M}_n(\C)$ by $\norm{A_n}_{HS}$ and define it by
	\[
	\norm{A_n}_{HS}^2 = \tr(A_n^*A_n) = \sum_{1\leq i,j\leq n}\abs{(A_n)_{ij}}^2 = \sum_{i=1}^n\sigma_i(A_n)^2.
	\]
	The operator norm of $A_n$ will be denoted by
	\[
	\norm{A_n} = \max_{\norm{x}=1}\norm{A_nx}= \sigma_1(A_n) = \sqrt{\lambda_{\text{max}}(A_n^*A_n)},
	\]
	where $\lambda_{\max}(A_n^*A_n)$ is the largest eigenvalue of $A_n^*A_n$. Since this matrix $A_n^*A_n$ is positive semi-definite, all of its eigenvalues are real and nonnegative.
	
	Any reference to almost all $z\in \C$ is meant with respect to Lebesgue measure on $\C$.
	
	As for asymptotic notation, if $X$ and $Y$ are quantities depending on $n$, we use $X = O(Y)$, $X\ll Y$, or $Y\gg X$ to mean that $\abs{X}\leq CY$ for all sufficiently large $n$ and a constant $C$ that does not depend on $n$. Additionally, by $X = o(Y)$, we mean that $X/Y\to 0 $ as $n\to\infty$. If an asymptotic quantity depends on some ambient constants $\alpha_1,\ldots,\alpha_k$, this will be denoted as $O_{\alpha_1,\ldots,\alpha_k}(\cdot)$, $o_{\alpha_1,\ldots,\alpha_k}(\cdot)$, or $\ll_{\alpha_1,\ldots,\alpha_k}$.
	
	We will use $I_n$ to denote the $n\times n$ identity matrix. 
	
	For an event $E$, we will denote the indicator of $E$ either by $\mathbbm{1}_E$ or $\mathbbm{1}\{E\}$.
	
	For two complex-valued random variables $X$ and $Y$, each with finite second moment, we denote the covariance of $X$ and $Y$ by
	\[
	\cov(X,Y) := \E[(X-\E X)\overline{(Y-\E Y)}]=\E[X\overline{Y}]-(\E X)(\E \overline{Y})
	\]
	and when $\var(X),\var(Y)>0$, the correlation of $X$ and $Y$ by
	\[
	\corr(X,Y) := \frac{\cov(X,Y)}{\sqrt{\var (X)}\sqrt{\var (Y)}}.
	\]
	
	A measurable function $f$ is said to be uniformly integrable in probability for a sequence $(\mu_n)_{n\geq 1}$ of random probability measures if for all $\eps>0$, 
	\[
	\lim_{T\to\infty}\limsup_{n\to\infty}\P\parens{\int_{\abs{f}>T}\abs{f(s)}\,d\mu_n(s)>\eps}=0.
	\]
	
	We recall that for two probability measures $\mu,\nu$ on $\R$ with its Borel $\sigma$-algebra, the Lévy distance between $\mu$ and $\nu$ is defined as
	\[
	L(\mu,\nu):= \inf\set{\eps>0:\mu(A)\leq \nu(A^{\eps})+\eps\text{ and }\nu(A)\leq \mu(A^{\eps})+\eps\text{ for all Borel sets }A\subseteq\R},
	\]
	where $A^{\eps}$ is the $\eps$-neighborhood of $A$:
	\[
	A^{\eps} := \bigcup_{x\in A}\set{y\in\R:\abs{y-x}<\eps}.
	\]
	
	A sequence of random measures $(\mu_n)_{n\geq1}$ on $\R$ (resp.\ $\C$) with its Borel $\sigma$-algebra is said to converge weakly in probability to the deterministic probability measure $\mu$ as $n\to\infty$ if for every bounded and continuous real-valued function $f$ on $\R$ (resp.\ $\C$) and every $\eps>0$:
	\[
	\lim_{n\to\infty}\P\parens{\abs{\int f\,d\mu_n-\int f\,d\mu}>\eps}=0.
	\]
	
	We recall here that for two sequences $(\mu_n)_{n\geq 1}$ and $(\nu_n)_{n\geq 1}$ of random probability measures on $\R$, $L(\mu_n,\nu_n)\to0$ in probability implies weak in probability convergence of $\mu_n-\nu_n\to 0$ as $n\to\infty$, given that one of $(\mu_n)_{n\geq1}$ or $(\nu_n)_{n\geq1}$ is tight in probability. \cite{billingsley}.
	
	For a finite set $S$, we will denote its cardinality by $\abs{S}$.
	
	The rest of the paper is devoted to the proof of Theorem \ref{thm:maintheorem}.
	
	With $M_n = (np_n)^{-1/2}(X_n+F_n)$, we want to show that the empirical spectral measure $\mu_{M_n}$ converges weakly in probability to the measure $\mu_{\rho_1,\theta,\tau,p}$ as $n\to\infty$. The main tool we will be utilizing to prove this weak convergence of measures is the logarithmic potential.
	
	We define $\cal{P}(\C)$ to be the set of probability measures that integrate $\log(\cdot)$ in a neighborhood of infinity. For such a probability measure $\mu\in\cal{P}(\C)$, we define the \textit{logarithmic potential of $\mu$}, $U_\mu:\C\to (-\infty,\infty]$, by
	\[
	U_\mu(z) = -\int_{\C}\log\abs{w-z}\,d\mu(w).
	\] 
	\subsection{Preliminary techniques}
	Our main aim of the paper is to apply a classical result that connects the behavior of the logarithmic potentials  to the behavior of the underlying ESMs and empirical singular value measures. Due to a Hermitization argument that dates back to at least Girko \cite{girko1}, we may write
	\[
	U_{\mu_{M_n}}(z) = -\frac{1}{n}\sum_{i=1}^n\log(\sigma_i(M_n-zI_n)).
	\]
	This powerful equation allows us to instead focus our attention on the singular values of $M_n-zI_n$, which are themselves square roots of eigenvalues of the Hermitian matrix $(M_n-zI_n)^*(M_n-zI_n)$. Although we are able to work with a Hermitian matrix, we have an additional parameter $z\in\C$ to keep track of.

	The outline of proof for this paper is motivated by the following classical result as presented in \cite[Lemma 4.3]{aroundthecirc} used to show weak in probability convergence of $\mu_{M_n}$:
	
	\begin{lem}[Hermitization]\label{lem:hermitization}
		Let $(A_n)_{n\geq1}$ be a sequence of complex random matrices where $A_n$ is $n\times n$. Suppose that there exists a family of (non-random) probability measures $(\nu_z)_{z\in\C}$ on $[0,\infty)$ such that, for almost all $z\in\C$,
		\begin{enumerate}[label=(\roman*)]
			\item\label{item:proveweakinprob} $\nu_{A_n-zI_n}$ converges weakly in probability to $\nu_z$ as $n\to\infty$ and
			\item\label{item:proveuniformintegrability} $\log(\cdot)$ is uniformly integrable in probability for $(\nu_{A_n-zI_n})_{n\geq 1}$.
		\end{enumerate}
		Then there exists a probability measure $\mu$ that integrates $\log(\cdot)$ in a neighborhood of infinity such that
		\begin{enumerate}
			\item[(j)]\label{item:resultweakinprob} $\mu_{A_n}$ converges weakly in probability to $\mu$ as $n\to\infty$ and
			\item[(jj)]\label{item:resultdeterminestheellipticlaw} for almost all $z\in\C$, $$U_\mu(z) = -\int_0^\infty\log (s)\,d\nu_z(s).$$
		\end{enumerate}
	\end{lem}
	
	At the level of eigenvalues, we will be interested in a rotated version of our sparse elliptic random matrix. For much of the proof, $M_n$ will be playing the role of $A_n$ above. However, for the proof of Theorem \ref{thm:maintheorem}, we will first assume we have proven the result for the rotated matrix $e^{-i\theta/2}M_n$. This rotation is to force the mirrored, off-diagonal entries of $X_n$ to satisfy:
	\begin{equation}\label{eq:realmixedmoment}
		\E[(e^{-i\theta/2}\delta_{ij}^{(n)}x_{ij})(e^{-i\theta/2}\delta_{ji}^{(n)}x_{ji})]= e^{-i\theta}\E[\delta_{ij}^{(n)}\delta_{ji}^{(n)}]\E[x_{ij}x_{ji}] = \rho_1p_n(\tau_n+p_n)\in\R.
	\end{equation}
	Compare this with the computation done in \eqref{eq:mixedmoment}, modulo dividing by $p_n$.
	
	At first, we aim to show that in the notation of Lemma \ref{lem:hermitization}, $\mu = \mu_{\rho_1,0,\tau,p}$
	To do this, we will be using \textit{(jj)} from Lemma \ref{item:resultdeterminestheellipticlaw} after showing that the family $(\nu_z)_{z\in\C}$ \textbf{determines the elliptic law with parameters $\rho_1$, $\tau$, $\theta=0$, and $p$}:
	\begin{equation}\label{eq:determineellipticlaw}
		U_{\mu_{\rho_1,0,\tau,p}}(z) = -\int_0^\infty\log(s)\,d\nu_z(s)\quad\text{for almost all $z\in\C$}.
	\end{equation}
	The equality of measures $\mu=\mu_{\rho_1,0,\tau,p}\in\cal{P}(\C)$ then follows from the unicity of the logarithmic potential, see for example \cite[Lemma 4.1]{aroundthecirc}.
	
	In this vein, the subsequent sections of this paper will be devoted to proving an auxiliary theorem, which is Theorem \ref{thm:maintheorem} with $\theta = 0$, so that $\E[\xi_1\xi_2]$ is real.
	
	\begin{thm}\label{thm:auxiliarymaintheorem}
		Assuming all of the hypotheses and notation coming from Theorem \ref{thm:maintheorem} in the case that $\theta=0$, we have that
		$
		\mu_{M_n} \to \mu_{\rho_1,0,\tau,p}
		$
		weakly in probability as $n\to\infty$.
	\end{thm}
	
	\subsection{Outline of proof}
	
	We will provide a proof of Theorem \ref{thm:maintheorem}---assuming Theorem \ref{thm:auxiliarymaintheorem} holds---at the end of this section. Therefore, for the rest of the paper, we will only be concerned with establishing Theorem \ref{thm:auxiliarymaintheorem}.
	
	To prove Theorem \ref{thm:auxiliarymaintheorem}, the paper will go through the following steps, following the classical proof steps as in the proof of the circular law using Girko's Hermitization technique. Section \ref{sec:convofESVMs} will prove \textit{\ref{item:proveweakinprob}} in Lemma \ref{lem:hermitization} for a family of deterministic measures $(\nu_z)_{z\in\C}$ on $[0,\infty)$ that is known to satisfy \eqref{eq:determineellipticlaw}. We prove integrability of the logarithm for $(\nu_{M_n-zI_n})_{n\geq 1}$ in probability (this is \textit{\ref{item:proveuniformintegrability}} in Lemma \ref{lem:hermitization}) in Section \ref{sec:uniformintegrability} and control the least singular value in Section \ref{sec:lsvappendix}.
	
	The proof of our least singular value bound, Proposition \ref{prop:leastsingularvalue}, is inspired by both \cite{gotzecircularlaw} and \cite{andrewelliptic}, bringing together the arguments from each paper to deal with both the sparsity and dependence in our random matrix model. Following the \textit{bootstrapping method} as seen in \cite[Section 4]{gotzecircularlaw},  we accumulate unit vectors $v$ for which we may lower bound
	$
	\norm{(M_n-zI_n)v}
	$
	and thus control our least singular value
	\[
	\sigma_n(M_n-zI_n) = \min _{\norm{v}=1}\norm{(M_n-zI_n)v}.
	\]
	
	This work on sparse elliptic random matrices is interesting and important because it inherits the most difficult aspects of controlling the least singular value for both sparse and elliptic models. While the weak convergence of the empirical singular value measures follows similarly as in the dense i.i.d.\ case with the circular law or even as in the dense elliptic case with the elliptic law, controlling the least singular value does not follows from similar techniques from any one paper.
	
	For instance, the sparsity inherent in the model leads to failures in applying \cite[Theorem 5.1]{andrewelliptic} or altering the truncation argument in its proof. It turns out that controlling $\norm{(M_n-zI_n)v}$ for compressible vectors $v$ led to most of the difficulty for the sparse elliptic model. On the other hand, in adapting the more novel proof method of \cite{revisited} for the sparse i.i.d.\ case, we ran into difficulties translating their ``quasirandomness events" to our elliptic model. Adapting \cite{rudelsontik} to the complex elliptic case or sparsifying the work done in \cite{naumov} could be other courses of action to dealing with our sparse elliptic model, but ones which the authors did not explore. For more details on the difficulties of sparsity and elliptic dependence coming together, see Section \ref{sec:lsvappendix}.
	
	We now give a proof of Theorem \ref{thm:maintheorem} assuming Theorem \ref{thm:auxiliarymaintheorem}:
	\begin{proof}[Proof of Theorem \ref{thm:maintheorem}]
		First, we consider the rotated matrix
		$\widehat{M}_n = e^{-i\theta/2}M_n$.
		This matrix can be viewed as a sparse elliptic random matrix coming from Definition \ref{def:sparseellipticmodel} with atom variable $(e^{-i\theta/2}\xi_1,e^{-i\theta/2}\xi_2)$. As in \eqref{eq:realmixedmoment}, the mixed moment of our atom variable is real-valued in this case. Therefore, $\theta = 0$ and we have by Theorem \ref{thm:auxiliarymaintheorem} that $\mu_{\widehat{M}_n}\to \mu_{\rho_1,0,\tau,p}$ weakly in probability as $n\to\infty$.
		
		We now rotate back to get a result about our original, non-rotated matrix.
		Since the eigenvalues of $\widehat{M}_n$ are the eigenvalues of $M_n$ rotated by an angle of $-\theta/2$ radians in the complex plane, we have that $\mu_{M_n}$ is nothing more than the push-forward of $\mu_{\widehat{M}_n}$ by $f:\C\to\C$, where $f(z) = e^{i\theta/2}z$ is rotation by $\theta/2$ radians in the complex plane:
		$
		\mu_{M_n} = f_*(\mu_{\widehat{M}_n})
		$.
		
		Since $f$ is continuous and taking the push-forward by a continuous function preserves weak convergence of measures (see \cite[Theorem 2.7]{billingsley}, say), we have that 
		\begin{equation}\label{eq:finalstrike} 
			\mu_{M_n} = f_*(\mu_{\widehat{M}_n}) \to f_*(\mu_{\rho_1,0,\tau,p}).
		\end{equation}
		weakly in probability as $n\to\infty$ as well. Finally, we know that $\mu_{\rho_1,0,\tau,p}$ has a density given by $$\frac{1}{\pi(1-(\rho_1(\tau+p))^2)}\mathbbm{1}_{E_{\rho_1(\tau+p)}}(z).$$ Therefore, $f_*(\mu_{\rho_1,0,\tau,p})$ will have a density as well; in fact, given that our rotation map $f$ is bijective, this new density is given by
		\begin{align*}
			C_{\rho_1,\tau,p}^{-1}\mathbbm{1}_{E_{\rho_1(\tau+p)}}(f^{-1}(z))
			&=
			C_{\rho_1,\tau,p}^{-1}\mathbbm{1}_{E_{\rho_1(\tau+p)}}(e^{-i\theta/2}z)\\&=
			C_{\rho_1,\tau,p}^{-1}\mathbbm{1}_{e^{i\theta/2}E_{\rho_1(\tau+p)}}(z),
		\end{align*}
		where $C_{\rho_1,\tau,p}=\pi(1-(\rho_1(\tau+p))^2)$
		It follows that $f_*(\mu_{\rho_1,0,\tau,p})=\mu_{\rho_1,\theta,\tau,p}$, the uniform measure on the ellipsoid $e^{i\theta/2}E_{\rho_1(\tau+p)}$ in the complex plane---Theorem \ref{thm:maintheorem} now follows from \eqref{eq:finalstrike}.
	\end{proof}
	
	We recall that all subsequent sections of the paper will be dedicated to proving Theorem \ref{thm:auxiliarymaintheorem}.
	
	\section{Convergence of the empirical singular value measures}\label{sec:convofESVMs}
	
	This section is dedicated to showing the weak, in probability convergence of the empirical singular value measures of the following matrix:
	$$M_n-zI_n = (np_n)^{-1/2}(X_n+F_n) - zI_n.$$ 
	More specifically, for almost every $z\in\C$, we show the in-probability weak convergence of $\nu_{M_n-zI_n}$ to some deterministic measure $\nu_z$. We also want this family $(\nu_z)_{z\in\C}$ to satisfy \eqref{eq:determineellipticlaw}. The main result of this section is the following.
	\begin{prop}\label{prop:weakconvergence}
		Let $(X_n)_{n\geq 1}$ be a sparse elliptic random matrix model coming from Definition \ref{def:sparseellipticmodel} and satisfying the hypotheses of Theorem \ref{thm:auxiliarymaintheorem}. Furthermore, let $(F_n)_{n\geq 1}$ be a sequence of deterministic matrices over $\C$ satisfying $\rank(F_n) = o(n)$ and \eqref{eq:boundedhilbertschmidt}. Then for almost all $z\in\C$, $\nu_{M_n-zI_n}\to\nu_z$  weakly in probability as $n\to\infty$. Furthermore, $(\nu_z)_{z\in\C}$ satisfies \eqref{eq:determineellipticlaw}.
	\end{prop}
	
	\begin{remark}
		We recall that this proposition and the results used to prove it do not require the stronger sparsity condition seen in Assumption \ref{ass:sparsitylevel}. We only need the fact that $np_n\to\infty$.
	\end{remark}
	
	As will be seen at the beginning of the proof of Proposition \ref{prop:weakconvergence}, we will not need to worry about the deterministic perturbation $F_n$ given our assumptions on its rank. Thus, for the sake of notation, we denote our normalized, rotated, and shifted sparse random elliptic matrix coming from Definition \ref{def:sparseellipticmodel} by
	\begin{equation}\label{eq:rotatedmatrix}
		Y_{n,z} := (np_n)^{-1/2}X_n-zI_n
	\end{equation}
	for the remainder of this section. Note that many of the auxiliary proofs of this section are concerned with $Y_{n,z}$ and not the full matrix $M_n-zI_n$ appearing in the statement of Proposition \ref{prop:weakconvergence}.

	Since the eigenvalues of $Y_{n,z}^*Y_{n,z}$ are $\sigma_1(Y_{n,z})^2\geq\cdots\geq\sigma_n(Y_{n,z})^2$, we have that
	\[
	\mu_{Y_{n,z}^*Y_{n,z}}((-\infty,x]) = \nu_{Y_{n,z}}((-\infty,\sqrt{x}])\quad\text{for all }x\geq0,
	\]
	which will allow us to focus on the Hermitian matrix $Y_{n,z}^*Y_{n,z}$ rather than $Y_{n,z}$. Before proving Proposition \ref{prop:weakconvergence}, we will prove the following:
	\begin{enumerate}
		\item The Stieltjes transform of $\nu_{Y_{n,z}}$ concentrates around its mean, the Stieltjes transform of $\E\nu_{Y_{n,z}}$, with high probability. This will allow us to focus on the expected value of the Stieltjes transform instead of the random Stieltjes transform itself.
		\item The measures $\nu_{Y_{n,z}}$ are close to $\nu_{\widetilde{Y}_{n,z}}$ in the Lévy distance with high probability, where $\widetilde{Y}_{n,z}$ is a truncated, centralized, and rescaled version of $Y_{n,z}$.
		\item The measures $\nu_{Y_{n,z}}$ are close to $\nu_{W_{n,z}}$ with high probability, where $W_{n,z}$ is a random matrix model that we know satisfies $\nu_{W_{n,z}}\to\nu_z$ weakly in probability and $(\nu_z)_{z\in\C}$ satisfies \eqref{eq:determineellipticlaw} for almost all $z\in\C$.
	\end{enumerate}
	
	\subsection{Step 1: Concentration of the Stieltjes transform}
	
	Step 1 will be an application of Lemma 7.7 in \cite{ellipticlaw}:
	\begin{lem}\label{lem:concentrationofESM}
		Let $({X_n})_{n\geq1}$ be a sparse elliptic random matrix model coming from Definition \ref{def:sparseellipticmodel} and satisfying the hypotheses of Theorem \ref{thm:auxiliarymaintheorem}. We denote the resolvent of $Y_{n,z}^*Y_{n,z}$ by
		\[
		R_{n,z}(\alpha) := (Y_{n,z}^*Y_{n,z}-\alpha I_n)^{-1}
		\]
		for $\alpha\in\C$ with $\Imag(\alpha)\neq0$. Then
		\[
		\E\abs{\frac{1}{n}\tr R_{n,z}(\alpha)-\E\left[\frac{1}{n}\tr R_{n,z}(\alpha)\right]}^4 \leq C\frac{c_\alpha^4}{n^2}
		\]
		uniformly in $z$. Here $C$ is an absolute constant and
		\[
		c_\alpha = \frac{1}{\abs{\Imag(\alpha)}}+\frac{\abs{\alpha}}{\abs{\Imag(\alpha)}^2}.
		\]
		Furthermore, for any fixed $\alpha\in \C$ with $\Imag(\alpha)\neq0$ and any fixed $z\in\C$,
		\[
		\frac{1}{n}\tr R_{n,z}(\alpha) = \E\frac{1}{n}\tr R_{n,z}(\alpha) + O_\alpha(n^{-1/8}) \quad \text{almost surely}.
		\]
	\end{lem}
	
	\begin{proof}
		Our random matrix $Y_{n,z}$ will be playing the role of the $R_n$ as in the proof of Lemma 7.7 in \cite{ellipticlaw}. Furthermore, our resolvent $R_{n,z}(\alpha)$ is the resolvent $H_n(\alpha)$ appearing in Lemma 7.7 in \cite{ellipticlaw}. Otherwise, even with sparsity, the proof follows exactly as it does in \cite{ellipticlaw}.
	\end{proof}

	\subsection{Step 2: The truncation step}
	
	Before truncating, we first remove the diagonal elements of our sparse elliptic random matrix model $X_n$, in a proof inspired by the discussion following the statement of Theorem 2.5 in \cite{BaiSilverstein}.
	
	\begin{lem}\label{lem:zerodiagonal}
		Let $({X_n})_{n\geq1}$ be a sparse elliptic random matrix model coming from Definition \ref{def:sparseellipticmodel} and satisfying the hypotheses of Theorem \ref{thm:auxiliarymaintheorem}. Let $X_n^0$ be the matrix formed from $X_n$ after replacing the diagonal entries by $0$ and furthermore let $Y_{n,z}^0$  be  defined similarly as in \eqref{eq:rotatedmatrix}, with $X_n$ replaced by $X_n^0$. For any $z\in\C$, we have
		$
		\nu_{Y_{n,z}}-\nu_{Y_{n,z}^0}\to0
		$
		weakly in probability as $n\to\infty$\footnote{In other words, for every bounded continuous function $f:\R\to\R$, $\int f \,d\nu_{Y_{n,z}}-\int f\,d\nu_{Y_{n,z}^0}$ converges to 0 in probability.}.
	\end{lem}
	\begin{proof}
		Let $\nu_n$ denote the empirical singular value measure of $Y_{n,z}^*Y_{n,z}$ for some fixed $z\in\C$. For this proof, we define $X_n^0$ to be the matrix formed from $X_n$ by replacing the diagonal entries with 0 and $X_n^1$ to be the matrix formed from $X_n$ after multiplying each diagonal entry $\delta_{ii}^{(n)}x_{ii}$ by the indicator $\mathbbm{1}\{\abs{x_{ii}}\leq n^{1/4}\}$.
		We analogously define $Y_{n,z}^0$ and $Y_{n,z}^1$ and the corresponding empirical singular value measures $\nu_n^0$ and $\nu_n^1$ for $(Y_{n,z}^0)^*Y_{n,z}^0$ and $(Y_{n,z}^1)^*Y_{n,z}^1$.
		
		The goal of this proof is to show that $\nu_n-\nu_n^0\to0$ weakly in probability as $n\to\infty$. First, we show the weak in probability convergence of $\nu_n-\nu_n^1\to0$. For this, we use Theorem A.44 in \cite{BaiSilverstein}:
		\[
		\sup_{x\in\R}\abs{\nu_n((-\infty,x])-\nu_n^1((-\infty,x])}\leq \frac{1}{n}\rank(X_n-X_n^1).
		\]
		The matrix $X_n-X_n^1$ is a diagonal matrix so that
		\[
		\frac{1}{n}\rank(X_n-X_n^1) = \frac{1}{n}\sum_{i=1}^n\mathbbm{1}\{\delta_{ii}^{(n)}x_{ii}\mathbbm{1}\{\abs{x_{ii}}>n^{1/4}\}\neq0\}.
		\]
		We show that this random variable on the right converges to 0 in probability by Markov's inequality. Indeed, taking expectation gives
		\[
		\E\frac{1}{n}\rank(X_n-X_n^1) = \P(\delta_{11}^{(n)}x_{11}\mathbbm{1}\{\abs{x_{11}}>n^{1/4}\}\neq0)\leq
		p_n\P(\abs{x_{11}}>n^{1/4})=o(1).
		\]
		Therefore, this in probability uniform convergence of the cumulative distribution functions $\nu_n((-\infty,x])-\nu_n^1((-\infty,x])$ to 0 implies the in probability weak convergence of $\nu_n-\nu_n^1$ to 0.
		
		Now, to compare $\nu_n^0$ and $\nu_n^1$, we have by Corollary A.42 in \cite{BaiSilverstein}:
		\[ L^4(\nu_{n}^0,\nu_{n}^1)\leq\frac{2}{n^3p_n}\left[\tr((Y^0_{n,z})^*Y_{n,z}^0+(Y^1_{n,z})^*Y_{n,z}^1)\tr\left((X_n^1-X_n^0)^*(X_n^1-X_n^0)\right)\right].
		\]
		
		Recall the definition of the Lévy distance $L$ given in Section \ref{sec:definitions}. We will prove that this quantity on the right converges to 0 in probability by showing that $\frac{2}{n^2p_n}\tr\left((X_n^1-X_n^0)^*(X_n^1-X_n^0)\right)$ converges to 0 in probability and that $\frac{1}{n}\tr((Y^0_{n,z})^*Y_{n,z}^0+(Y^1_{n,z})^*Y_{n,z}^1)$ is bounded in probability. These claims are routine applications of Markov's inequality. For this first, we compute for the diagonal matrix $X_n^1-X_n^0$:
		\begin{align*}
			\E\frac{2}{n^2p_n}\tr\left((X_n^1-X_n^0)^*(X_n^1-X_n^0)\right) &= \frac{2}{n^2p_n}\sum_{i=1}^n\E[\delta_{ii}^{(n)}\abs{x_{ii}}^2\mathbbm{1}\{\abs{x_{ii}}\leq n^{1/4}\}]\\&=
			\frac{2}{n}\E[\abs{x_{11}}^2\mathbbm{1}\{\abs{x_{11}}\leq n^{1/4}\}]\leq \frac{2}{\sqrt{n}}.
		\end{align*}
		The first claim then follows by Markov's inequality. For the second claim, we compute
		\begin{align*}
			\E\tr((Y^0_{n,z})^*Y_{n,z}^0) &= \E\norm{(np_n)^{-1/2}X_n^0-zI_n}^2_{HS}\\&=
			\frac{1}{np_n}\sum_{i\neq j}\E[\delta_{ij}^{(n)}\abs{x_{ij}}^2]+n\abs{z}^2\\&=
			\frac{1}{n}\sum_{i\neq j}1+n\abs{z}^2\leq n(1+\abs{z}^2).
		\end{align*}
		Dividing by $n$ gives the desired control for applying Markov's to show the boundedness in probability. For $Y_{n,z}^1$, we analogously compute
		\begin{align*}
			\E\tr((Y^1_{n,z})^*Y_{n,z}^1) 
			\leq\frac{2}{n}\parens{n^2+n^{1.5}} + 2n\abs{z}^2.
		\end{align*}
		Dividing be $n$ here again gives the desired control to invoke Markov's inequality. It now follows that that $L(\nu_n^0,\nu_n^1)$ converges to 0 in probability so that $\nu_n^0-\nu_n^1\to0$ weakly in probability. Via the established weak convergences $\nu_n-\nu_n^1\to0$, $\nu_n^0-\nu_n^1\to0$, and the triangle inequality, we have that $\nu_n-\nu_n^0\to0$ weakly in probability.
	\end{proof}
	
	\begin{remark}
		By the preceding lemma, we assume without loss of generality for the remainder of this section that the diagonal entries of $X_n$ are 0.
	\end{remark}

	We continue on with the truncation by associating to our sparse elliptic matrix model $(X_n)_{n\geq1}$ two other sequences of random matrices: for all $1\leq i,j\leq n$, we define $(\widehat{X}_n)_{n\geq1}$ and $(\widetilde{X}_n)_{n\geq1}$ by
	\[
	(\widehat{X}_n)_{ij} = \begin{cases}
		\delta_{ij}^{(n)}\widehat{x}_{ij}& i\neq j\\
		0 &i=j
	\end{cases}
	\]
	where for each $i\neq j$, $\widehat{x}_{ij} := x_{ij}\mathbf{1}_{\set{|x_{ij}|\leq (np_n)^{1/4}}}-\E[x_{ij}\mathbf{1}_{\set{|x_{ij}|\leq (np_n)^{1/4}}}]$
	and
	\[
	(\widetilde{X}_n)_{ij} =
	\begin{cases}
		\delta_{ij}^{(n)}\frac{\widehat{x}_{ij}}{\sqrt{\E|\widehat{x}_{ij}|^2}} &i\neq j\\
		0 & i=j
	\end{cases}.
	\]
	Similarly, we define
	\[
	\widehat{H}_{n,z} = \left((np_n)^{-1/2}\widehat{X}_n-zI_n\right)^*\left((np_n)^{-1/2}\widehat{X}_n-zI_n\right)
	\]
	and
	\[
	\widetilde{H}_{n,z} = \left((np_n)^{-1/2}\widetilde{X}_n-zI_n\right)^*\left((np_n)^{-1/2}\widetilde{X}_n-zI_n\right).
	\]
	We prove the following lemma for our truncated, centralized, and rescaled matrix $\widetilde{X}_n$:
	
	\begin{lem}\label{lem:truncation}
		Let $({X_n})_{n\geq1}$ be a sparse elliptic random matrix model coming from Definition \ref{def:sparseellipticmodel} and satisfying the hypotheses of Theorem \ref{thm:auxiliarymaintheorem}. Then uniformly for any $|z|\leq M$ (for some fixed $M>0$), we have that with $$H_{n,z} = Y_{n,z}^*Y_{n,z}= \parens{(np_n)^{-1/2}X_n-zI_n}^*\parens{(np_n)^{-1/2}X_n-zI_n}$$ and $\widetilde{H}_{n,z}$ as defined above, $L(\nu_{H_{n,z}},\nu_{\widetilde{H}_{n,z}})\to 0$
		in probability, as $n\to\infty$. Recall the definition of the Lévy distance $L$ as given in Section \ref{sec:definitions}.
		
	\end{lem}
	\begin{proof}
		By \cite[Corollary A.42]{BaiSilverstein},
		\begin{equation}\label{eq:levybound}
			L^4(\nu_{H_{n,z}},\nu_{\widehat{H}_{n,z}})\leq\frac{2}{n^3p_n}\left[\tr(H_{n,z}+\widehat{H}_{n,z})\tr\left((X_n-\widehat{X}_n)^*(X_n-\widehat{X}_n)\right)\right].
		\end{equation}
		We will show that this right hand side converges to 0 in probability by showing that $\frac{2}{n^2p_n}\tr\left((X_n-\widehat{X}_n)^*(X_n-\widehat{X}_n)\right)$ converges to 0 in probability and that $\frac{1}{n}\tr(H_{n,z}+\widehat{H}_{n,z})$ is bounded in probability. Both of these claims are routine applications of Markov's inequality. For the first,
		\begin{align*}
			\E \tr\left((X_n-\widehat{X}_n)^*(X_n-\widehat{X}_n)\right) 
			&=
			\sum_{i\neq j}\E[\delta_{ij}^{(n)}|x_{ij}\mathbbm{1}_{\abs{x_{ij}}>(np_n)^{1/4}}|^2]
			\\&\leq
			\frac{n^2}{2}p_n\E[\abs{x_{12}}^2\mathbbm{1}_{\abs{x_{12}}>(np_n)^{1/4}}]+\frac{n^2}{2}p_n\E[\abs{x_{21}}^2\mathbbm{1}_{\abs{x_{21}}>(np_n)^{1/4}}]
		\end{align*}
		in the first equality, we are using that we already replaced the diagonal entries of $X_n$ with zeros. Therefore, by Markov's inequality and the Dominated Convergence Theorem, $\frac{2}{n^2p_n}\tr\left((X_n-\widehat{X}_n)^*(X_n-\widehat{X}_n)\right)$ converges to 0 in probability as $n\to\infty$. For the second claim, we have already shown that $\E\tr H_n = O_M(n)$ in the proof of Lemma \ref{lem:zerodiagonal}, see the computation of $\E\tr((Y^0_{n,z})^*Y_{n,z}^0)$. By Markov's inequality, we have that $\frac{1}{n}\tr H_n$ is bounded in probability. For $\frac{1}{n}\tr \widehat{H}_n$, we have
		\begin{align*}
			\E \tr \widehat{H}_n 
			\leq
			\frac{n}{2}\E[ \abs{x_{12}}^2\mathbbm{1}_{\abs{x_{12}}\leq (np_n)^{1/4}}]+\frac{n}{2}\E[ \abs{x_{21}}^2\mathbbm{1}_{\abs{x_{21}}\leq (np_n)^{1/4}}]+nM^2.
		\end{align*}
		Dropping the indicators and using Markov's inequality, it follows that $\frac{1}{n}\tr \widehat{H}_n$ is bounded in probability as well. Therefore, by \eqref{eq:levybound} and the previous discussion, $L(\nu_{H_{n,z}},\nu_{\widehat{H}_{n,z}})$ converges to 0 in probability.
		
		Now, we turn to comparing $\nu_{\widehat{H}_{n,z}}$ and $\nu_{\widetilde{H}_{n,z}}$. Again by \cite[Corollary A.42]{BaiSilverstein}, we have
		\begin{equation}\label{eq:levybound2}
			L^4(\nu_{\widehat{H}_{n,z}},\nu_{\widetilde{H}_{n,z}})\leq\frac{2}{n^3p_n}\left[\tr(\widehat{H}_{n,z}+\widetilde{H}_{n,z})\tr\left((\widehat{X}_n-\widetilde{X}_n)^*(\widehat{X}_n-\widetilde{X}_n)\right)\right].
		\end{equation}
		As above, will show that $\frac{2}{n^2p_n}\tr\left((\widehat{X}_n-\widetilde{X}_n)^*(\widehat{X}_n-\widetilde{X}_n)\right)$ converges to 0 in probability and that $\frac{1}{n}\tr(\widehat{H}_{n,z}+\widetilde{H}_{n,z})$ is bounded in probability. For the first statement, we have the expected value
		\begin{align*}
			&\E\tr\left(\widehat{X}_n-\widetilde{X}_n)^*(\widehat{X}_n-\widetilde{X}_n)\right) 
			= \sum_{i\neq j}\E[\delta_{ij}^{(n)}|\widehat{x}_{ij}|^2]\left(1-\frac{1}{\sqrt{\E|\widehat{x}_{ij}|^2}}\right)^2
			\\
			\\&\leq
			\frac{n^2p_n}{2}\max_{i\neq j}\parens{1-\frac{1}{\sqrt{\E|\widehat{x}_{ij}|^2}}}^2\parens{\E \abs{x_{12}}^2\mathbbm{1}_{\abs{x_{12}}\leq(np_n)^{1/4}}+\E \abs{x_{21}}^2\mathbbm{1}_{\abs{x_{21}}\leq(np_n)^{1/4}}}\\
			&\leq n^2p_n\max_{i\neq j}\parens{1-\frac{1}{\sqrt{\E|\widehat{x}_{ij}|^2}}}^2
		\end{align*}
		We now note that by the Dominated Convergence Theorem,
		$\lim\limits_{n\to\infty}\E|\widehat{x}_{ij}|^2=1$
		uniformly for $i\neq j$. Therefore,
		$$   
		\frac{1}{n^2p_n}\E\tr\left((\widehat{X}_n-\widetilde{X}_n)^*(\widehat{X}_n-\widetilde{X}_n)\right) = o(1)
		$$
		after another application of the Dominated Convergence Theorem. Thus, By Markov's inequality, $\frac{2}{n^2p_n}\tr\left((\widehat{X}_n-\widetilde{X}_n)^*(\widehat{X}_n-\widetilde{X}_n)\right)$ converges to 0 in probability. From the arguments above, we already know that $\frac{1}{n}\tr \widehat{H}_n$ is bounded in probability. For $\frac{1}{n}\tr \widetilde{H}_n$, we compute
		\[
		\E \tr \widetilde{H}_n = \frac{1}{np_n}\sum_{i\neq j}\E\abs{(\widetilde{X}_n)_{ij}}^2+n\abs{z}^2 \leq
		n+n\abs{z}^2\leq
		n(1+M^2).
		\]
		Again by Markov's inequality, we have that $\frac{1}{n}\tr \widetilde{H}_n$ is bounded in probability and thus by \eqref{eq:levybound2} we have $L(\nu_{\widehat{H}_{n,z}},\nu_{\widetilde{H}_{n,z}})\to 0$ in probability. Combining the convergence of $L(\nu_{H_{n,z}},\nu_{\widehat{H}_{n,z}})$ and $L(\nu_{\widehat{H}_{n,z}},\nu_{\widetilde{H}_{n,z}})$ to 0 in probability via the triangle inequality gives the desired result of this lemma.
	\end{proof}
	
	Before the comparison step, we also confirm that truncating, centering, and rescaling our matrix entries does not significantly alter the moments of the real and imaginary parts of the entries of our new matrix.
	
	\begin{lem}\label{lem:newmoments}
		With $\widetilde{X}_n$ as defined before Lemma \ref{lem:truncation},
		\begin{align*}
			&\E[\Real((\widetilde{X}_n)_{ij})^a\Imag((\widetilde{X}_n)_{ij})^b\Real((\widetilde{X}_n)_{ji})^c\Imag((\widetilde{X}_n)_{ji})^d]\\=& \E[\Real(({X}_n)_{ij})^a\Imag(({X}_n)_{ij})^b\Real(({X}_n)_{ji})^c\Imag(({X}_n)_{ji})^d] + o(p_n)
		\end{align*}
		uniformly in $i\neq j$ and for all nonnegative integers $a,b,c,d$ satisfying $a+b+c+d = 2$.
	\end{lem}
	
	\begin{proof}
		When one of $a$, $b$, $c$, or $d$ is 2, we have that (in the case of $a=2$)
		\[
		\E [\Real((X_n)_{ij})^2 - \Real((\widehat{X}_n)_{ij})^2]\leq
		p_n(\E[\Real(x_{ij})^2\mathbbm{1}_{\abs{x_{ij}}>(np_n)^{1/4}}]+(\E[\Real(x_{ij})\mathbbm{1}_{\abs{x_{ij}}\leq (np_n)^{1/4}}])^2),
		\]
		where we used the independence of $x_{ij}$ and $\delta_{ij}^{(n)}$. By the Dominated Convergence Theorem, this quantity on the right hand side is $o(p_n)$ uniformly in $i\neq j$. Now, we have that
		\begin{equation}\label{eq:3.4match2}
			\E [\Real((\widehat{X}_n)_{ij})^2 - \Real((\widetilde{X}_n)_{ij})^2]= \parens{1-\frac{1}{\E\abs{\widehat{x}_{ij}}^2}}\E\Real((\widehat{X}_n)_{ij})^2.
		\end{equation}
		
		From the proof of Lemma \ref{lem:truncation}, we have that this quantity is $o(p_n)$ uniformly in $i\neq j$ as well.
		A similar argument holds for when $b$, $c$, or $d$ is 2.
		
		For when there are only ones appearing in our sum $a+b+c+d=2$ (we will only walk through the case when $a=b=1$),  we have that
		\begin{align}
			\nonumber\E [\Real((\widehat{X}_n)_{ij})\Imag((\widehat{X}_n)_{ij})] &= p_n\E[\Real(x_{ij})\Imag(x_{ij})\mathbbm{1}_{\abs{x_{ij}}\leq(np_n)^{1/4}}]\\
			\label{eq:3.4match}&\quad -
			p_n\E[\Real(x_{ij})\mathbbm{1}_{\abs{x_{ij}}\leq(np_n)^{1/4}}]
			\E[\Imag(x_{ij})\mathbbm{1}_{\abs{x_{ij}}\leq(np_n)^{1/4}}].
		\end{align}
		We have that $\E[\Real((X_n)_{ij})\Imag((X_n)_{ij})]=p_n\E[\Real(x_{ij})\Imag(x_{ij})]$ and so subtracting this from \eqref{eq:3.4match}, we have that
		\[
		\E [\Real((\widehat{X}_n)_{ij})\Imag((\widehat{X}_n)_{ij})]-\E [\Real(({X}_n)_{ij})\Imag(({X}_n)_{ij})]=o(p_n)
		\]
		by the Dominated Convergence Theorem, uniformly in $i\neq j$. Using a similar computation like \eqref{eq:3.4match2} and the Dominated Convergence Theorem, we also have that
		\[
		\E [\Real((\widetilde{X}_n)_{ij})\Imag((\widetilde{X}_n)_{ij})]-\E [\Real((\widehat{X}_n)_{ij})\Imag((\widehat{X}_n)_{ij})]=o(p_n)
		\]
		uniformly in $i\neq j$. Similar arguments using the Dominated Convergence and various other facts already mentioned here give the desired result for any four nonnegative integers $a,b,c,d$ that sum to 2.
	\end{proof}
	
	\subsection{Step 3: The comparison step}
	
	With the truncation out of the way, we now prove a comparison lemma for our matrix $Y_{n,z}$. Before the lemma, though, we give some notation. For the vector $(u_1,u_2,u_3,u_4):=(\Real(\xi_1),\Imag(\xi_1),\Real(\xi_2),\Imag(\xi_2))$, we define the covariances
	\begin{equation}\label{eq:sparsecovariances}
		c_{ij} = \E[u_iu_j]
	\end{equation}
	for $1\leq i\leq j\leq 4$. Given our assumptions on our atom variables $(\xi_1,\xi_2)$, we do have some relations between the $c_{ij}$. For example, $c_{11}+c_{22}=1=c_{33}+c_{44}$. Also, given that $\E[\xi_1\xi_2]=\rho_1\in[0,1)$, it must be that $c_{13}-c_{24}=\rho_1$ and $c_{14}+c_{23}=0$. We furthermore define the value $\rho':=\tau+p\in[-1,1]$.
	
	\begin{lem}\label{lem:comparison}
		Let $(X_n)_{n\geq 1}$ be a sparse elliptic random matrix model coming from Definition \ref{def:sparseellipticmodel} with atom variable $(\xi_1,\xi_2)$ satisfying $\abs{\E[\xi_1\overline{\xi_2}]}<1$ and  $\E[\xi_1\xi_2] = \rho_1\in(-1,1)$, and sparsity parameters $(p_n)_{n\geq1}$ and $\tau$. Furthermore, let $(W_n)_{n\geq1}$ be a dense elliptic random matrix model coming from Definition \ref{def:ellipticmodel} with complex-valued atom variable $(\psi_1,\psi_2)$. We specify that $(\psi_1,\psi_2)$ is a jointly Gaussian vector with the vector $\Psi = (\Real(\psi_1),\Imag(\psi_1),\Real(\psi_2),\Imag(\psi_2))^T$ satisfying:
		\begin{equation}\label{eq:densecovariancematrix}
			\E\Psi\Psi^T = \begin{pmatrix}
				c_{11} & c_{12} & \rho' c_{13} & \rho' c_{14}\\
				c_{12} & c_{22} & \rho' c_{23} & \rho' c_{24}\\
				\rho' c_{13} & \rho' c_{23} & c_{33} & c_{34}\\
				\rho' c_{14} & \rho' c_{24} & c_{34} & c_{44}\\
			\end{pmatrix}.
		\end{equation}
		Recall that the $c_{ij}$ are coming from \eqref{eq:sparsecovariances}. We note here that in this setup, $\E[\psi_1\psi_2]=\rho_1\rho' = \rho_1(\tau+p)\in (-1,1)$. Then for any fixed $z\in\C$, $\nu_{Y_{n,z}}-\nu_{\frac{1}{\sqrt{n}}W_n-zI_n}\to0$ weakly in probability as $n\to\infty$, where $Y_{n,z}$ is defined in \eqref{eq:rotatedmatrix}.
	\end{lem}
	
	\begin{proof}
		This proof is inspired by and based on Lindeberg  replacement \cite{lindeberg} and other proofs. For references of some of these, see Section 2.2.4 in \cite{taotrmt} or \cite{lindebergchatterjee,taolindeberg}. For this proof, $i$ will denote the complex number $\sqrt{-1}$ and not a summation index. We follow the argument as in \cite[Lemma 7.10]{ellipticlaw} and begin by defining for any $n\times n$ matrix $A$, the shift $(\cdot)_z$ and resolvent $R(\cdot)_{z,\alpha}$
		\[
		A_z := A-zI_n\quad\text{and}\quad
		R(A)_{z,\alpha}= (A_z^*A_z-\alpha I_n)^{-1},
		\]
		where $z\in\C$ and $\alpha\in\C$ with $\Imag(\alpha)>0$. Since we are fixing these values of $z$ and $\alpha$, we will suppress the dependence on them in the notation $R(A):=R(A)_{z,\alpha}$ First, by Lemma \ref{lem:truncation}, it suffices to show that for almost all $z\in\C$,
		\[
		\nu_{( (np_n)^{-1/2}\widetilde{X}_n)_z^*((np_n)^{-1/2}\widetilde{X}_n)_z}-\nu_{(n^{-1/2}W_n)_z^*(n^{-1/2}W_n)_z}\to0
		\] 
		weakly in probability. Here, $\widetilde{X}_n$ is the truncated, centered, and rescaled version of $X_n$ as defined before Lemma \ref{lem:truncation}. Thus, without loss of generality, we may assume that the non-Bernoulli variables in the entries of $X_n$ have mean zero, unit variance, and are almost surely bounded by $(np_n)^{1/4}$. In addition, we are still assuming that the diagonal entries of $X_n$ are 0 from Lemma \ref{lem:zerodiagonal}.
		By \cite[Theorem B.9]{BaiSilverstein}, it suffices to prove for almost all $z\in\C$ and all $\alpha\in\C$ with $\Imag(\alpha)>0$, that the following Stieltjes transforms satisfy
		\[
		\frac{1}{n}\tr R((np_n)^{-1/2}X_n)-\frac{1}{n}\tr R(n^{-1/2}W_n)\to 0
		\]
		in probability. Finally, by Lemma \ref{lem:concentrationofESM}, it is enough to show for almost all $z\in\C$ and all $\alpha\in\C$ with $\Imag(\alpha)>0$ that
		\begin{equation}\label{eq:expectedstieltjes}
			\E\frac{1}{n}\tr R((np_n)^{-1/2}X_n)-\E\frac{1}{n}\tr R(n^{-1/2}W_n)\to 0.
		\end{equation}
		To deal with these expressions in \eqref{eq:expectedstieltjes}, we define for indices $a,b\in \{1,\ldots,n\}$ with $a\neq b$ (given the diagonal entries of $X_n$ are taken to be 0), the function $s$ given by
		\begin{equation}\label{eq:sdef}
			s(x_1,x_2,x_3,x_4) = \frac{1}{n}\tr R(A+x_1E_{ab}+ix_2E_{ab}+x_3E_{ba}+ix_4E_{ba}).
		\end{equation}
		Here $A$ is any fixed matrix and $E_{ab}$ is the matrix unit which has a one in entry $(a,b)$ and zeros elsewhere. Note that $s$ here depends on $a$, $b$, $z$, $\alpha$, and $A$, but for ease of notation, we will suppress this dependence.
		
		We expand $s$ via a Taylor series centered at $(0,0,0,0)$:
		\begin{equation}\label{eq:taylorexpansion}
			s(x_1,x_2,x_3,x_4) = s(0,0,0,0)+\sum_{j=1}^4\frac{\partial s}{\partial x_j}(0,0,0,0)x_j +
			\frac{1}{2}\sum_{j,k=1}^4\frac{\partial^2 s}{\partial x_j\partial x_k}(0,0,0,0)x_jx_k + \eps,
		\end{equation} 
		where
		\begin{equation}\label{eq:epsilon}
			\abs{\eps}\leq CM(\abs{x_1}^3+\abs{x_2}^3+\abs{x_3}^3+\abs{x_4}^3).
		\end{equation} Here, $C$ is an absolute constant and $$M = \sup\limits_{1\leq j,k,l\leq 4}\sup\limits_{x_1,x_2,x_3,x_4\in\R}\abs{\frac{\partial^3 s}{\partial x_j\partial x_k\partial x_l}(x_1,x_2,x_3,x_4)}.$$
		To get our hands on formulas for partial derivatives of $s$, we compute the partial derivatives of the entries of our matrix $R(A)$. Using the resolvent identity\footnote{For any $A,B\in\cal{M}_n(\C)$ and $z\in\C$ that is not an eigenvalue of either $A$ or $B$, we have $(A-zI_n)^{-1}-(B-zI_n)^{-1} = (A-zI_n)^{-1}(B-A)(B-zI_n)^{-1}$.}, we have that for any $1\leq j,k,l,m\leq n$,
		\begin{align*}
			\frac{\partial (R(A))_{jk}}{\partial \Real(A_{lm})} &= -[(R(A)A_z^*)_{jl}(R(A))_{mk}+(R(A))_{jm}(A_zR(A))_{lk}],\\
			\frac{\partial (R(A))_{jk}}{\partial \Imag(A_{lm})} &= -i[(R(A)A_z^*)_{jl}(R(A))_{mk}-(R(A))_{jm}(A_zR(A))_{lk}].
		\end{align*}
		For the partial derivative of $s$ with respect to $x_1$, we sum up the diagonal elements to recover our trace:
		\begin{align*}
			\frac{\partial s}{\partial x_1} &= \frac{1}{n}\sum_{j=1}^n\frac{\partial}{\partial \Real(A_{ab})}(R(A))_{jj}\\&=
			\frac{-1}{n}\sum_{j=1}^n[(R(A)A_z^*)_{ja}(R(A))_{bj}+(R(A))_{jb}(A_zR(A))_{aj}]\\&=
			\frac{-1}{n}\parens{(R(A)R(A)A_z^*)_{ba}+(A_zR(A)R(A))_{ab}}
			\\&=
			\frac{-2}{n}\Real\parens{(A_zR(A)R(A))_{ab}}.
		\end{align*}
		Taking the absolute value, we have
		\[
		\abs{\frac{\partial s}{\partial x_1}}\leq \frac{2}{n}\abs{(A_zR(A)R(A))_{ab}}\leq \frac{2}{n}\norm{A_zR(A)R(A)}\leq
		\frac{2}{n}\norm{\sqrt{A_z^*A_z}R(A)^2}.
		\]
		By the spectral theorem, this last quantity, the operator norm of a normal matrix, is equal to the spectral radius of said matrix. Therefore, by the spectral mapping theorem,
		\[
		\norm{\sqrt{A_z^*A_z}R(A)^2} \leq \sup_{t\geq 0}\frac{\sqrt{t}}{\abs{t-\alpha}^2}\leq
		\frac{1}{\abs{\Imag(\alpha)}^2}+\sup_{t\geq 1}\frac{t}{\abs{t-\alpha}^2}
		\leq
		\frac{1+\abs{\Imag(\alpha)}+\abs{\alpha}}{\abs{\Imag(\alpha)}^2}.
		\]
		Finally, for the $x_1$ partial derivative, we have
		$
		\frac{\partial s}{\partial x_1} = O_\alpha\parens{\frac{1}{n}}
		$.
		Note that this bound is uniform for inputs $x_1,x_2,x_3,x_4\in\R$, $z\in\C$ and $A$. By a similar argument for the partial derivative in $x_2$, $x_3$, and $x_4$, we conclude that $\frac{\partial s}{\partial x_j} = O_\alpha\parens{\frac{1}{n}}$
		uniformly in $1\leq j\leq 4$, $x_1,x_2,x_3,x_4\in\R$, $z\in\C$, and for any matrix $A$. The computations for higher partial derivatives of $s$ follow with more applications of the resolvent identity; we have analogous uniform bounds on higher partial derivatives:
		\begin{equation}\label{eq:fderivatives}
			\frac{\partial^2 s}{\partial x_j\partial x_k} = O_\alpha\left(\frac{1}{n}\right),\quad
			\frac{\partial^3 s}{\partial x_j\partial x_k\partial x_l} = O_\alpha\left(\frac{1}{n}\right).
		\end{equation}
		Of course, the dependence on $\alpha$ appearing in each $O_{\alpha}(\cdot)$ is not uniform. We can now show that \eqref{eq:expectedstieltjes} holds by utilizing our bounds on $s$ and its partial derivatives and incorporating the $O(n^2)$ off-diagonal entries of our matrix---recall that the diagonal entries of our matrix are assumed to be zero---via the triangle inequality. That is, we verify
		\begin{align*}
			\E s\left(\frac{\delta_1\Real(\xi_1)}{\sqrt{np_n}},\frac{\delta_1\Imag(\xi_1)}{\sqrt{np_n}},\frac{\delta_2\Real(\xi_2)}{\sqrt{np_n}},\frac{\delta_2\Imag(\xi_2)}{\sqrt{np_n}}\right) &= \E s\parens{\frac{\Real(\psi_1)}{\sqrt{n}},\frac{\Imag(\psi_1)}{\sqrt{n}},\frac{\Real(\psi_2)}{\sqrt{n}},\frac{\Imag(\psi_2)}{\sqrt{n}}} \\&\quad+o_\alpha(n^{-2}),
		\end{align*}
		where we take $A$ to be any matrix independent of $\xi_1$, $\xi_2$, $\psi_1$ and $\psi_2$. We recall here that $\xi_1$ and $\xi_2$ are assumed to have mean zero, unit variance, and are almost surely bounded by $(np_n)^{1/4}$.
		
		To this end, let $A$ be any such matrix independent of $\xi_1$, $\xi_2$, $\psi_1$, and $\psi_2$. For convenience, we denote by $\delta_1$ and $\delta_2$ two Bernoulli($p_n$) variables independent of $\xi_1$ and $\xi_2$ with covariance $\cov(\delta_1,\delta_2) = \tau_{n}p_n$. Then with
		\[
		x_1 = \frac{\delta_1 \Real(\xi_1)}{\sqrt{np_n}},\quad
		x_2 = \frac{\delta_1 \Imag(\xi_1)}{\sqrt{np_n}},\quad
		x_3 = \frac{\delta_2 \Real(\xi_2)}{\sqrt{np_n}},\quad
		x_4 = \frac{\delta_2 \Imag(\xi_2)}{\sqrt{np_n}},
		\]
		we have that
		\begin{align}
			\nonumber \E s(x_1,x_2,x_3,x_4) &= \E s(0,0,0,0)+\sum_{j=1}^4\E \frac{\partial s}{\partial x_j}(0,0,0,0)\E x_j \\\nonumber&+
			\frac12\sum_{j,k=1}^4\E\frac{\partial^2 s}{\partial x_j\partial x_k}(0,0,0,0)\E x_jx_k + \E\eps
			\\\label{eq:sparsetaylorexpansion}&=
			\E s(0,0,0,0)+
			\frac12\sum_{j,k=1}^4\E\frac{\partial^2 s}{\partial x_j\partial x_k}(0,0,0,0)\E x_jx_k + \E\eps.
		\end{align}
		
		Recall that the $\eps$ here is coming from \eqref{eq:epsilon}. In \eqref{eq:sparsetaylorexpansion}, we used the mean 0 hypothesis of the $\xi_i$. Applying \eqref{eq:taylorexpansion}, \eqref{eq:fderivatives}, and our bound on the $\xi_i$ from the truncation step, we get that
		\begin{align*}
			\E\abs{\eps} &\leq CM(\E\abs{x_1}^3+\E\abs{x_2}^3+\E\abs{x_3}^3+\E\abs{x_4}^3)
			\\&\leq\frac{C_\alpha}{n}\parens{2(np_n)^{-3/2}p_n(\E\abs{\xi_1}^3+\E\abs{\xi_2}^3)}\\&\leq
			\frac{C_\alpha}{n}\parens{4(np_n)^{-3/2}p_n(np_n)^{1/4}}
			= o_\alpha\left(\frac{1}{n^2}\right).    
		\end{align*}
		Similarly expanding with
		\[
		y_1 = \frac{\Real(\psi_1)}{\sqrt{n}},\quad
		y_2 = \frac{\Imag(\psi_1)}{\sqrt{n}},\quad
		y_3 = \frac{\Real(\psi_2)}{\sqrt{n}},\quad
		y_4 = \frac{\Imag(\psi_2)}{\sqrt{n}},\quad
		\]
		we have that
		\begin{align}
			\nonumber\E s(y_1,y_2,y_3,y_4) &= \E s(0,0,0,0)+\sum_{j=1}^4\E \frac{\partial s}{\partial y_j}(0,0,0,0)\E y_j \\\nonumber&+
			\frac12\sum_{j,k=1}^4\E\frac{\partial^2 s}{\partial y_j\partial y_k}(0,0,0,0)\E y_jy_k + \E\eps
			\\\label{eq:densetaylorexpansion}&=
			\E s(0,0,0,0)+\frac12\sum_{j,k=1}^4\E\frac{\partial^2s}{\partial y_j\partial y_k}(0,0,0,0)\E y_jy_k+\E\eps.
		\end{align}
		Again, we used that $\E \psi_1=\E \psi_2 =0$. Given that our Gaussian variables $\psi_1$ and $\psi_2$ have finite third moment, we have similar bounds
		as above: $\E\abs{\eps} = O_\alpha\parens{\frac{1}{n^{2.5}}} = o_\alpha\left(\frac{1}{n^2}\right)$.
		Almost everything in the above Taylor polynomial expansions will cancel upon subtraction  or have the desired $o_{\alpha}(n^{-2})$ control (in the $\eps$ error terms). Thus, we need to examine the second order terms more closely. Throughout this part, we will be using Lemma \ref{lem:newmoments} to talk about the mixed moments of the real and imaginary parts of the entries of our truncated matrix. First, we compute the covariance matrix for the real vector $v:=(x_1,x_2,x_3,x_4)^T$:
		\[
		\E vv^T = \frac{1}{n}\begin{pmatrix}
			c_{11}' & c_{12}' & \rho_{n}c_{13}' & \rho_{n}c_{14}'\\
			c_{12}' & c_{22}' & \rho_{n}c_{23}' & \rho_{n}c_{24}'\\
			\rho_{n}c_{13}' & \rho_{n}c_{23}' & c_{33}' & c_{34}'\\
			\rho_{n}c_{14}' & \rho_{n}c_{24}' & c_{34}' & c_{44}'\\
		\end{pmatrix},
		\]
		where the values of $c_{jk}'$ are the entries of the $4\times 4$ covariance matrix coming from the vector $(\Real(\xi_1),\Imag(\xi_1),\Real(\xi_2),\Imag(\xi_2))^T$. Furthermore, $\rho_{n} = \tau_n+p_n$ so that $\lim_{n\to\infty}\rho_n=\tau+p=\rho'$.
		
		Given the matrix entries of \eqref{eq:densecovariancematrix} and Lemma \ref{lem:newmoments}, it follows that $\E[x_jx_k]=\E[y_jy_k]+o(n^{-1})$ uniformly for $1\leq j,k\leq 4$. Combining this with our second derivative bound \eqref{eq:fderivatives}, we have that
		\[
		\abs{\sum_{j,k=1}^4\E\frac{\partial^2s}{\partial x_j\partial x_k}(0,0,0,0)\E x_jx_k-\sum_{j,k=1}^4\E\frac{\partial^2s}{\partial x_j\partial x_k}(0,0,0,0)\E y_jy_k} = o_\alpha\parens{\frac{1}{n^2}}.
		\]
		
		Therefore, subtracting \eqref{eq:sparsetaylorexpansion} from \eqref{eq:densetaylorexpansion}, we have that
		\[
		\E s(x_1,x_2,x_3,x_4) = \E s(y_1,y_2,y_3,y_4)+o_\alpha\parens{\frac{1}{n^2}}.
		\]
		After invoking the triangle inequality and letting $a$ and $b$ in \eqref{eq:sdef} range over the $O(n^2)$ nonzero entries of our matrices, we have
		\[
		\E \frac{1}{n}\tr R((np_n)^{-1/2}X_n)_{z,\alpha}=
		\E \frac{1}{n}\tr R(n^{-1/2}W_n)_{z,\alpha} +o_\alpha(1).
		\]
		for all $z\in\C$ and all $\alpha\in \C$ with $\Imag(\alpha)\neq0$.
	\end{proof}
	
	\subsection{Proof of Proposition \ref{prop:weakconvergence}}
	
	The preceding Lemma \ref{lem:comparison} is not enough on its own to prove Proposition \ref{prop:weakconvergence}. We need the following lemma to confirm the weak in probability convergence of the empirical singular value measures $\nu_{n^{-1/2}W_n-zI_n}$ appearing in Lemma \ref{lem:comparison}. The proof of this lemma can be found in Appendix \ref{sec:ellipticlawappendix}.
	\begin{lem}\label{lem:determinetheellipticlaw}
		Let $(W_n)_{n\geq1}$ be a dense elliptic random matrix coming from Definition \ref{def:ellipticmodel} with complex-valued atom variable $\psi = (\psi_1,\psi_2)^T$ comprised of jointly Gaussian variables with mean zero and having the following covariance structure:
		\begin{equation}\label{eq:psuedocovariancestructure}
			\E[\psi\psi^T] = \begin{pmatrix}
				\E[\psi_1^2] & \E[\psi_1\psi_2]\\
				\E[\psi_2\psi_1] & \E[\psi_2^2]
			\end{pmatrix}
			=:
			\begin{pmatrix}
				\mu_{11} & \rho\\
				\rho & \mu_{22}
			\end{pmatrix}
		\end{equation}
		and
		\begin{equation}\label{eq:covariancestructure}
			\E[\psi\psi^*] = \begin{pmatrix}
				\E[\abs{\psi_1}^2] & \E[\psi_1\overline{\psi_2}] \\
				\E[\overline{\psi_1}\psi_2] & \E[\abs{\psi_2}^2]
			\end{pmatrix}
			=:
			\begin{pmatrix}
				1 &  \mu_{12}\\
				\overline{\mu_{12}} & 1
			\end{pmatrix},
		\end{equation}
		for $\rho\in(-1,1)$ and any $\mu_{11},\mu_{12},\mu_{22}\in\C$ with $\abs{\mu_{ij}}\leq1$ for all $1\leq i,j\leq 2$. Then for almost all $z\in\C$, $\nu_{n^{-1/2}W_n-zI_n}\to\nu_z$
		weakly in probability as $n\to\infty$. Furthermore, the family $(\nu_z)_{z\in\C}$ determines the elliptic law with parameter $\rho$:
		\begin{equation}\label{eq:determineellipticspecial}
			U_{\mu_\rho}(z) = -\int_0^\infty\log(s)\,d\nu_z(s)\quad\text{for almost all $z\in\C$.}
		\end{equation}
		Here, $\mu_\rho$ is the probability measure on $\C$ having density $\frac{1}{\pi(1-\rho^2)}\mathbbm{1}_{E_\rho}$ with respect to Lebesgue measure on $\C$. (The ellipsoid $E_\rho$ is defined in \eqref{eq:ellipse}.) 
	\end{lem}
	
	\begin{remark}
		The only second moment assumptions we make on the atom variable $(\psi_1,\psi_2)$ appearing in Lemma \ref{lem:determinetheellipticlaw} is that they have variance 1 (see \eqref{eq:covariancestructure}) and satisfy $\E[\psi_1\psi_2]=\rho$ (see \eqref{eq:psuedocovariancestructure}).
		Although similar to Lemma 7.11 in \cite{ellipticlaw}, this Lemma \ref{lem:determinetheellipticlaw} does not demand the specialized covariance structure imposed on the atom variables in the main results of \cite{ellipticlaw}.
	\end{remark}
	
	\begin{proof}[Proof of Proposition \ref{prop:weakconvergence}]
		For our matrix model given by $M_n=(np_n)^{-1/2}(X_n+F_n)$, we have that for all $z\in\C$, $\nu_{M_n-zI_n}-\nu_{Y_{n,z}}\to 0$ weakly almost surely. This follows because $\rank(F_n)=o(n)$. For a reference of this result, see \cite[Theorem A.44]{BaiSilverstein}. From Lemma \ref{lem:comparison} and Lemma \ref{lem:determinetheellipticlaw}, we have that for our matrix $Y_{n,z}=(np_n)^{-1/2}X_n-zI_n$, $\nu_{Y_{n,z}} \to \nu_z$ weakly in probability as $n\to\infty$ for almost all $z\in\C$ and thus $\nu_{M_n-zI_n} \to \nu_z$  weakly in probability as $n\to\infty$ for almost all $z\in\C$, where $(\nu_z)_{z\in\C}$ satisfies \eqref{eq:determineellipticlaw}. 
	\end{proof}
	
	\section{Uniform integrability of the logarithm}\label{sec:uniformintegrability}
	
	The main result of this section
	concerns the uniform integrability of the logarithm with respect to the singular value measures $\nu_{M_n-zI_n}$.
	
	\begin{prop}[Uniform integrability of the logarithm]\label{prop:uniformintegrability}
		Let $(X_n)_{n\geq1}$ be a sequence of random matrices coming from Definition \ref{def:sparseellipticmodel} and satisfying the hypotheses of Theorem \ref{thm:auxiliarymaintheorem} with the sparsity parameters $(p_n)_{n\geq1}$ satisfying Assumption \ref{ass:sparsitylevel}. Furthermore, let $(F_n)_{n\geq 1}$ be a sequence of deterministic matrices over $\C$ satisfying $\rank(F_n) = o(n)$ and \eqref{eq:boundedhilbertschmidt}. Then for all $z\in\C$, $\log(\cdot)$ is uniformly integrable in probability for $(\nu_{M_n-zI_n})_{n\geq1}$. In other words, for all $\eps>0$,
		\[
		\lim_{T\to\infty}\limsup_{n\to\infty}\P\parens{\int_{\abs{\log s}>T}\abs{\log s}\,d\nu_{M_n-zI_n}(s)>\eps}=0.
		\]
	\end{prop}
	
	To prove this proposition, we require both control of the least singular value (coming from Proposition \ref{prop:leastsingularvalue})
	and the intermediate singular values (coming from Proposition \ref{prop:othersingularvalues}):
	
	\begin{proof}[Proof of Proposition \ref{prop:uniformintegrability}]
		We fix $z\in\C$ and $\eps>0$. For any $T>1$ and $n\in\N$,
		we consider the event
		\[
		A_{T,n} = \set{\int_{\abs{\log s}>T}\abs{\log s}\,d\nu_{M_n-zI_n}(s)>\eps} = \set{\frac1n\sum_{i=1}^n\mathbbm{1}\{\abs{\log\sigma_i}>T\}\abs{\log\sigma_i}>\eps}
		\]
		where $\sigma_n\leq \cdots\leq\sigma_1$ are the singular values of $M_n-zI_n$. We upper bound the probability $\P(A_{T,n})$ by
		\[
		\P(A_{T,n})\leq \P(A_{T,n}^{(1)})+\P(A_{T,n}^{(2)}).
		\]
		These events on the right are defined via  small and large singular values:
		\begin{align*}
			A_{T,n}^{(1)} &:= \set{\frac1n\sum_{i=n-n^{1-q_n/4}}^n\mathbbm{1}\{\abs{\log\sigma_i}>T\}\abs{\log\sigma_i}>\frac{\eps}{2}},\\
			A_{T,n}^{(2)} &:= 
			\set{\frac1n\sum_{i=n^{1-q_n/4}}^{n-1}\mathbbm{1}\{\abs{\log\sigma_{n-i}}>T\}\abs{\log\sigma_{n-i}}>\frac{\eps}{2}}.
		\end{align*}
		
		We further decompose these two sets to isolate the events in which the singular values are large or small:
		\begin{align*}
			A_{T,n}^{(1)} &\subseteq \set{\frac1n\sum_{i=n-n^{1-q_n/4}}^n\mathbbm{1}\{\log\sigma_i>T\}\log\sigma_i>\frac{\eps}{4}}\cup\set{\frac1n\sum_{i=n-n^{1-q_n/4}}^n\mathbbm{1}\{\log(\sigma_i^{-1})>T\}\log(\sigma_i^{-1})>\frac{\eps}{4}}\\
			&=:B_{T,n}^{(1)}\cup C_{T,n}^{(1)}\\
			A_{T,n}^{(2)} &\subseteq \set{\frac1n\sum_{i=n^{1-q_n/4}}^{n-1}\mathbbm{1}\{\log\sigma_{n-i}>T\}\log\sigma_{n-i}>\frac{\eps}{4}}\cup\set{\frac1n\sum_{i=n^{1-q_n/4}}^{n-1}\mathbbm{1}\{\log(\sigma_{n-i}^{-1})>T\}\log(\sigma_{n-i}^{-1})>\frac{\eps}{4}}\\
			&=:B_{T,n}^{(2)}\cup C_{T,n}^{(2)}.
		\end{align*}
		
		To bound $\P(C_{T,n}^{(1)})$, we use the control of the least singular value from Proposition \ref{prop:leastsingularvalue}. Namely, if $E_n$ is the event that $\sigma_n> (Cn)^{-5\log^2n}$, then $\P(E_n^c)=o(1)$ by Proposition \ref{prop:leastsingularvalue}, independent of $T$. On $E_n$, we have that $\log\sigma_n\geq-5\log^2n\log(Cn)$ and similarly for all larger singular values $\sigma_i$ for $n-n^{1-q_n/4}\leq i\leq n-1$. Thus, we bound
		\begin{align*}
			\P(C_{T,n}^{(1)})&\leq \P(E_n^c)+\P(E_n\cap C_{T,n}^{(1)})\\&\leq
			\P(E_n^c) + \P\parens{E_n\cap\set{\frac{5\log^2n\log(Cn)}
					{n^{q_n/4}}>\frac{\eps}{4}}}
		\end{align*}
		for all $n\gg1$ with this term being independent of $T$. Taking the limsup here gives the desired convergence to 0 for any $T$: $\lim_{T\to\infty}\limsup_{n\to\infty}\P(C_{T,n}^{(1)})=0$. Note that we used our assumption that $q_n\log\log n\to\infty$ here.
		
		For $\P(C_{T,n}^{(2)})$, we have by Proposition \ref{prop:othersingularvalues} (with $(np_n)^{-1/2}F_n-zI_n$ playing the role of $M$) that almost surely for $n\gg1$ and for each $n^{1-q_n/4}\leq i\leq n-1$,
		$
		\sigma_{n-i}\geq c_0\frac{i}{n}
		$ 
		for some absolute constant $c_0$. With this almost sure bound, we have that on this event guaranteed by Proposition \ref{prop:othersingularvalues},
		\[
		\frac1n\sum_{i=n^{1-q_n/4}}^{n-1}\mathbbm{1}\{\log(\sigma_{n-i}^{-1})>T\}\log(\sigma_{n-i}^{-1}) \leq
		\frac{1}{n}\sum_{i\leq (n/c_0)\exp(-T)}\parens{\log(n/i)+\abs{\log c_0}}
		\]
		This sum on the right involving the first summand is recognized as a Riemann sum and upper bounded by the integral
		\[
		\int_0^{\exp(-T)/c_0}\log(1/x)\,dx = \frac{e^{-T}}{c_0}(T+1+\log c_0)
		\]
		and the second sum may be upper bounded by $\frac{e^{-T}}{c_0}\abs{\log c_0}$.
		Both of these bounds are uniform in $n$, as long as $n\gg1$. Therefore, we have that by taking $T\to\infty$ and a limsup in $n$, $\lim_{T\to\infty}\limsup_{n\to\infty}\P(C_{T,n}^{(2)})=0$ as well. 
		
		Now, we deal with the events $B_{T,n}^{(i)}$. Both arguments are similar, so we show the argument to control $\P(B_{T,n}^{(1)})$. On this set (as well as on $B_{T,n}^{(2)}$), the singular values are large. Namely, on  the event $\set{\log(\sigma_i)>T}$, we have that $\log\sigma_i\leq Te^{-2T}\sigma_i^2$, using that the function $x\mapsto xe^{-2x}$ is decreasing on $[T,\infty)$. Thus, we bound
		\begin{align*}
			\P(B_{T,n}^{(1)}) \leq
			\P\parens{Te^{-2T}\cdot\frac1{n}\sum_{i=1}^n\sigma_i^2>\frac{\eps}{4}}&\leq\P\parens{Te^{-2T}\norm{M_n-zI_n}^2_{HS}>\frac{n\eps}{4}}\\&\leq Te^{-2T}\cdot\frac{4}{\eps n} \E\norm{M_n-zI_n}_{HS}^2,
		\end{align*}
		
		by Markov's inequality. Therefore, $\lim_{T\to\infty}\limsup_{n\to\infty}\P(B_{T,n}^{(1)})=0$ after applying our uniform bound (in $n$) on $\frac1n\E \norm{M_n-zI_n}_{HS}^2$ as calculated in \eqref{eq:hilbertshcmidtbound}. The argument for controlling $\P(B_{T,n}^{(2)})$ is similar.
		Bringing all of these bounds together gives the desired uniform integrability result of the proposition.
	\end{proof}
	In the following two subsections, we provide the results used in the proof of Proposition \ref{prop:uniformintegrability}.
	
	\subsection{A bound on the least singular value}\label{sec:leastsingularvalue}
	
	This subsection is dedicated to stating a proposition about the least singular value of our sparse elliptic random matrix model. The proof of this result can be found in Section \ref{sec:lsvappendix}. 
	
	\begin{prop}[Least singular value]\label{prop:leastsingularvalue}
		Let $(X_n)_{n\geq 1}$ be a sparse elliptic random matrix model coming from Definition \ref{def:sparseellipticmodel} and satisfying the hypotheses of Theorem \ref{thm:auxiliarymaintheorem} with the sparsity parameters $(p_n)_{n\geq1}$ satisfying Assumption \ref{ass:sparsitylevel}. Furthermore, let $(F_n)_{n\geq 1}$ be a sequence of deterministic matrices over $\C$ that satisfy \eqref{eq:boundedhilbertschmidt}\footnote{Note that for control of the least singular value, we do not require the control $\rank(F_n) = o(n)$ as we did throughout Section \ref{sec:convofESVMs}}. Then there exists a constant $C>0$ depending on the atom variable $(\xi_1,\xi_2)$ and other model parameters such that for any $z\in\C$,\[
		\P\parens{\sigma_n(M_n-zI_n)\leq (Cn)^{-5\log^2n}}=o(1).
		\]
	\end{prop}
	
	We also state, for completeness, control of the largest singular value (or operator norm) of $M_n-zI_n$.
	
	\begin{lem}[Largest singular value]\label{lem:largestsv}
		Under the hypotheses of Proposition \ref{prop:leastsingularvalue}, we have that for every $z\in\C$
		\[
		\P(\sigma_1(M_n-zI_n)\geq n\sqrt{p_n})\leq \frac{3(\abs{z}^2+1+\var(x_{11})+n^{-2}p_n^{-1}\norm{F_n}_{HS}^2)}{np_n}.
		\]
	\end{lem}
	The proof of this result can be found in Appendix \ref{sec:app:largestsingval}.

	\begin{remark}\label{rem:controlofopnorm}
		Note here that this result together with our assumptions on $(F_n)_{n\geq1}$ and the fact that $np_n\to\infty$,
		$\P(\sigma_1(M_n-zI_n)\geq Cn\sqrt{p_n}) = o(1)$ for any $C\geq 1$.
	\end{remark}
	
	\subsection{A bound on the intermediate singular values}\label{sec:intermediatesingvals}
	
	This result is similar to those proven by Tao and Vu in \cite[Section 5]{taouniversality}. 
	
	\begin{prop}\label{prop:othersingularvalues}
		Let $(X_n)_{n\geq1}$
		be a sequence of random matrices coming from Definition \ref{def:sparseellipticmodel} and satisfying the hypotheses of Theorem \ref{thm:auxiliarymaintheorem}. There exists $c_0 > 0$ such that the following holds. Then almost surely for all $n\gg1$, for all $n^{1-q_n/4}\leq i \leq n-1$, and for any deterministic $n\times n$ matrix $M$, $\sigma_{n-i}((np_n)^{-1/2}X_n+M)\geq c_0\frac{i}{n}$.
	\end{prop}
	
	To prove Proposition \ref{prop:othersingularvalues}, we first prove an analog to Lemma 7.6 coming from \cite{ellipticlaw}.
	
	\begin{lem}[Distance of a random vector to a subspace]\label{lem:disttosubspace}
		Let $x$ and $y$ be complex-valued random variables with mean zero and unit variance. Let $\eta:=(\eta_1,\ldots,\eta_{n-1})$ be a random vector
		in $\C^{n-1}$  of the form $\eta_i = p_{n}^{-1/2}\delta_i^{(n)}\xi_i$, where we will assume
		\begin{itemize}
			\item[(i)] The variables $\eta_1,\ldots,\eta_{n-1}$ are jointly independent.
			
			\item[(ii)] Each $\delta_{i}^{(n)}$ is distributed as a Bernoulli($p_{n}$) variable and each $\xi_i$ is distributed as either $x$ or $y$.
			\item[(iii)] For each $1 \leq i \leq n-1$, $\delta_{i}^{(n)}$ and $\xi_i$ are independent.
		\end{itemize}	
		Under these assumptions, there exist absolute constants $C,c_0>0$ such that the following holds. For any $0 < a < \tfrac12$, all $n\gg1$, any deterministic
		vector $v\in\C^{n-1}$, and any deterministic subspace $H$ of $\C^{n-1}$ with $1\leq \dim(H)\leq n-1-\frac{n-1}{(np_n)^a}$,
		\[
		\P\parens{\dist(R,H) \leq \tfrac12\sqrt{n-1-\dim(H)}}\leq C\exp(-c_0(np_n)^{1-2a}),
		\]
		where $R = \eta+v$.
	\end{lem}
	
	This result is based on similar results like \cite[Corollary 2.1.19]{taotrmt}.
	
	\begin{proof}
		Throughout the proof, we define $k := n-1 -\dim (H)$ so that $k\geq \frac{n-1}{(np_n)^a}$. If $H'$ is the subspace spanned by $H$ and $v$, then $\dim(H')\leq \dim(H)+1$ and  $\dist(\eta,H')=\dist(R,H')\leq \dist(R,H)$. Thus, as we only consider $n\gg1$ and we only increase the dimension by 1, we may without loss of generality assume that $v=0$. We begin with a truncation of our non-Bernoulli variables $\xi_i$ making up the $\eta_i$. By Markov's inequality and our assumptions on $\xi_i$, we have $\P(\abs{\xi_i} > (np_n)^{a/2}) \leq (np_n)^{-a}$ for all $1\leq i\leq n-1$. For all such $i$, it follows that $\P(\abs{\eta_i} > (np_n)^{a/2}/\sqrt{p_n}) \leq (np_n)^{-a}$ as well.
		
		To see how many of the $\eta_i$ we can truncate with high probability, define the random subset of `bad' indices  $J :=\set{1\leq i\leq n-1:\abs{\eta_i}> (np_n)^{a/2}/\sqrt{p_n}}$.

		By a special case of Hoeffding's inequality\footnote{If $X_1,\ldots,X_n$ are independent random variables taking values in $[0,1]$ and $S_n = X_1+\cdots+X_n$ is their sum, then $\P(S_n\geq \E S_n+ t)\leq \exp(-2t^2/n)$ for all $t>0$.} \cite{hoeffding}:, $\P\parens{\abs{J}>\E|J| + t} \leq \exp(-2t^2/n)$
		for all $t>0$. Since $\E|J| = np_n\max\limits_{i=1,2}\P(\abs{\xi_i}>(np_n)^{a/2})=o( (np_n)^{1-a}) = o(k)$, by Hoeffding's inequality, we have that for all $n\gg1$, $\P(|J|>k/4) \leq \exp(-ck^2/n)\leq \exp(-c(np_n)^{1-2a})$, where $c>0$ is an absolute constant.
		
		With this bound, we turn to bounding
		\[
		\P(\dist(R,H)\leq \tfrac12\sqrt{k},|J|\leq k/4).
		\]
		To do so, we fix a deterministic subset $S\subset \set{1,\ldots,n-1}$ satisfying $|S|\leq k/4$. We will prove the desired result after conditioning on the event  $\Omega_S = \set{J = S}$ as well as the $\sigma$-algebra $\sigma_S$ generated by the random variables $\{\eta_i:i\in S\}$. For notation's sake we will denote any further conditioning with respect to $\Omega_S$ and $\sigma_S$ by $\P_S$ and $\E_S$.
		
		In an application of Talagrand's concentration inequality to come, we would like to use that $R=\eta$ has mean zero. On the event $\Omega_S$, though, it might be that $\E_S [R]\neq 0$. To rectify this, we define
		\[
		Y = (\eta-\E_S\eta)_{S^c} = \eta - (\E_S\eta)_{S^c} - \eta_{S}\in \C^{n-1},
		\]
		where for any vector $v\in\C^{n-1}$ and collection of indices $S\subseteq\{1,\ldots,n-1\}$, we define the vector $v_S\in\C^{n-1}$ by $(v_S)_i = v_i\mathbbm{1}_{i\in S}$. If we let $W$ be the subspace spanned by $H$, $(\E_S\eta)_{S^c}$, and $\eta_{S}$, then we have
		\[
		\dist(R,H)\geq \dist(R,W)= \dist(Y,W)
		\]
		and $\dim(W)\leq \dim(H)+2$. By Talagrand's inequality \cite{talagrand} (also see \cite[Theorem 2.1.13]{taotrmt}), we have for any $t>0$,
		\begin{equation}\label{eq:talagrand}
			\P_S\parens{\abs{\dist(Y,W)-M_S}\geq t}\leq C'\exp\left(\frac{-c'p_nt^2}{(np_n)^a}\right),
		\end{equation}
		where $C',c>0$ are absolute constants. Here $M_S$ is a median of $\dist(Y,W)$ under $\Omega_S$ and $\sigma_S$. Integrating the tail bound given by Talagrand's inequality, we may compare the median $M_S$ to $\E_S\dist^2(Y,W)$ via
		\[
		M_S \geq \sqrt{\E_S\dist^2(Y,W)}-C''(np_n)^{a/2}/\sqrt{p_n},
		\]
		where $C''>0$ is an absolute constant. We now derive an estimate of the expected value $\E_S\dist^2(Y,W)$. If $P$ denotes the projection matrix onto $W^\perp$, then
		\begin{equation}\label{eq:projection}
			\E_S\dist^2(Y,W) = \E_S[Y^*PY]=\sum_{i,j=1}^{n-1}\E_S[\overline{Y_{i}}Y_{j}]P_{ij} =\sum_{i=1}^{n-1}\E_S[\abs{Y_i}^2]P_{ii}=\sum_{i\not \in S}\E_S[\abs{Y_i}^2]P_{ii}.
		\end{equation}
		First, $\E_S|Y_i|^2 = 1-o(1)$ uniformly in $i\not\in S$.
		This follows from the computation
		\[
		\E_S|Y_i|^2 = \frac{\E[|\xi_i|^2\mathbbm{1}_{|\xi_i|\leq L_n}]}{1-p_nr_n}-\frac{p_n|\E [\xi_i\mathbbm{1}_{|\xi_i|\leq L_n}]|^2}{(1-p_nr_n)^2},
		\]
		where $L_n =(np_n)^{a/2}$ and $r_n = \P(|\xi_i|> L_n)$. This uniform $1-o(1)$ control in $i\not\in S$ follows from applications of the Dominated Convergence Theorem, with our assumption that each $\xi_i$ is distributed as either $x$ or $y$.
		
		Since $\tr P = n-1 - \dim W$ and $0\leq P_{ii}\leq 1$, we have
		\[
		\sum_{i\not\in S} P_{ii} = \tr P-\sum_{i\in S} P_{ii} \geq k-1-|S|.
		\]
		We continue from \eqref{eq:projection}. On the event $\set{|S| \leq k/4}$, it follows that $\sum_{i\not\in S} P_{ii} \geq 3k/4 - 1$ and furthermore for $n\gg1$,
		\[
		\E_S\dist^2(Y,W) \geq (1-o(1))\cdot\frac{3}{4}k.
		\]
		
		Bringing everything together, on $\set{|S| \leq k/4}$, we have that if $\dist(Y,W)\leq \frac12\sqrt{k}$, then $$\dist(Y,W) - M_S \leq \parens{\frac12-\sqrt{\tfrac{3}{4}(1-o(1))}+C''\frac{(np_n)^{a/2}}{\sqrt{p_nk}}}\sqrt{k}.$$
		
		By our assumption on $k$, this third expression in the parentheses is $o(1)$ and thus the entire expression in the parentheses is
		\[
		\dist(Y,W) - M_S \leq \parens{\frac12-\frac{\sqrt{3}}{2}+o(1)}\sqrt{k}.
		\]
		It follows that we may pick an absolute constant $c_1>0$ so that for $n\gg 1$, $\dist(Y,W)-M_S\leq-c_1\sqrt{k}$. An application of \eqref{eq:talagrand} with $t = c_1\sqrt{k}$ gives
		\[
		\P_S(\dist( Y,W)\leq \tfrac12\sqrt{k},|S|\leq k/4)\leq
		C'\exp\parens{\frac{-c'c_1^2p_nk}{(np_n)^a}}\leq C'\exp\parens{-c_0(np_n)^{1-2a}}.
		\]
		using our assumption that $k\geq \frac{n-1}{(np_n)^a}$. Integrating out the conditioning and applying our bound on $\P(|J|>k/4)$, we have our desired result.
	\end{proof}
	
	We now apply Lemma \ref{lem:disttosubspace} to our specific assumptions on the sparsity parameters $(p_n)_{n\geq 1}$.
	
	\begin{cor}\label{cor:disttosubspace}
		Under the assumptions and definition of $\eta$ coming from Lemma \ref{lem:disttosubspace}, there exist absolute constants $C,c_0>0$ such that the following holds. For all $n\gg1$, any deterministic vector $v\in\C^{n-1}$, and any deterministic subspace $H$ of $\C^{n-1}$ with $1\leq \dim (H)\leq n-1 - n^{1-q_n/4}$,
		\[
		\P\parens{\dist(R,H) \leq \tfrac12\sqrt{{n-1}-\dim(H)}}\leq C\exp(-c_0n^{q_n/2}),
		\]
		where $R = \eta+v$. Recall that  the existence of the sequence $(q_n)_{n\geq1}$ is invoked using Assumption \ref{ass:sparsitylevel}.
	\end{cor}
	\begin{proof}
		Taking $a = \tfrac14$ in Lemma \ref{lem:disttosubspace}, the result follows from Assumption \ref{ass:sparsitylevel}, namely, that $np_n\geq n^{q_n}$.
	\end{proof}

	Since it follows more or less standard arguments found in \cite{ellipticlaw,taotrmt}, we give a truncated proof of Proposition \ref{prop:othersingularvalues}.
	
	\begin{proof}[Proof of Proposition \ref{prop:othersingularvalues}]
		This proof follows exactly as in \cite[Lemma 7.5]{ellipticlaw} found in \cite[Section 7.1]{ellipticlaw} up to the following changes.
		\begin{itemize}
			\item Replace each appearance of the dense elliptic matrix $X_n$ by our sparse elliptic matrix $p_{n}^{-1/2}X_n$.
			
			\item Replace each citation and application of \cite[Lemma 7.6]{ellipticlaw} by our Corollary \ref{cor:disttosubspace}. The summability required for the Borel--Cantelli Lemma follows via Assumption \ref{ass:sparsitylevel}, namely, that $q_n\log\log n\to \infty$. Our analogous probability coming from the union bound is of the form
			\[
			\sum_{n=1}^\infty C'n^2 \exp(-c_0n^{q_n/2}),
			\]
			where for $n\gg 1$, $C'n^2 \exp(-c_0n^{q_n/2}) \leq C'n^{-3}$, say.
		\end{itemize}
	\end{proof}
	
	\subsection{Proof of Theorem \ref{thm:auxiliarymaintheorem}}
	
	With our previous propositions, we are now in a position to prove Theorem \ref{thm:auxiliarymaintheorem}.
	
	\begin{proof}[Proof of Theorem \ref{thm:auxiliarymaintheorem}]

		For the sequence of empirical spectral measures $(\mu_{M_n-zI_n})_{n\geq 1}$ and corresponding empirical singular value measures $(\nu_{M_n-zI_n})_{n\geq 1}$ associated with this matrix, we have proven that for almost every $z\in\C$,
		\begin{itemize}
			\item[(i)] $\nu_{M_n-zI_n}$ converges weakly in probability to $\nu_z$ as $n\to\infty$ (this is Proposition \ref{prop:weakconvergence}).
			\item[(ii)] The family $(\nu_z)_{z\in\C}$ satisfies \eqref{eq:determineellipticlaw} (this is Lemma \ref{lem:determinetheellipticlaw}).
			\item[(iii)] $\log(\cdot)$ is uniformly integrable for $(\nu_{M_n-zI_n})_{n\geq1}$ (this is Proposition \ref{prop:uniformintegrability}).
		\end{itemize}
		Therefore, by Lemma \ref{lem:hermitization}, there exists a $\mu\in\cal{P}(\C)$ such that $\mu_{M_n}$ converges weakly in probability to $\mu$ as $n\to\infty$. Moreover, since our family $(\nu_z)_{z\in\C}$ satisfies \eqref{eq:determineellipticlaw}, we have that
		\[
		U_\mu(z) = -\int_0^\infty \log(s)\,d\nu_z(s) = U_{\mu_{\rho_1,0,\tau,p}}(z)
		\]
		for almost all $z\in\C$, so that $\mu =\mu_{\rho_1,0,\tau,p}$. Recall that we are using the unicity of the logarithmic potential as seen in \cite[Lemma 4.1]{aroundthecirc}.
	\end{proof}

	\section{Proof of Proposition \ref{prop:leastsingularvalue}}\label{sec:lsvappendix}
	
	This section is dedicated to the proof of Proposition \ref{prop:leastsingularvalue}. More specifically, we will be proving and applying the following theorem inspired by \cite[Theorem 5.1]{andrewelliptic}. We will be combining the arguments from the proof of \cite[Theorem 5.1]{andrewelliptic} as well as the proof of \cite[Theorem 4.1]{gotzecircularlaw}.
	
	\begin{thm}\label{thm:andrewsmallsingval}
		Let $Z_n\in\cal{M}_n(\C)$
		be a random matrix satisfying the following: $\sqrt{p_n}Z_n$ is distributed as a sparse elliptic random matrix coming from Definition \ref{def:sparseellipticmodel} and satisfying the hypotheses of Theorem \ref{thm:auxiliarymaintheorem}. Namely, the sparsity parameters $(p_n)_{n\geq 1}$ are as in  Definition \ref{def:sparseellipticmodel} and satisfy Assumption \ref{ass:sparsitylevel}. Suppose there exists $a > 0$ such that
		the events 
		\[
		\mathcal{E}_{ij} := \{ |Z_{ij}| \leq \tfrac{a}{\sqrt{p_n}}, |Z_{ji}| \leq \tfrac{a}{\sqrt{p_n}} \} 
		\]
		defined for $i \neq j$ satisfy 
		\[ \min_{i < j} \P(\mathcal{E}_{ij})\geq b > 0, \quad  \min_{i \neq j} \var(Z_{ij} \mid \mathcal{E}_{ij})\geq \sigma^2 > 0, \]
		and 
		\[ \max_{i < j} \left| { \corr( Z_{ij} \mid \mathcal{E}_{ij}, Z_{ji} \mid \mathcal{E}_{ji}) } \right|\leq \rho < 1. \]
		Then for each $K\geq 1$, there exists $C = C(a, b, \sigma,\rho,K) > 0$ such that for any deterministic matrix $M\in \cal{M}_n(\C)$, 
		\[ 
		\P \left( \sigma_n(Z_n +M) \leq (Cn)^{-4\log^2n}, \sigma_1(Z_n+M) \leq Kn^{3/2}\sqrt{p_n} \right) = o_{a,b,\sigma,\rho,K}(1).
		\]
	\end{thm}
	
	Assuming this Theorem \ref{thm:andrewsmallsingval}, we give a proof of Proposition \ref{prop:leastsingularvalue}.
	
	\begin{proof}[Proof of Proposition \ref{prop:leastsingularvalue}]
		For ease of notation, we define our scaled matrix by $Y_n = (np_n)^{-1/2}X_n$ and the deterministic summand by $M = (np_n)^{-1/2}F_n-zI_n$. We will be using the fact that
		\begin{align}
			\label{eq:lsvprob}&\P\parens{\sigma_n(M_n-zI_n)\leq (Cn)^{-5\log^2n}}=\P\parens{\sigma_n(Y_n+M)\leq (Cn)^{-5\log^2n}}\\ \nonumber&\leq
			\P\parens{\sigma_n(Y_n+M)\leq (Cn)^{-4\log^2n-1/2},\sigma_1(Y_n+M)\leq n\sqrt{p_n}}+\P(\sigma_1(Y_n+M)\geq n\sqrt{p_n}).
		\end{align}
		From Lemma \ref{lem:largestsv} and Remark \ref{rem:controlofopnorm}, we have that
		\[
		\P(\sigma_1(Y_n+M)\geq n\sqrt{p_n})=
		\P(\sigma_1(M_n-zI_n)\geq n\sqrt{p_n}) = o(1),
		\]
		so we are now in a place to be able to apply Theorem \ref{thm:andrewsmallsingval}. In order to use this result, we note that $Y_n = n^{-1/2}Z_n$, where $Z_n := p_n^{-1/2}X_n$. Thus, we need to show that there exists $a > 0$ such that
		the events 
		\begin{equation} \label{eq:defEij}
			\mathcal{E}_{ij}^a := \left\{ |(Z_n)_{ij}| \leq \frac{a}{\sqrt{p_n}}, |(Z_n)_{ji}| \leq \frac{a}{\sqrt{p_n}} \right\} 
		\end{equation}
		defined for $i \neq j$ satisfy 
		\[ \min_{i < j} \P(\mathcal{E}_{ij}^a) \geq b > 0, \quad  \min_{i \neq j} \var((Z_n)_{ij} \mid \mathcal{E}_{ij}^a)\geq \sigma^2 > 0, \]
		and 
		\[ \max_{i < j} \left| { \corr( (Z_n)_{ij} \mid \mathcal{E}_{ij}^a, (Z_n)_{ji} \mid \mathcal{E}_{ji}^a) } \right|\leq \rho < 1. \]
		We will also be proving that $a$, $b$, $\sigma$, and $\rho$ do not depend on $n$.
		Since the entries of our matrix have identical distributions above and below the diagonal, we need only check that there exists $a'>0$ such that for all $a\geq a'$,  $\P(\cal{E}_{12}^a)\geq b>0$, $$\min(\var((Z_n)_{12}\given \cal{E}_{12}^a),\var((Z_n)_{21}\given \cal{E}_{21}^a))\geq \sigma^2>0,$$ and $\abs{\corr((Z_n)_{12}\given\cal{E}^a_{12}, (Z_n)_{21}\given\cal{E}^a_{21})}\leq \rho<1$.
		
		We have that $b>0$ for $a'$ large enough by Markov's inequality:
		\begin{align*}
			&\P((\cal{E}_{12}^a)^c) \leq \P(\abs{(Z_n)_{12}}>ap_n^{-1/2})+\P(\abs{(Z_n)_{21}}>ap_n^{-1/2})\\&=
			\P (\delta_{12}^{(n)}\abs{x_{12}}>a)+\P (\delta_{21}^{(n)}\abs{x_{21}}>a)=
			p_n(\P (\abs{x_{12}}>a)+\P (\abs{x_{21}}>a))\leq
			\frac{2}{a^2}.
		\end{align*}
		As long as $a'>2$, say, then $b>0$.
		
		For $\sigma^2$, we only check that $\var((Z_n)_{12}\given\cal{E}^a_{12})\geq \sigma^2>0$ as the argument for $\var((Z_n)_{21}\given\cal{E}^a_{21})$ is similar. This conditional variance is equal to
		\[
		\var((Z_n)_{12}\given \cal{E}^a_{12}) = \frac{1}{p_n}\var(\delta^{(n)}_{12}x_{12}\given \cal{E}^a_{12})=
		\frac{1}{p_n}\parens{\frac{\E(\delta_{12}^{(n)}|x_{12}|^2\mathbbm{1}\{\cal{E}_{12}^a\})}{\P(\cal{E}_{12}^a)}-\abs{\frac{\E(\delta_{12}^{(n)}x_{12}\mathbbm{1}\{\cal{E}_{12}^a\})}{\P(\cal{E}_{12}^a)}}^2}.
		\]
		This second moment is calculated to be
		\begin{align*}
			\E(\delta_{12}^{(n)}|x_{12}|^2\mathbbm{1}\{\cal{E}_{12}^a\}) &= 
			\P(\delta_{12}^{(n)}=1,\delta_{21}^{(n)}=0)\E(\abs{x_{12}}^2\mathbbm{1}_{\abs{x_{12}}\leq a})\\&
			\qquad+\P(\delta_{12}^{(n)}=1,\delta_{21}^{(n)}=1)\E(\abs{x_{12}}^2\mathbbm{1}_{\abs{x_{12}}\leq a}\mathbbm{1}_{\abs{x_{21}}\leq a})
		\end{align*}
		and the mean is similarly given by
		\begin{align*}
			\E(\delta_{12}^{(n)}x_{12}\mathbbm{1}\{\cal{E}_{12}^a\}) &= 
			\P(\delta_{12}^{(n)}=1,\delta_{21}^{(n)}=0)\E(x_{12}\mathbbm{1}_{\abs{x_{12}}\leq a})\\&\qquad+\P(\delta_{12}^{(n)}=1,\delta_{21}^{(n)}=1)\E(x_{12}\mathbbm{1}_{\abs{x_{12}}\leq a}\mathbbm{1}_{\abs{x_{21}}\leq a}).
		\end{align*}
		By the Dominated Convergence Theorem and our mean and variance assumptions on $x_{12}$, we choose $a_1\geq 2$ large enough so that for all $a\geq a_1$, both
		\[
		\min(\E(\abs{x_{12}}^2\mathbbm{1}_{\abs{x_{12}}\leq a}),\E(\abs{x_{12}}^2\mathbbm{1}_{\abs{x_{12}}\leq a}\mathbbm{1}_{\abs{x_{21}}\leq a}))>\frac34
		\]
		and
		\[
		\max\{\abs{\E x_{12}\mathbbm{1}_{\abs{x_{12}}\leq a}},\abs{\E x_{12}\mathbbm{1}_{\abs{x_{12}}\leq a}\mathbbm{1}_{\abs{x_{21}}\leq a}}\}<\frac14
		\]
		hold.
		Since for such $a\geq a_1$, we also have that $\frac12<\P(\cal{E}_{12}^a)\leq1$, it follows that
		\[
		\frac{\E(\delta_{12}^{(n)}|x_{12}|^2\mathbbm{1}\{\cal{E}_{12}^a\})}{\P(\cal{E}_{12}^a)}\geq \frac34\P(\delta_{12}^{(n)}=1)=\frac34p_n
		\]
		and
		\[
		\abs{\frac{\E(\delta_{12}^{(n)}x_{12}\mathbbm{1}\{\cal{E}_{12}^a\})}{\P(\cal{E}_{12}^a)}}^2\leq \frac12\P(\delta_{21}^{(n)}=1)=\frac12p_n.
		\]
		Dividing by $p_n$, we have that for all $a\geq a_1$, $\var((Z_n)_{12}\given \cal{E}^a_{12})>\frac14$. A similar argument holds for $\var((Z_n)_{21}\given \cal{E}^a_{21})$, for all $a\geq a_2$.
		
		Finally, we deal with the correlation $\corr((Z_n)_{12}\given\cal{E}^a_{12}, (Z_n)_{21}\given\cal{E}^a_{21})$. Unconditionally, we have that for all $n\geq 1$,
		\[
		|\corr((Z_n)_{12},(Z_n)_{21})| = \frac{\E[\delta_{12}^{(n)}\delta_{21}^{(n)}]}{p_n}|\E[x_{12}\overline{x_{21}}]|\leq |\E[x_{12}\overline{x_{21}}]|<1.
		\]
		Examining the proofs from the last two steps, we actually have that uniformly in $n$,
		\begin{align*}
			\E[(Z_n)_{12}\given \cal{E}_{12}^a] & = o(1), & \E[(Z_n)_{21}\given \cal{E}_{12}^a] &= o(1),\\
			\E[|(Z_n)_{12}|^2\given \cal{E}_{12}^a] & = 1+o(1), & \E[|(Z_n)_{21}|^2\given \cal{E}_{12}^a] &= 1+o(1).
		\end{align*}
		where each o(1) is understood as $a\to\infty$. For the conditional mixed moment, we also compute that uniformly in $n$,
		\[
		\E[(Z_n)_{12}\overline{(Z_n)_{21}}\given\cal{E}^a_{12}] = \frac{\E[\delta_{12}^{(n)}\delta_{21}^{(n)}]}{p_n}\E[x_{12}\overline{x_{21}}]+o(1).
		\]
		Bringing this all together, it follows that
		\[
		\corr((Z_n)_{12}\given\cal{E}^a_{12},(Z_n)_{21}\given\cal{E}^a_{21}) = \frac{\frac{\E[\delta_{12}^{(n)}\delta_{21}^{(n)}]}{p_n}\E[x_{12}\overline{x_{21}}]+o(1)}{1+o(1)}
		\]
		With $\eps  := (1-\abs{\E[x_{12}\overline{x_{21}}]})/4$, we may pick an $a_3$ large enough so that for all $a\geq a_3$,
		\[
		|\corr((Z_n)_{12}\given\cal{E}^a_{12},(Z_n)_{21}\given\cal{E}^a_{21})| \leq \frac{\abs{\E[x_{12}\overline{x_{21}}]}+\eps}{1-\eps} =: \rho < 1.
		\]
		Taking $a' = \max\{a_1,a_2,a_3\}$ completes the application of Theorem \ref{thm:andrewsmallsingval}.
		
		Returning to the main proof at hand, we see that with $K=1$, since
		\[
		\P\parens{\sigma_n(Y_n+M)\leq (Cn)^{-4\log^2n-1/2},\sigma_1(Y_n+M)\leq n\sqrt{p_n}}\]
		is equal to, after adjusting $C$,
		\[
		\P\parens{\sigma_n(Z_n+\sqrt{n}M)\leq (Cn)^{-4\log^2n},\sigma_1(Z_n+\sqrt{n}M)\leq n^{3/2}\sqrt{p_n}},
		\]
		it must be that
		\[
		\P\parens{\sigma_n(Y_n+M)\leq (Cn)^{-4\log^2n-1/2},\sigma_1(Y_n+M)\leq n\sqrt{p_n}} = o_{a,b,\sigma,\rho}(1).
		\]
		This, together with \eqref{eq:lsvprob}, gives the desired result of Proposition \ref{prop:leastsingularvalue}.
	\end{proof}

	The rest of the section is dedicated to the proof of Theorem \ref{thm:andrewsmallsingval}.
    
	As in the beginning of the discussion in \cite[Section 5.1]{andrewelliptic}, we will assume that the matrix $Z_n+M$ (appearing in Theorem \ref{thm:andrewsmallsingval}) and each principal submatrix of $Z_n+M$ are invertible with probability 1. To see why we can do this, define $Z_n' = Z_n+t\xi I_n$, where $t>0$ and $\xi$ is a real random variable that is uniformly distributed on $[-1,1]$ and is independent of $Z_n$. Although the common diagonal perturbation introduces dependence among the diagonal entries, we may instead condition on $\xi$ and regard the perturbation as part of $M$. For almost every realization of $\xi$, all principal submatrices of $Z_n'+M = Z_n + M + t\xi I_n$ are invertible with probability 1. Since the subsequent application of Weyl's inequality is uniform over all deterministic matrices $M$, we may apply the argument conditionally on $\xi$ and then integrate out the conditioning. By Weyl's inequality, we have deterministically that
	\[
	\max_{1\leq k \le n}\abs{\sigma_k(Z_n+M)-\sigma_k(Z_n'+M)}\leq t
	\]
	so that for any $s\geq t$,
	\[
	\P\parens{\sigma_n(Z_n+M)\leq t,\sigma_1(Z_n+M)\leq s}\leq \P\parens{\sigma_n(Z_n'+M)\leq2t,\sigma_1(Z_n'+M)\leq 2s}.
	\]
	By taking $t = (Cn)^{-4\log^2n}$ and $s = Cn^{3/2}\sqrt{p_n}$, we conclude that it suffices to proof Theorem \ref{thm:andrewsmallsingval} under the additional assumption that $Z_n+M$ and all of its principal submatrices are invertible with probability 1.
	\subsection{The setup}
	
	To prove Theorem \ref{thm:andrewsmallsingval}, the main idea will be to follow the approach of Rudelson and Vershynin \cite{rudtiksingularity} to deal with compressible and incompressible vectors of the unit sphere $\field S^{n-1}:=\set{v\in\C^n:\norm v =1 }$.
	
	We define, for $a\in(0,1)$, the set of sparse vectors in $\C^n$:
	\[
	\Sparse(a) := \{v\in\C^n:|\supp(v)|\leq an\},
	\]
	where $\supp(v) := \{1\leq i\leq n: v_i\neq0\}$. We recall here that $\abs{\supp(v)}$ denotes the cardinality of the finite set $\supp(v)\subseteq \{1,\ldots,n\}$. To define compressible and incompressible vectors, we introduce another parameter $b\in(0,1)$, which will be a measure of how far a compressible vector $v$ is from the sparse vectors. Namely, $$\Comp(a,b) :=\{v\in\mathbb S^{n-1}:\dist(v,\Sparse(a))\leq b\}$$ and $\Incomp(a,b) := \mathbb S^{n-1}\setminus \Comp(a,b)$.
	Above, we define $\dist(v,\Sparse(a)) = \inf\limits_{u\in\Sparse(a)}\norm{v-u}$, where $\norm\cdot$ is the Euclidean norm on complex vectors.
	
	As will be seen in the proof, we will construct parameters $a_{L_*}$ and $b_{n,L_*}$ and a decomposition of the unit sphere $\field S^{n-1} = \Comp(a_{L_*},b_{n,L_*})\cup\Incomp(a_{L_*},b_{n,L_*})$ so that for any $x\geq 0$,
	\begin{align}
		\P(\sigma_n(Z_n+M)\leq x) &=\nonumber \P\parens{\inf_{v\in\field S^{n-1}}\norm{(Z_n+M)v}\leq x}\\&\label{app:eq:comp}\leq
		\P\parens{\inf_{v\in\Comp(a_{L_*},b_{n,L_*})}\norm{(Z_n+M)v}\leq x}\\&\label{app:eq:incomp}\quad\quad+\P\parens{\inf_{v\in\Incomp(a_{L_*},b_{n,L_*})}\norm{(Z_n+M)v}\leq x}.
	\end{align}
	
	Above and throughout this section, we define 
	\begin{equation}\label{eq:Ldefn}
		L := \ceil{\frac{\log n}{\log(np_n)-2\log\log n-\log 3}}
	\end{equation}
	to be an upper bound for the number of applications of our bootstrapping technique to prove Theorem \ref{thm:andrewsmallsingval}. For $n\gg1$, we will be using that
	\begin{equation}\label{eq:Lineq}
		L\leq \frac{\log n}{\log\log n}\leq\log n.
	\end{equation}
    Although the recursive definitions of $a_l$ and $b_{n,l}$ will be given later, we define a stopping index $L_*$ which we will use to stop the bootstrapping argument used in \cite[Section 4]{gotzecircularlaw}. We define
    \begin{equation}\label{eq:stoppingindexdef}
        L_* := \min\{2\leq l \leq L: np_na_{l-1}p_{n,l-1}\geq 1\}.
    \end{equation}

    \begin{remark}
        For the definition \eqref{app:eq:pnldefn} to come, we only define $p_{n,l}$ for $1\leq l < L_*$. Of course, this definition can be extended to any $l\geq 1$ and this is what is meant by the previous definition \eqref{eq:stoppingindexdef}.

        The set we are taking a minimum over is indeed nonempty for $n\gg1$. For this, we use our growth assumption on $np_n$ coming from Assumption \ref{ass:sparsitylevel} as well as \eqref{app:eq:pnldefn}. 
    \end{remark}
	In Section \ref{app:sec:asymptotics}, we will prove bounds on our parameters $a_{L_*}$ and $b_{n,{L_*}}$ once we are at our final stage after $L_*-1\leq L$ applications of the bootstrapping argument.
	
	The most classic approach, following Rudelson and Vershynin, among others, involves controlling the probabilities over vectors coming from \textit{one} compressible set and \textit{one} incompressible set. Following the methods of \cite[Section 5.3]{andrewelliptic}, we will
	start by defining $a_{1}$ and $b_{n,1}$ as well as controlling the probability over vectors coming from $\Comp(a_{1}p_{n,1},b_{n,1})$; the definition of these $p_{n,l}$ values is given in \eqref{app:eq:pnldefn}. Next, for each $2\leq l\leq L_*$, we will apply a bootstrapping technique as seen in \cite{gotzecircularlaw} to gain control over vectors coming from other compressible levels defined by
	$$\Comp_l :=\Incomp_{l-1}\cap\Comp(a_{l}p_{n,l},b_{n,l}),$$ where for $2\leq l \leq L_*$,
	\[
	\Incomp_{l-1} := \bigcap_{i=1}^{l-1}\Incomp(a_{i}p_{n,i},b_{n,i}).
	\]
	We heuristically think of the iteration $a_lp_{n,l}\mapsto a_{l+1}p_{n,l+1}$ as controlling the probabilities arising from vectors that are less sparse. In addition, the iteration of $b_{n,l}\mapsto b_{n,l+1}$ will decrease the distance to the corresponding set of less sparse vectors. Throughout this process, we will iteratively define the values of $a_{l}$ and $b_{n,l}$ given the values $a_{i}$ and $b_{n,i}$ already defined for $i<l$. The main results we will be proving in this section involve controlling our collection of compressible vectors coming from $\bigcup_{l=1}^{L_*}\Comp_l$ and the final collection of incompressible vectors, $\Incomp_{L_*}$.
	
	For the compressible vectors, we have the following lemmas. Before their statements, we define $K_n := Kn^{3/2}\sqrt{p_n}$, where $K$ is as in the statement of Theorem \ref{thm:andrewsmallsingval}. We first control vectors coming from $\Comp_1 = \Comp(a_1p_{n,1},b_{n,1})$, as well as define the quantities $a_1$ and $b_{n,1}$.
	
	\begin{lem}[Control of compressible vectors at the first stage] \label{app:lem:beginningcomp}
		There exist absolute constants $C_{\ref{app:lem:beginningcomp}}, c_{\ref{app:lem:beginningcomp}}> 0$ and values of $a_1$ and $b_{n,1}$ such that for any deterministic $M\in\cal{M}_n(\C)$ and any deterministic vector $u\in \C^n$,
		\begin{equation}\label{app:eq:firstcompcontrol}
			\P \left( \inf_{v \in \Comp_1} \|(Z_n+M)v-u \| \leq \frac{\eps_{\ref{lem:lsv:distance}}}{2 \sqrt{a_1}}, \sigma_1(Z_n+M) \leq K_n \right) \leq C_{\ref{app:lem:beginningcomp}}\exp(-c_{\ref{app:lem:beginningcomp}}np_n).
		\end{equation}
		The quantity $\eps_{\ref{lem:lsv:distance}}$ is coming from Lemma \ref{lem:lsv:distance}.
	\end{lem}
	
	For the other sets of compressible vectors, we have an analogous lemma.
	
	\begin{lem}\label{app:lem:bootstrapping2}
		There exists a constant $\delta_{\ref{app:lem:bootstrapping2}}>0$ depending only on $K$ and $\widetilde{c}_{\ref{app:lem:bootstrap}}$ such that the following holds. For each $l = 2,\ldots,L_*$, suppose the values of $a_i$ and  $b_{n,i}$ have been defined for $i<l$.  For any deterministic vector $w\in \C^n$,
		\[
		\P\left(\inf_{v\in\Comp_l}\norm{(Z_n+M)v-w}\leq\gamma_{n,l}/4\,\,,\,\sigma_1\parens{Z_n+M}\leq K_n\right)\leq
		\exp(-\tfrac{\widetilde{c}_{\ref{app:lem:bootstrap}}}{16}n\min(np_na_{l-1}p_{n,l-1}, 1)).
		\]
		Here, we define $a_l = \delta_{\ref{app:lem:bootstrapping2}} a_{l-1}$ (only when $l<L_*$),
		\[
		\gamma_{n,l} := \overline{c}_{\ref{app:lem:bootstrap}}\sqrt{n}b_{n,l-1}\sqrt{a_{l-1}p_{n,l-1}},
		\]
		and
		\[
		b_{n,l} := \frac{\gamma_{n,l}}{4K_n}.
		\]
		Recall that $\Comp_l = \Incomp_{l-1}\cap\Comp(a_lp_{n,l},b_{n,l})$.
	\end{lem}
	
	For the incompressible vectors at the final stage, we have the following proposition.
	
	\begin{prop}\label{app:prop:userfriendlyincompconrtol}
		There exists an absolute constant $C_{\ref{app:prop:userfriendlyincompconrtol}}>0$ such that for any $n\geq C_{\ref{app:prop:userfriendlyincompconrtol}}$ and any $t\geq 0 $,
		\[
		\P\parens{\inf_{v\in \Incomp(a_{L_*},b_{n,L_*})}\norm{(Z_n+M)v}\leq t,\sigma_1(Z_n+M)\leq K_n}
		\]
		is no more than
		\[
		C_{\ref{app:prop:userfriendlyincompconrtol}}\parens{(C_{\ref{app:prop:userfriendlyincompconrtol}}n)^{\log ^2n}\sqrt{t}+o(1)}^{1/4},
		\]
		where the $o(1)$ term does not depend on $t$.
	\end{prop}


	Using Lemma \ref{app:lem:beginningcomp}, Lemma \ref{app:lem:bootstrapping2}, and Proposition \ref{app:prop:userfriendlyincompconrtol}, we give a proof of Theorem \ref{thm:andrewsmallsingval}.
	
	\begin{proof}[Proof of Theorem \ref{thm:andrewsmallsingval}]
		With the decomposition
		$
		\field S^{n-1} = \Comp(a_{L_*},b_{n,L_*})\cup\Incomp(a_{L_*},b_{n,L_*})
		$
		of the unit sphere in $\C^n$, we have that for any $t\geq 0$,
		\begin{align}
			\nonumber&\P(\sigma_n(Z_n+M)\leq t,\sigma_1(Z_n+M)\leq K_n)\\\label{app:eq:finalproof1}&\qquad\leq
			\P\parens{\inf_{v\in\Comp(a_{L_*},b_{n,L_*})}\norm{(Z_n+M)v}\leq t, \sigma_1(Z_n+M)\leq K_n}\\\label{app:eq:finalproof2}&\qquad+
			\P\parens{\inf_{v\in\Incomp(a_{L_*},b_{n,L_*})}\norm{(Z_n+M)v}\leq t, \sigma_1(Z_n+M)\leq K_n}.
		\end{align}
		
		By our decomposition $\Comp(a_{L_*},b_{n,L_*})\subseteq \bigcup_{l=1}^{L_*}\Comp_l$, Lemma \ref{app:lem:beginningcomp}, and Lemma \ref{app:lem:bootstrapping2}, we may bound \eqref{app:eq:finalproof1} by
		\begin{align}
			\nonumber&\sum_{l=1}^{L_*}\P\parens{\inf_{v\in \Comp_l}\norm{(Z_n+M)v}\leq t,\sigma_1(Z_n+M)\leq K_n}\\\nonumber&\qquad\leq
			C_{\ref{app:lem:beginningcomp}}\exp(-c_{\ref{app:lem:beginningcomp}}np_n)+\sum_{l=2}^{L_*}\exp(-\tfrac{\widetilde{c}_{\ref{app:lem:bootstrap}}}{16}n\min(np_na_{l-1}p_{n,l-1}, 1))\\
			\label{app:eq:finalproof3}&\qquad= O(\log n\exp(-c_0np_n))\\
			&\qquad\nonumber=o(1),
		\end{align}
		as long as $t\leq \gamma_{n,L_*}/4$, where $c_0 = \min(\tfrac{\widetilde{c}_{\ref{app:lem:bootstrap}}}{16},c_{\ref{app:lem:beginningcomp}})$. The $O(\log n\exp(-c_0np_n))$ bound is coming from one viable bound of $L_*\leq L$ and Assumption \ref{ass:sparsitylevel}, which implies that $\min(np_na_{l-1}p_{n,l-1}, 1)\geq p_n$ for $n\gg1$. The values for $\gamma_{n,l}$, $2\leq l\leq L_*$, are also defined iteratively via Lemma \ref{app:lem:bootstrapping2}. Finally, to bound \eqref{app:eq:finalproof2}, we apply Proposition \ref{app:prop:userfriendlyincompconrtol} for an upper bound of
		\begin{equation}\label{app:eq:finalproof4}
			C_{\ref{app:prop:userfriendlyincompconrtol}}\parens{(C_{\ref{app:prop:userfriendlyincompconrtol}}n)^{\log ^2n}\sqrt{t}+o(1)}^{1/4},
		\end{equation}
		which holds for any $t\geq 0$. Combining \eqref{app:eq:finalproof3} and \eqref{app:eq:finalproof4}, after plugging in
		\[
		t = (C_{\ref{app:lem:asymptotics}}n)^{-4\log^2n}\leq b_{n,L_*}\leq \gamma_{n,L_*}/4,
		\]
		we arrive at the desired overall $o(1)$ decay in the result of the Theorem \ref{thm:andrewsmallsingval}. The first inequality above is coming from Lemma \ref{app:lem:asymptotics}.
	\end{proof}
	
	Throughout this section, we define for any $x\in(0,1]$,
	\begin{equation}\label{eq:hat}
		\widehat{x} = \frac{x}{\log(2/p_n)\log(K_n)}.
	\end{equation}
	Throughout, we will be taking $n$ large enough so that $\log(K_n)\gg1$ and so since $p_n\leq 1$ as well, we have that $\widehat{x}\leq x$ for any $x$. Finally, we define the values of ${p}_{n,l}$ for $1\leq l< L_*$:
	\begin{equation}\label{app:eq:pnldefn}
		p_{n,l}= (n\widehat{p}_n)^{l-1}\widehat{p}_n = \frac{(n\widehat{p}_n)^l}{n},
	\end{equation}
	where $\widehat{p}_n$ is as in \eqref{eq:hat}. For the end of the bootstrapping argument, we further define
    \begin{equation}\label{eq:endbootstrappingdefs}
        p_{n,L_*}:=1 \quad\text{and}\quad a_{L_*} := \widehat{\delta_{\ref{app:lem:bootstrapping2}}},
    \end{equation}
    where $\delta_{\ref{app:lem:bootstrapping2}}$ will be chosen in Lemma \ref{app:lem:bootstrapping2}.
	Throughout this section, we are taking the inherent assumption that $n\gg1$ (larger than some absolute constant) even if we do not state it in every result that follows. In the next two subsections, we prove Lemma \ref{app:lem:beginningcomp}, Lemma \ref{app:lem:bootstrapping2}, and Proposition \ref{app:prop:userfriendlyincompconrtol} used in the proof of Theorem \ref{thm:andrewsmallsingval}.
	
	\subsection{Proof of Lemma \ref{app:lem:beginningcomp} and Lemma \ref{app:lem:bootstrapping2}}\label{app:sec:comp}
	
	We begin with the proof of Lemma \ref{app:lem:beginningcomp}, controlling the probability over vectors coming from $\Comp_1:=\Comp(a_{1}{p}_{n,1},b_{n,1})$ as well as defining these values of $a_{1}$ and $b_{n,1}$. Truncated proofs of results involved in the proof of this lemma can be found in Appendix \ref{sec:app:firststagecompressibles}, as they are standard.

	\begin{proof}[Proof of Lemma \ref{app:lem:beginningcomp}]
		Fix any $0<a_1\leq \delta_{\ref{lem:lsv:distance}}$, where $\delta_{\ref{lem:lsv:distance}}$ is coming from Lemma \ref{lem:lsv:distance}. We will choose a value for $a_1$ in this proof. We start by examining the sparse vectors. Following \cite[Lemma 5.6]{andrewelliptic}, we decompose
		\[
		\Comp(a_1\widehat{p}_n,b_{n,1}) = \bigcup_{\pi\subseteq\{1,\ldots,n\},\abs{\pi}=\floor{a_1 n\widehat{p}_n}} S_\pi,
		\]
		where 
		\[
		S_{\pi} = \set{x\in\Comp(a_1\widehat{p}_n,b_{n,1}):\dist(x,\Sparse_\pi(a_1\widehat{p}_n))\leq b_{n,1}}
		\]
		(for a value of $b_{n,1}$ to be defined below) and
		\[
		\Sparse_{\pi}(a_1\widehat{p}_n) = \set{y\in\Sparse(a_1\widehat{p}_n):\supp(y)\subseteq\pi}.
		\]
		Recall that $p_{n,1}=\widehat{p}_n$. Therefore, we may upper bound the probability in  \eqref{app:eq:firstcompcontrol} by
		\begin{equation}\label{app:eq:compcontrol2}
			\sum_{\pi\subseteq\{1,\ldots,n\},\abs{\pi}=\floor{a_1 n\widehat{p}_n}}\P \left( \inf_{v \in S_\pi} \|(Z_n+M)v - u \| \leq \frac{\eps_{\ref{lem:lsv:distance}}}{2 \sqrt{a_1}}, \sigma_1(Z_n+M) \leq K_n\right),
		\end{equation}
		
		where $\eps_{\ref{lem:lsv:distance}}$ is coming from Lemma \ref{lem:lsv:distance}. Now, fix such a $\pi\subseteq \{1,\ldots,n\}$ with $\abs{\pi}=\floor{a_1 n\widehat{p}_n}$. 
		
		Suppose that there exists some $v\in S_{\pi}$ such that $\norm{(Z_n+M)v - u}\leq \frac{\eps_{\ref{lem:lsv:distance}}}{2\sqrt{a_1}}$ and $\sigma_1(Z_n+M)\leq K_n$. Since $v\in S_\pi$, there exists some $y\in \Sparse_\pi(a_1\widehat{p}_n)$ such that $\norm{v-y}\leq b_{n,1}$. Therefore,
		\begin{align*}
			\frac{\eps_{\ref{lem:lsv:distance}}}{2\sqrt{a_1}}\geq\norm{(Z_n+M)v-u}&\geq\norm{(Z_n+M)y-u}-\sigma_1(Z_n+M)\norm{v-y}\\&\geq\norm{(Z_n+M)y-u}- K_nb_{n,1}.
		\end{align*}
		Taking $b_{n,1}:= \min\parens{\frac{\eps_{\ref{lem:lsv:distance}}}{4\sqrt{a_1} K_n},\frac14}$ gives $\norm{(Z_n+M)y-u}\leq \frac{3\eps_{\ref{lem:lsv:distance}}}{4\sqrt{a_1}}$ and also that $\norm{y}\geq\frac{3}{4}$. We now further bound this $\norm{(Z_n+M)y-u}$ from below:
		\[
		\norm{(Z_n+M)y-u}^2 = \norm{\sum_{i\in\pi}y_i\Col_i -u}^2\geq \max_{i\in\pi}\abs{y_i}^2\dist^2(\Col_i,H_i),
		\]
		where $H_i = \Span(\{\Col_j:j\in\pi\setminus\{i\}\}\cup\{u\})$. Continuing, we have
		\begin{align*}
			&\max_{i\in\pi}\abs{y_i}^2\dist^2(\Col_i,H_i)\geq
			\frac{1}{\abs{\pi}}\sum_{i\in\pi}\abs{y_i}^2\cdot\min_{i\in\pi}\dist^2(\Col_i,H_i)\\&\geq
			\frac{1}{a_1n\widehat{p}_n}\left(\frac{3}{4}\right)^2\min_{i\in\pi}\dist^2(\Col_i,H_i)
			\geq
			\frac{1}{a_1n}\left(\frac{3}{4}\right)^2\min_{i\in\pi}\dist^2(\Col_i,H_i).
		\end{align*}
		Summarizing, we have shown that if there exists a $v\in S_\pi$ such that $\norm{(Z_n+M)v-u}\leq \frac{\eps_{\ref{lem:lsv:distance}}}{2\sqrt{a_1}}$ and $\sigma_1(Z_n+M)\leq K_n$, then $\min\limits_{i\in\pi}\dist(\Col_i,H_i)\leq\eps_{\ref{lem:lsv:distance}}\sqrt{n}$ and so
		\[
		\P \left( \inf_{v \in S_\pi} \norm{(Z_n+M)v -u} \leq \frac{\eps_{\ref{lem:lsv:distance}}}{2 \sqrt{a_1}}, \sigma_1(Z_n+M) \leq K_n\right)
		\leq
		\P\parens{\min\limits_{i\in\pi}\dist(\Col_i,H_i)\leq\eps_{\ref{lem:lsv:distance}}\sqrt{n}}.
		\]
		By Corollary \ref{app:cor:mindistance}, it follows that
		\begin{align*}
			\P \left( \inf_{v \in S_\pi} \norm{(Z_n+M)v-u} \leq \frac{\eps_{\ref{lem:lsv:distance}}}{2 \sqrt{a_1}}, \sigma_1(Z_n+M) \leq K_n\right) 
			&\leq C_{\ref{lem:lsv:distance}}np_n\exp(-c_{\ref{lem:lsv:distance}}np_n). 
		\end{align*}
		We may now upper bound \eqref{app:eq:compcontrol2} by
		\begin{equation}\label{app:eq:unionbound}
			\binom{n}{\floor{a_1 n \widehat{p}_n}}\cdot C_{\ref{lem:lsv:distance}}np_n\exp(-c_{\ref{lem:lsv:distance}}np_n).
		\end{equation}
		By Stirling's formula, for some absolute constant $c'>0$,
		\[
		\frac{1}{n}\log\binom{n}{\floor{a_1n\widehat{p}_n}}\leq -c'a_1\widehat{p}_n\log(a_1\widehat{p}_n).
		\]
		Further simplifying this right hand side and choosing $a_1>0$ small enough (depending only on $c_{\ref{lem:lsv:distance}}$) yields an absolute constant $c_{\ref{lem:lsv:distance}}>c''>0$ such that
		\[
		\frac{1}{n}\log \binom{n}{\floor{a_1n\widehat{p}_n}}\leq c''p_n
		\]
		for all $n\gg1 $. Returning to \eqref{app:eq:unionbound}, we have a final upper bound of \eqref{app:eq:compcontrol2} (and hence of \eqref{app:eq:firstcompcontrol} as well) of the form $C_{\ref{app:lem:beginningcomp}}\exp(-c_{\ref{app:lem:beginningcomp}}np_n)$ for all $n\gg1$, after absorbing the $np_n$ outside the exponential.
	\end{proof}
	
	Recall that in addition to controlling vectors coming from $\Comp(a_{1}p_{n,1},b_{n,1})$, we have chosen the initial value of 
	$a_1$ 
	and defined $b_{n,1} = \min\parens{\frac{\eps_{\ref{lem:lsv:distance}}}{4\sqrt{a_1}K_n},\frac14}$.

	\subsubsection{Proof of Lemma \ref{app:lem:bootstrapping2}}\label{sec:bootstrap}
	Lemma \ref{app:lem:beginningcomp} furnishes a starting point to control a subset of compressible vectors of the unit sphere $\mathbb{S}^{n-1}$. However, the infimum over vectors coming from $\Comp(a_{1}{p}_{n,1},b_{n,1})$ guaranteed in Lemma \ref{app:lem:beginningcomp} is not sufficient. To facilitate controlling more vectors of the unit sphere, we will turn to $\Incomp_1 :=\Incomp(a_{1}{p}_{n,1},b_{n,1}) = \field S^{n-1}\setminus \Comp(a_{1}{p}_{n,1},b_{n,1})$. Within this incompressible set of vectors $\Incomp_1$, we will isolate a compressible subset that we will further be able to control:
	\[
	\Comp_2 =  \Incomp_1 \cap \Comp(a_{2}p_{n,2},b_{n,2}).
	\]
	Heuristically, we think of the vectors coming from $\Comp_2$ as ``more incompressible" and ``less sparse" in comparison to those coming from $\Comp_1$. We will continue this iteration until we arrive at our final incompressible set $\Incomp_{L_*}$ over which we have the strongest probability bounds. See Figure \ref{fig:bootstrapping} for a visualization of this accumulation.
	
	\begin{figure}
		\centering
		\includegraphics[width=\linewidth]{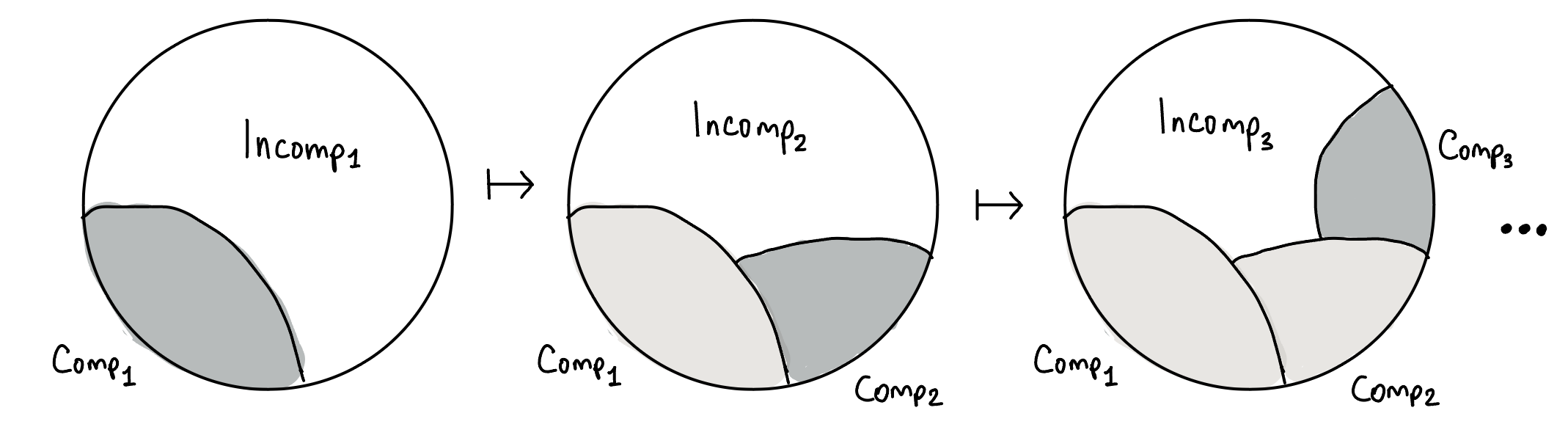}
		\caption{Accumulating compressible subsets of $\mathbb{S}^{n-1} = \Incomp_1\cup\Comp_1$.}\label{fig:bootstrapping}
	\end{figure}
	
	We will be defining these parameters $a_{2}$ and $b_{n,2}$ (as well as later values of $a_l$ and $b_{n,l}$) iteratively. Continuing, 
	this section will give control over the remaining $\Comp_l$ for $2\leq l \leq L_*$. 

	Lemma \ref{app:lem:infcomp} below allows us to upgrade a pointwise bound over vectors coming from some subset $I\subseteq \mathbb{S}^{n-1}$ to one involving an infimum over the intersection of $I$ and a compressible subset. This lemma is critical as our compressible levels $\Comp_l = \Incomp_{l-1}\cap\Comp(a_{l}p_{n,l},b_{n,l})$ are subsets of this form.
	
	\begin{lem}\label{app:lem:infcomp} Assume there exist an
		absolute constant $c>0$, a deterministic $w\in \C^n$, and values $\gamma\in(0,1)$ and $q\in(0,1]$ such that
		for any $v\in I\subseteq \field S^{n-1}$,
		\begin{equation}\label{app:eq:pointwiseboundonI}
			\P\parens{\norm{(Z_n+M)v-w}\leq\gamma,\sigma_1(Z_n+M)\leq K_n}\leq\exp(
			-c nq).
		\end{equation}
		If $\gamma$ satisfies
		\begin{equation}\label{eq:gammahypothesis}
			\frac{\log(K_n/\gamma)}{\log(2/p_n)\log(K_n)}\leq C
		\end{equation}
		for some absolute constant $C>0$, then there exists $\delta_{\ref{app:lem:infcomp}}>0$ depending only on $K$,  $c$, and $C$ such that
		\begin{equation}\label{app:eq:infcomp1}
			\P\parens{\inf_{v\in I\cap \Comp(\delta_{\ref{app:lem:infcomp}}\widehat
					{q}, b_{n})}\norm{(Z_n+M)v-w}\leq\gamma/4 , \sigma_1\parens{Z_n+M}\leq
				K_n} \leq\exp(-cn q/4),
		\end{equation}
		where $b_n:=\gamma/(4K_n)$.
	\end{lem}
	
	
	The proof of this lemma may be found in Appendix \ref{app:sec:bootstrapnew} as it  closely follows the work done in \cite[Proposition 4.6]{gotzecircularlaw}, even with the dependence inherent in our elliptic random matrix model. In order to apply Lemma \ref{app:lem:infcomp}, we need a pointwise bound on vectors coming from these incompressible subsets of $\field S^{n-1}$. Since this proof is inspired by Lemma 4.6 in \cite{gotzecircularlaw}, we include a restatement of this result that remains true at our sparsity level.
	
	\begin{lem}\label{lem:gotzetikhincomp}
		Let $a_n,b_n\in(0,1)$, let $v\in \Incomp(a_n,b_n)$, and let $A_n\in\cal{M}_n(\C)$ be a sparse i.i.d.\ matrix associated to the model in Theorem \ref{thm:sparsecircularlaw} with complex-valued atom variable $\xi$, normalized in such a way that its entries have unit variance. In other words, the entries of $A_n$ are given by $(A_n)_{ij}=p_{n}^{-1/2}\delta_{ij}^{(n)}\xi_{ij}$. Then there exist positive constants $c_1$ and $c_2$ depending on $K$ and $\xi$ such that for any deterministic matrix $M\in\cal{M}_n(\C)$, any deterministic vector $w\in\C^n$, and any $0<\tau\leq c_2\sqrt{n}b_n\sqrt a_n$,
		\[
		\P(\norm{(A_n+M)v+w}\leq\tau,\sigma_1(A_n+M)\leq K_n )\leq\exp(-c_1n\min(np_na_n,1)).
		\]
	\end{lem}
	
	\begin{remark}
		The matrix $A_n$ in this Lemma \ref{lem:gotzetikhincomp} is a sparse i.i.d.\ matrix where we only need to assume that $np_n\to\infty$ as $n\to\infty$.
	\end{remark}
	
	\begin{proof}[Proof of Lemma \ref{lem:gotzetikhincomp}]
		Examining the proofs of Lemma 4.6 coming from \cite{gotzecircularlaw} (as well as Lemma 4.4 and Lemma 4.5 in \cite{gotzecircularlaw}), the stronger sparsity assumption of $p_n^{-1} = O(n^{1-\theta})$ for some fixed $0< \theta \leq 1$ is not used.
        Furthermore, to match the scaling used in \cite[Lemma 4.4, Lemma 4.5]{gotzecircularlaw}, we will apply these lemmas with the matrices and vectors $\widetilde{A_n}=A_n/\sqrt{n}$, $\widetilde{M}=M/\sqrt{n}$ and $\widetilde{w} = w/\sqrt{n}$.
        Therefore, the result of Lemma \ref{lem:gotzetikhincomp} follows from \cite[Lemma 4.6]{gotzecircularlaw} under the weaker sparsity assumptions of Theorem \ref{thm:sparsecircularlaw}.
	\end{proof}
	
	Extending Lemma \ref{lem:gotzetikhincomp} to the elliptic case, we have our analog of Lemma \ref{lem:gotzetikhincomp}.
	
	\begin{lem}\label{app:lem:bootstrap}
		Let $a_n,b_n\in(0,1)$ where $na_n\geq 1$ and $v\in\Incomp(a_n,b_n)$. Then there exist positive constants $\widetilde{c}_{\ref{app:lem:bootstrap}},\overline{c}_{\ref{app:lem:bootstrap}}$ depending on $K$ and the atom variable $(\xi_1,\xi_2)$ such that for any $\gamma\leq \overline{c}_{\ref{app:lem:bootstrap}}\sqrt{n}b_n\sqrt{a_n}$ and any deterministic $w\in \C^n$,
		\[
		\P(\norm{(Z_n+M)v-w}\leq \gamma, \sigma_1(Z_n+M)\leq K_n)\leq
		\exp(-\widetilde{c}_{\ref{app:lem:bootstrap}}n\min(np_na_n, 1)).
		\]
	\end{lem}
	\begin{proof}
		We will apply Lemma \ref{lem:gotzetikhincomp} (and thus also the techniques of Lemma 4.4 and Lemma 4.5 from \cite{gotzecircularlaw}).
		
		We start with the simpler case of $n=2k$ being even. We decompose
		\[
		(Z_n+M)v -w= \begin{pmatrix}
			A & B\\C & D
		\end{pmatrix}
		\begin{pmatrix}
			v_1\\v_2
		\end{pmatrix}-\begin{pmatrix}
			w_1\\w_2
		\end{pmatrix},
		\]
		where $A,B,C,D\in\cal{M}_k(\C)$ and $v_i,w_i\in \C^k$ for $i=1,2$. 
		Since $v\in \Incomp(a_n,b_n)$, we have that for at least one $i\in\{1,2\}$, $\frac{v_i}{\norm{v_i}}\in\Incomp(a_n,\frac{b_n}{\sqrt{2}\norm{v_i}})$. Without loss of generality, choose $i=1$. 
		Since
		\[
		\norm{(Z_n+M)v-w}^2=\norm{Av_1+Bv_2-w_1}^2+\norm{Cv_1+Dv_2-w_2}^2\geq\norm{Cv_1+Dv_2-w_2}^2,
		\]
		we turn to the event $\norm{Cv_1+Dv_2-w_2}\leq\gamma$. From the structure of our elliptic matrix $Z_n$, we have that $C$ is an i.i.d.\ matrix (plus a deterministic summand from $M$) and $D$ is a matrix independent of $C$ (also containing a deterministic summand from $M$). Therefore, we may condition on $D$, apply Lemma \ref{lem:gotzetikhincomp} and integrate out the conditioning. In other words,
		\[
		\P(\norm{Cv_1+Dv_2-w_2}\leq\gamma,\sigma_1(C)\leq K_n)\leq \E_D\P_C(\norm{Cv_1+Dv_2-w_2}\leq\gamma,\sigma_1(C)\leq K_n),
		\]
		where $\P_C$ is the probability taken with respect to the entries of $C$ after conditioning on the entries in $D$ and $\E_D$ is expectation taken with respect to the entries in $D$. To be explicit, we rewrite
		\[
		Cv_1 +Dv_2 -w_2= \norm{v_1}\parens{(Z_k+M_C)\frac{v_1}{\norm{v_1}} + \frac{u}{\norm{v_1}}},
		\]
		where $u$ is a deterministic vector that is made up of $Dv_2$ (which is entirely deterministic after conditioning on $D$) and $w_2$. 
		Also, $Z_k$ here is actually the lower-left $k\times k$ i.i.d.\ block of our $n\times n$ matrix $Z_n$ and $M_C$ is the deterministic summand coming from $C$. From Lemma \ref{lem:gotzetikhincomp}, we get positive constants $c_1$ and $c_2$ depending on $K$ and the atom variable $(\xi_1,\xi_2)$ such that for any $t\leq c_2\sqrt{k}\frac{b_n}{\sqrt{2}\norm{v_1}}\sqrt{a_n}$,
		\begin{align*}
			\P_C\left(\norm{(Z_k+M_C)\frac{v_1}{\norm{v_1}} + \frac{u}{\norm{v_1}}}\leq t,\sigma_1(Z_k+M_C)\leq K_n\right)&\leq \exp(-c_1k\min(kp_na_n,1))\\&\leq 
			\exp(-\tfrac{c_1}{4}n\min(np_na_n,1)).
		\end{align*}
		Applying this line with $t = \gamma/\norm{v_1}$, decreasing the constant $\overline{c}_{\ref{app:lem:bootstrap}}$ by an absolute factor, and integrating out the conditioning with $\E_D$ gives our desired result in the case when $n=2k$ is even. Note here that we also need the fact that the largest singular value of a submatrix is bounded by that of the entire matrix.

		The case that $n=2k+1$ is odd is a bit more difficult to deal with, in that an uneven block partition will have parts that do not work as nicely with the elliptic structure of our matrix. To get around this, we first consider the partition $v = \begin{pmatrix}
			v_1,v_2,v_3
		\end{pmatrix}^T$, where $v_1,v_3\in\C^k$ and $v_2\in\C$. Since $v\in\Incomp(a_n,b_n)$, we can choose an $i \in \{1,3\}$ such that $v_i\neq0$ and $\frac{v_i}{\norm{v_i}}\in \Incomp(\frac{a_n}{2},\frac{b_n}{\sqrt{2}\norm{v_i}})$. Here, we use the fact that $na_n\geq1$ to rule out the mass of $v$ being concentrated on $v_2$. As above, without loss of generality, we choose $i=1$.
		
		Then, we set up the block decomposition
		\[
		(Z_n+M)v - w = \begin{pmatrix}
			A & u_1& B\\u_2^T  &\alpha&u_3^T\\C &u_4& D
		\end{pmatrix}
		\begin{pmatrix}
			v_1 \\ v_2 \\v_3
		\end{pmatrix} -\begin{pmatrix}
			w_1\\w_2\\w_3
		\end{pmatrix}= \begin{pmatrix}
			Av_1+v_2u_1+Bv_3-w_1\\u_2^Tv_1+\alpha v_2+u_3^Tv_3-w_2\\Cv_1+v_2u_4+Dv_3-w_3
		\end{pmatrix},
		\]
		where $A,B,C,D\in\cal{M}_k(\C)$, $u_1,u_2,u_3,u_4, v_1,v_3,w_1,w_3\in\C^k$, and $\alpha,v_2,w_2\in\C$.  As above,
		\[
		\norm{(Z_n+M)v-w}^2\geq \norm{Cv_1+v_2u_4+Dv_3-w_3}^2.
		\]
		Given the structure of our elliptic matrix $Z_n$, we have that $C$ is an i.i.d.\ matrix (plus a deterministic summand) with $C$, $D$, and $u_4$ jointly independent. Note that $D$ and $u_4$ each contain deterministic summands as well, coming from $M$.
		
		In the event that $\norm{(Z_n+M)v-w}\leq\gamma$, we have that $\norm{Cv_1+v_2u_4+Dv_3-w_3}\leq \gamma$ as well. After conditioning on the entries of $D$ and $u_4$ and using the independence of $C$, $D$, and $u_4$, 
		\begin{align*}
			&\P(\norm{Cv_1+v_2u_4+Dv_3-w_3}\leq \gamma,\sigma_1(C)\leq K_n)\\&\qquad\qquad\qquad\qquad\leq \E_{D,u_4}\P_C(\norm{Cv_1+v_2u_4+Dv_3-w_3}\leq \gamma,\sigma_1(C)\leq K_n),
		\end{align*}
		where $\E_{D,u_4}$ is expectation with respect to the random variables in $D$ and $u_4$ and $\P_C$ is the probability with respect to the entries of $C$ after conditioning on $D$ and $u_4$. Again, we note here that
		\[
		Cv_1+v_2u_4+Dv_3-w_3 = \norm{v_1}\parens{(Z_k+M_C)\frac{v_1}{\norm{v_1}}+\frac{u}{\norm{v_1}}},
		\]
		where $u\in\C^k$ is a deterministic vector made up of $Dv_3$ and $v_2u_4$---which are fixed by the conditioning---all of the deterministic contributions from $M$, and $w_3$. As above, $Z_k$ is the lower-left $k\times k$ block of the sparse elliptic random matrix $Z_n$ and $M_C$ is the deterministic summand in $C$. From Lemma \ref{lem:gotzetikhincomp}, it follows that there exist some positive constants $c_1'$ and $c_2'$ depending only on $K$ and $(\xi_1,\xi_2)$ such that for any $t\leq c_2'\sqrt{k}\frac{b_n}{2\norm{v_1}}\sqrt{a_n}$
		\begin{align*}
			\P_C\parens{\norm{(Z_k+M_C)\frac{v_1}{\norm{v_1}}+\frac{u}{\norm{v_1}}}\leq t,\sigma_1(Z_k+M_C)\leq K_n}&\leq \exp(-c_1'k\min(kp_na_n/2, 1))\\ &\leq
			\exp(-\tfrac{c_1'}{18}n\min(np_n a_n, 1)).
		\end{align*}
		As above, applying this result with $t = \gamma/\norm{v_1}$ and integrating out the conditioning in this bound that is uniform over realizations of $D$ and $u_4$, the lemma follows when $n=2k+1$ is odd. We select the constants $\widetilde{c}_{\ref{app:lem:bootstrap}} := \min(\tfrac{c_1}{4},\tfrac{c_1'}{18})$ and $\overline{c}_{\ref{app:lem:bootstrap}} = \min(\tfrac{c_2}{2},\tfrac{c_2'}{2\sqrt{3}})$ to bring the even and odd cases together and complete the proof.
	\end{proof}
	
	We finish this subsection with a proof of Lemma \ref{app:lem:bootstrapping2}.
	
	\begin{proof}[Proof of Lemma \ref{app:lem:bootstrapping2}]
		Fix an $l\in\{2,\ldots,L_*\}$ and suppose the values of $a_i$ and  $b_{n,i}$ have been determined for $i<l$. To invoke Lemma \ref{app:lem:infcomp} and take an infimum over the set $$\Comp_l = \Incomp_{l-1}\cap \Comp(a_{l}p_{n,l},b_{n,l})$$ (for these values of  $a_l$ and $b_{n,l}$ to be defined), we must show pointwise control for vectors coming from $\Incomp_{l-1}\subseteq\Incomp(a_{l-1}p_{n,l-1},b_{n,l-1})$. By Lemma \ref{app:lem:bootstrap} with
		$a_n = a_{l-1}p_{n,l-1}$ and $b_n = b_{n,l-1}$, we have that
		\begin{align*}
			&\P\parens{\norm{(Z_n+M)v-w}\leq \overline{c}_{\ref{app:lem:bootstrap}}\sqrt{n}b_{n,l-1}\sqrt{a_{l-1}p_{n,l-1}} ,\sigma_1(Z_n+M)\leq K_n}\\&\qquad\qquad\leq
			\exp(-\tfrac{\widetilde{c}_{\ref{app:lem:bootstrap}}}{4}n\min(np_na_{l-1}p_{n,l-1}, 1))
		\end{align*}
		for all vectors $v \in \Incomp_{l-1}$. We note that our Assumption \ref{ass:sparsitylevel} is used to confirm that $na_n\geq1$ in this application of Lemma \ref{app:lem:bootstrap}. Therefore, by Lemma \ref{app:lem:infcomp} with $q=q_n := \min(np_na_{l-1}p_{n,l-1}, 1)$ and
		$\gamma =\gamma_{n,l}:= \overline{c}_{\ref{app:lem:bootstrap}}\sqrt{n}b_{n,l-1}\sqrt{a_{l-1}p_{n,l-1}}$, there exists a $\delta_{\ref{app:lem:bootstrapping2}} = \delta_{\ref{app:lem:infcomp}}(K,\widetilde{c}_{\ref{app:lem:bootstrap}},C)$ (this $C$ is given in \eqref{eq:gammahypothesis}) such that
		\begin{align*}
			&\P\parens{\inf_{ v\in \Incomp_{l-1}\cap \Comp(\delta_{\ref{app:lem:bootstrapping2}}\widehat
					{q}_n, b_{n,l})}\norm{(Z_n+M)v-w}\leq\gamma_{n,l}/4 , \sigma_1\parens{Z_n+M}\leq
				K_n} \\&\qquad\qquad\leq\exp(-\tfrac{\widetilde{c}_{\ref{app:lem:bootstrap}}}{16}n\min(np_na_{l-1}p_{n,l-1}, 1)),
		\end{align*}
		where $b_{n,l} := \frac{\gamma_{n,l}}{4K_n}$. Note here that by Lemma \ref{app:lem:asymptotics}, $\gamma = \gamma_{n,l}$ does satisfy \eqref{eq:gammahypothesis}. In fact, by \eqref{eq:asymptotics1}, we may apply Lemma \ref{app:lem:infcomp} uniformly as $\gamma = \gamma_{n,l}$ varies with $l$. By the definition of $q_n$, 
        \begin{equation}\label{eq:compressiblesetcheck}
            \delta_{\ref{app:lem:bootstrapping2}}\widehat{q}_n = 
            \begin{cases}
                \delta_{\ref{app:lem:bootstrapping2}}a_{l-1}p_{n,l}&l<L_*\\
                \widehat{\delta_{\ref{app:lem:bootstrapping2}}}&l=L_*
            \end{cases} = a_{l}p_{n,l},
        \end{equation} where we define $a_l := \delta_{\ref{app:lem:bootstrapping2}} a_{l-1}$ when $l<L_*$ and note that $a_{L_*}p_{n,L_*}=\widehat{\delta_{\ref{app:lem:bootstrapping2}}}$ by \eqref{eq:endbootstrappingdefs}. Therefore \eqref{eq:compressiblesetcheck} is true for all $2\leq l \leq L_*$ and each compressible level of vectors we may take an infimum over according to Lemma \ref{app:lem:infcomp} is in fact $\Incomp_{l-1}\cap \Comp(a_{l}p_{n,l},b_{n,l}) = \Comp_l$. 
		Finally, continuing from the previous inequality,
		\begin{align*}
			&\P\left(\inf_{v\in\Comp_l}\norm{(Z_n+M)v-w}\leq\gamma_{n,l}/4\,\,,\,\sigma_1(Z_n+M)\leq K_n\right)\\&\qquad\qquad\leq
			\exp(-\tfrac{\widetilde{c}_{\ref{app:lem:bootstrap}}}{16}n\min(np_na_{l-1}p_{n,l-1},1)),
		\end{align*}
		as desired.
	\end{proof}
	
	\subsection{Proof of Proposition \ref{app:prop:userfriendlyincompconrtol}}\label{app:sec:incomp}

	The goal of this section will be to control the probability coming from vectors of the final incompressible set $\Incomp_{L_*}\subseteq\Incomp(a_{L_*},b_{n,L_*})$. Recall that these parameters $a_{L_*}$ and $b_{n,L_*}$ are given in \eqref{eq:endbootstrappingdefs} and defined recursively via Lemma \ref{app:lem:bootstrapping2}, respectively. One of the main results to be proven in this subsection is the following:
	\begin{prop}\label{app:prop:incompcontrol} For any $t\geq 0 $,
		\begin{equation}\label{app:eq:incompcontrol}
			\P\parens{\inf_{v\in \Incomp(a_{L_*},b_{n,L_*})}\norm{(Z_n+M)v}\leq \frac{tb_{n,L_*}}{\sqrt{n}},\sigma_1(Z_n+M)\leq K_n}
		\end{equation}
		is no more than
		\begin{equation}\label{app:eq:bigincompbound}
			\frac{2C_{\ref{app:lem:distprob}}}{a_{L_*}}\parens{\frac{\log(\sqrt{n}b_{n,L_*}^{-1})}{a_{L_*}}\parens{K_nb_{n,L_*}^{-1}\sqrt{t}+\frac{1}{\sqrt{np_n}}}+L_* b_{n,L_*}^{-1} n\exp(-c_{\ref{app:cor:incomptest}}n\widehat{p}_n)+\exp(-c_{\ref{app:lem:distprob}}a_{L_*}^2n)}^{1/4}
		\end{equation}
	\end{prop}
	
	This bound is cumbersome and so we will prove the aforementioned user-friendly bound, Proposition \ref{app:prop:userfriendlyincompconrtol}, in Section \ref{app:sec:asymptotics}. 
	
	For these incompressible vectors, we have the following result, which relates our infimum over incompressible vectors to the average distance of matrix columns to certain subspaces. This is Lemma 3.5 in \cite{rudtiksingularity}:
	\begin{lem}\label{app:lem:distance}
		Let $A\in \cal{M}_n(\C)$ be a random matrix with columns
		$C_1, \ldots, C_n$ and fix $1\leq k\leq n$. Let $H_k$ be the
		span of these columns excluding $C_k$. Then, for any $t \geq 0$,
		\[
		\P\parens{%
			\inf_{ x \in \Incomp(a_{L_*},b_{n,L_*})} \norm{Ax} %
			\leq \frac{t b_{n,L_*}}{\sqrt n}} %
		\leq \frac{2}{a_{L_*} n} \sum_{k=1}^n \P \parens{\dist(C_k , H_k) \leq t}.
		\]
	\end{lem}
	
	Thus, our focus turns to proving the following lemma, which will imply Proposition \ref{app:prop:incompcontrol} after invoking Lemma \ref{app:lem:distance}:
	\begin{lem}\label{app:lem:distprob}
		In the context of Lemma \ref{app:lem:distance} with $A = Z_n+M$, there exist absolute constants $C_{\ref{app:lem:distprob}},c_{\ref{app:lem:distprob}}>0$ such that for each $1\leq k\leq n$ and any $t\geq 0$, the probability
		\[
		\P(\dist(C_k,H_k)\leq t,\sigma_1(Z_n+M)\leq K_n)
		\]
		is no more than
		\begin{equation}\label{app:eq:distprobresult}
			C_{\ref{app:lem:distprob}}\parens{\frac{\log(\sqrt{n}b_{n,L_*}^{-1})}{a_{L_*}}\parens{K_nb_{n,L_*}^{-1}\sqrt{t}+\frac{1}{\sqrt{np_n}}}+L_*b_{n,L_*}^{-1} n\exp(-c_{\ref{app:cor:incomptest}}n\widehat{p}_n)+\exp(-c_{\ref{app:lem:distprob}}a_{L_*}^2n)}^{1/4}.
		\end{equation}

	\end{lem}
	To prove Lemma \ref{app:lem:distprob}, we follow the techniques in \cite[Section 5.5]{andrewelliptic}, especially modifying some of the results therein to work with the sparsity of our elliptic random matrix. Most of these results will be quite similar. Namely, the anticoncentration bounds seen in \cite[Lemmas 5.10, 5.11]{andrewelliptic} will feature our new truncation level $a/ \sqrt{p_n}$ in place of $a$. The following two lemmas are obtained by replacing $a$ by $a/\sqrt{p_n}$ in the analogous lemmas coming from \cite{andrewelliptic}.
	
	\begin{lem}[\cite{andrewelliptic}, Lemma 5.10]\label{app:lem:anticon1}
		Let $\{(\xi_i, \psi_i) : 1 \leq i \leq n\}$ be a collection of independent complex-valued random tuples, and assume there exist $a, b, \sigma > 0$ such that the events
		\[ \mathcal{E}_i := \{ |\xi_i| \leq a/\sqrt{p_n}, |\psi_i| \leq a/\sqrt{p_n} \} \]
		for $1 \leq i \leq n$ satisfy 
		\[ b \leq \min_{1 \leq i \leq n} \P(\mathcal{E}_i), \qquad \sigma^2 \leq \min_{1 \leq i \leq n} \var( \xi_i \mid \mathcal{E}_i), \qquad \sigma^2 \leq \min_{1 \leq i \leq n} \var( \psi_i \mid \mathcal{E}_i). \]
		In addition, assume there exists $\rho < 1$ such that
		\begin{equation} \label{eq:assumprhocorr}
			\max_{1 \leq i \leq n} | \corr( \xi_i \mid \mathcal{E}_i, \psi_i \mid \mathcal{E}_i ) | \leq \rho. 
		\end{equation}
		Then for any $\delta, \tau \in (0,1)$, any $w=(w_i)_{i=1}^n \in \Incomp(\delta, \tau)$, any $w' = (w'_i)_{i=1}^n \in \C^n$ with $\| w' \| \leq 1$, any $J \subset [n]$ with $|J| \geq n(1-\delta/4)$, and any $t \geq 0$, we have
		\begin{align*}
			\sup_{z \in \C} \P &\left( \left| \sum_{i \in J} (\xi_i w_i + \psi_i w_i') - z \right| \leq t \tau \right) \\
			&\leq \frac{C_{\ref{app:lem:anticon1}}}{\sigma \sqrt{\delta b (1 - \rho)}} \sqrt{ \left \lceil \frac{1}{2} \log_2 \left( \frac{2}{\tau^2 \delta} \right) \right \rceil \left \lceil \log_2 \left( \frac{\sqrt{n}}{\tau} \right) \right \rceil } \left( t + \frac{a}{\sqrt{np_n}} \right) + \exp(-\delta n b / 32), 
		\end{align*}
		where $C_{\ref{app:lem:anticon1}} > 0$ is an absolute constant.
	\end{lem}
	\begin{lem}[\cite{andrewelliptic}, Lemma 5.11]\label{app:lem:anticon2}
		Let $\{(\xi_i, \psi_i) : 1 \leq i \leq n\}\cup\{ (\xi_i', \psi_i') : 1 \leq i \leq n\}$ be a collection of independent complex-valued random tuples with the property that $(\xi_i, \psi_i)$ has the same distribution as $(\xi_i', \psi_i')$ for $1 \leq i \leq n$, and assume there exist $a, b, \sigma > 0$ such that the events
		\[ \mathcal{E}_i := \{ |\xi_i| \leq a/\sqrt{p_n}, |\psi_i| \leq a/\sqrt{p_n} \} \]
		for $1 \leq i \leq n$ satisfy 
		\[ b \leq \min_{1 \leq i \leq n} \P(\mathcal{E}_i), \qquad \sigma^2 \leq \min_{1 \leq i \leq n} \var( \xi_i \mid \mathcal{E}_i), \qquad \sigma^2 \leq \min_{1 \leq i \leq n} \var( \psi_i \mid \mathcal{E}_i). \]
		In addition, let $\delta_1, \ldots, \delta_n$ be i.i.d.\ $\{0,1\}$-valued random variables with $\E \delta_i = c_{oo} \in (0,1]$, and assume $\delta_1, \ldots, \delta_n$ are independent of the random tuples $\{(\xi_i, \psi_i), (\xi_i', \psi_i') : 1 \leq i \leq n\}$. Then for any $\delta, \tau \in (0,1)$, any $w=(w_i)_{i=1}^n \in \Incomp(\delta, \tau)$, any $w' = (w'_i)_{i=1}^n \in \C^n$ with $\| w' \| \leq 1$, and any $t \geq 0$, we have
		\begin{align*}
			\sup_{z \in \C} \P &\left( \left| \sum_{i=1}^n (\delta_i \xi_i w_i + \delta_i \xi_i' w_i') - z \right| \leq t \tau \right) \\
			&\leq \frac{C_{\ref{app:lem:anticon2}}}{\sigma \sqrt{c_{oo} \delta b^2 }} \sqrt{ \left \lceil \frac{1}{2} \log_2 \left( \frac{2}{\tau^2 \delta} \right) \right \rceil \left \lceil \log_2 \left( \frac{\sqrt{n}}{\tau} \right) \right \rceil } \left( t + \frac{a}{\sqrt{n p_n}} \right) + \exp(-\delta n c_{oo} b^2 / 16), 
		\end{align*}
		where $C_{\ref{app:lem:anticon2}} > 0$ is an absolute constant. 
	\end{lem}
	
	What is difficult with our sparsity level is that we need to track the parameters $\delta$ and $\tau$ making up the incompressible sets mentioned in these lemmas. Our parameters $a_{L_*}$ and $b_{n,L_*}$ both depend on $n$, since $L_* = L_*(n)$ depends on $n$ as given in \eqref{eq:stoppingindexdef}. We now provide a list of lemmas with truncated proofs that we will be using to prove Lemma \ref{app:lem:distprob}.
	
	\begin{lem}[An analog of \cite{andrewelliptic}, Section 5.5.1]\label{app:lem:bernoullis}
		Let $\delta_1,\ldots,\delta_{n-1}$ be i.i.d. Bernoulli random variables, independent of $Z_n$ with $\E \delta_1 = a_{L_*}/16$. Then there is an absolute constant $c_{\ref{app:lem:bernoullis}}>0$ such that the random subset $\pi := \{1\leq i\leq n-1:\delta_i = 0\}$ has
		\begin{equation}\label{app:eq:picard}
			\abs{\pi^c} \leq \frac{a_{L_*}}{8}n
		\end{equation}
		with probability $1-\exp(-c_{\ref{app:lem:bernoullis}}a_{L_*}^2n)$.
	\end{lem}
	\begin{proof}
		For this proof, set $c_{oo}:=a_{L_*}/8$.
		Since $\abs{\pi^c}=\sum_{i=1}^{n-1}\mathbf{1}_{\set{\delta_i=1}}$ is a sum of independent variables with mean $\E\abs{\pi^c} = (n-1)c_{oo}/2$, we have by Hoeffding's inequality that
		\begin{align*}
			\P(\abs{\pi^c}>c_{oo}n)&=\P(\abs{\pi^c}-\E\abs{\pi^c}>c_{oo}n-(n-1)c_{oo}/2)\\
			&\leq \P(\abs{\pi^c}-\E\abs{\pi^c}>c_{oo}n/2)\\
			&\leq\exp\parens{-\frac{2c_{oo}^2n^2}{4(n-1)}}=
			\exp\left(-\frac{1}{128}a_{L_*}^2n\right).
		\end{align*}
		The lemma holds after setting $c_{\ref{app:lem:bernoullis}} := \frac{1}{128}$.
	\end{proof}
	
	\begin{lem}[An analog of \cite{andrewelliptic}, Lemma 5.17]\label{app:lem:Binvestimate}
		Fix $1\leq k\leq n$ and let $B\in\cal{M}_{n-1}(\C)$ be the random matrix formed after removing the $k$-th row and column from $Z_n+M$. In addition, define $u$ to be the $k$-th row of $Z_n+M$ with the $k$-th entry removed and $v$ be the $k$-th column of $Z_n+M$ with the $k$-th entry removed. Let $(u',v')$ be an independent copy of $(u,v)$ that is also independent of $Z_n$. Then there exist constants $C_{\ref{app:lem:Binvestimate}}, c_{\ref{app:lem:Binvestimate}} > 0$ such that $B$ has the following properties with probability at least $1 - 2C_{\ref{app:cor:incomptest}}L_*b_{n,L_*}^{-1} n\exp(-c_{\ref{app:cor:incomptest}}n\widehat{p}_n)$.  If $\sigma_1(B) \leq K_n$, then:
		\begin{enumerate}[label = (\roman*)]
			\item \label{item:lsv:geqHSnorm} for any $t_0 \geq 0$, with probability at least $1 - \frac{C_{\ref{app:lem:Binvestimate}}}{a_{L_*}} {\log(\sqrt{n}b_{n,L_*}^{-1})} \left( b_{n,L_*}^{-1} t_0 + (np_n)^{-1/2} \right)-\exp(-c_{\ref{app:lem:Binvestimate}}a_{L_*}^2n)$ in $u, u', \pi$, 
			\[ \|(u - u')_{\pi^c} B^{-1}_{\pi^c \times [n-1]} \| \geq t_0 \|B^{-1} \|_{HS}. \]
			\item \label{item:lsv:geqHSnorm2} for any $t_0 \geq 0$, with probability at least $1 - {\frac{C_{\ref{app:lem:Binvestimate}}}{a_{L_*}}} {\log(\sqrt{n}b_{n,L_*}^{-1})} \left( b_{n,L_*}^{-1} t_0 + (np_n)^{-1/2} \right)-\exp(-c_{\ref{app:lem:Binvestimate}}a_{L_*}^2n)$ in $v, v', \pi$, 
			\[ \| B^{-1}_{[n-1] \times \pi^c} (v - v')_{\pi^c} \| \geq t_0 \|B^{-1} \|_{HS}. \]
		\end{enumerate}
	\end{lem}
	
	\begin{proof}
		This proof follows as it does in \cite{andrewelliptic}, but our Lemma \ref{app:lem:anticon2} will be used in place of \cite[Lemma 5.11]{andrewelliptic}.
	\end{proof}
	
	Now for a truncated proof of Lemma \ref{app:lem:distprob}, inspired by the work in \cite[Section 5.5]{andrewelliptic} and following the recipe of steps given there. We note where changes need to be made and the new results that we have to facilitate said changes.
	
	\begin{proof}[Proof of Lemma \ref{app:lem:distprob}]
		We follow the steps as seen in the proof of \cite[Lemma 5.13]{andrewelliptic} with the following changes:
		\begin{enumerate}
			\item \textit{(Constructing a random subset $\pi$)} Our Lemma \ref{app:lem:bernoullis} is used in place of the work done in \cite[Section 5.5.1]{andrewelliptic}.
			
			\item  \textit{(Estimating $\norm{B^{-1}}_{HS}$)} Our Lemma \ref{app:lem:Binvestimate} is used in place of the work done in \cite[Section 5.5.2]{andrewelliptic}.
			
			\item \textit{(Working on the appropriate events)} The mathematical content of Section 5.5.3 to 5.5.6 of \cite{andrewelliptic} does not change. However, given our specific anticoncentration bounds Lemma \ref{app:lem:anticon1} and Lemma \ref{app:lem:anticon2} as well as Lemma \ref{app:lem:Binvestimate}, our work boils down to tracking the dependence on the sparsity parameter $p_n$ as well as the parameters $a_{L_*}$ and $b_{n,L_*}$ defining our final incompressible set of vectors.
		\end{enumerate}
		Tracking these dependences through the final subsections of \cite[Section 5.5]{andrewelliptic}, we arrive at a bound of the form \eqref{app:eq:distprobresult} for Lemma \ref{app:lem:distprob}.
	\end{proof}

	Finally, Proposition \ref{app:prop:incompcontrol} follows after using the bound in Lemma \ref{app:lem:distprob} and multiplying by $\frac{2}{a_{L_*}}$.

	\subsubsection{Asymptotics with $L_*$ and other bootstrapping parameters}\label{app:sec:asymptotics}
	We conclude Section \ref{app:sec:incomp} by proving bounds on the quantities $a_{L_*}$ and $b_{n,L_*}$ when $L_*$ is defined as in \eqref{eq:stoppingindexdef}.
	
	\begin{lem}\label{app:lem:asymptotics}
		Suppose $(p_n)_{n\geq 1}$ satisfies Assumption \ref{ass:sparsitylevel}. Then there exists an absolute constant $C_{\ref{app:lem:asymptotics}}>0$ such that for all $n\geq C_{\ref{app:lem:asymptotics}}$,
		\[
		a_{L_*}^{-1}\leq C_{\ref{app:lem:asymptotics}}(\log n)^2
		\quad\text{and}\quad
		\max_{1\leq l \leq L_*}b_{n,l}^{-1}\leq (C_{\ref{app:lem:asymptotics}}n)^{2L^2}.
		\]
        Furthermore,
        \begin{equation}\label{eq:asymptotics1}
            \max_{2\leq l \leq L_*}\frac{\log(K_n/\gamma_{n,l})}{\log(2/p_n)\log(K_n)} \leq C_{\ref{app:lem:asymptotics}},
        \end{equation}
        where the values of $\gamma_{n,l}$ are defined in Lemma \ref{app:lem:bootstrapping2}.
	\end{lem}
	
	\begin{proof}

		For $a_{L_*}$, we have by definition that $$a_{L_*}^{-1} = \delta_{\ref{app:lem:bootstrapping2}}^{-1}\log(2/p_n)\log(K_n)\leq C(\log n)^2$$ for some absolute constant $C>0$. 
        
        Fix $2\leq l \leq L_*$. To estimate $b_{n,l}$, we re-examine the proof of Lemma \ref{app:lem:bootstrapping2}. We have that
		\begin{align*}
			b_{n,l}^{-1}&\leq \frac{4K_n}{\overline{c}_{\ref{app:lem:bootstrap}}\sqrt{n}}b_{n,l-1}^{-1}\parens{a_{l-1}^{-1}p_{n,l-1}^{-1}}^{1/2}\\&\leq
			\parens{\frac{4K_n}{\overline{c}_{\ref{app:lem:bootstrap}}\sqrt{n}}}^2b_{n,l-2}^{-1}\parens{a_{l-1}^{-1}a_{l-2}^{-1}p_{n,l-1}^{-1}p_{n,l-2}^{-1}}^{1/2}\\\cdots&\leq
			\parens{\frac{4K_n}{\overline{c}_{\ref{app:lem:bootstrap}}\sqrt{n}}}^{l-1}b_{n,1}^{-1}\parens{a_{l-1}^{-1}a_{l-2}^{-1}\cdots a_{1}^{-1}p_{n,l-1}^{-1}p_{n,l-2}^{-1}\cdots p_{n,1}^{-1}}^{1/2}.
		\end{align*}
		Since we have a naive bound of the form $p_{n,l}^{-1}\leq n$ for each $l$, we get an upper bound of $b_{n,l}^{-1}$ which is uniform in $l$, of the form
		\[
		\parens{\frac{4K_n}{\overline{c}_{\ref{app:lem:bootstrap}}\sqrt{n}}}^{L_*-1}\frac{4K_n}{\eps_{\ref{lem:lsv:distance}}}(Cn)^{L_*^2-L_*}\leq (C'n)^{2L_*^2} \leq (C'n)^{2L^2}
		\]
		for some absolute constant $C'>0$. For the term $b_{n,1}$, we have that $b_{n,1}^{-1} \leq CK_n \leq Cn^{3/2}$, so this first term is also bounded by $(C' n)^{2L^2}$.

        For the final claim in this statement of this lemma, we have that for any such $2\leq l \leq L_*$, $\log (K_n/\gamma_{n,l}) \leq \log(4b_{n,l}^{-1})\leq 2L^2\log(C'n)$ and so
        \[
            \frac{\log (K_n/\gamma_{n,l})}{\log(2/p_n)\log(K_n)}\leq \frac{2L^2\log(C'n)}{\log(2/p_n)\log(K_n)} = O\parens{\frac{L^2}{\log(2/p_n)}}=O(1).
        \]
        Breaking into the cases $p_n \leq n^{-1/2}$ and $p_n > n^{-1/2}$ and using Assumption \ref{ass:sparsitylevel} in the former gives the $O(1)$ control. These three results, after choosing $C_{\ref{app:lem:asymptotics}}$ accordingly, together prove the overall lemma.
	\end{proof}

	These bounds lead to our user-friendly analog of Proposition \ref{app:prop:incompcontrol} that has been used in the proof of Theorem \ref{thm:andrewsmallsingval}.
	
	\begin{proof}[Proof of Proposition \ref{app:prop:userfriendlyincompconrtol}]
		We start with the two left-most terms appearing in \eqref{app:eq:bigincompbound}. By Assumption \ref{ass:sparsitylevel}, we have that $np_n\geq n^{q_n}$, where $q_n\log\log n\to\infty$. Thus, $(\log n)^3 = o(n\widehat{p}_n)$, for instance, and hence $\exp(-c_{\ref{app:cor:incomptest}}n\widehat{p}_n)$ dominates $a_{L_*}^{-4}L_*b_{n,L_*}^{-1} n\leq \exp(O((\log n)^3))$.
		
		For the $a_{L_*}^{-4}\exp(-c_{\ref{app:lem:distprob}}a_{L_*}^2n)$ term, by Lemma \ref{app:lem:asymptotics}, we have that $a_{L_*}^2 \geq c(\log n)^{-4}$ where $c = 1/C_{\ref{app:lem:asymptotics}}^2$ and since $L_*\leq L\leq \frac{\log n}{\log\log n}$ for large $n$, it follows that
		\[
		\exp(-c_{\ref{app:lem:distprob}}a_{L_*}^2n)\leq
		\exp(-c'n/(\log n)^4)
		\]
		and so
		\[
		a_{L_*}^{-4}\exp(-c_{\ref{app:lem:distprob}}a_{L_*}^2n)\leq
		C'(\log n)^8\exp(-c'n/(\log n)^4) = o(1)
		\]
		as well.
		
		Finally, we consider the $\frac{a_{L_*}^{-5}\log(\sqrt{n}b_{n,L_*}^{-1})}{\sqrt{np_n}}$ term. The numerator here is no more than
		\[
		a_{L_*}^{-5}\log(\sqrt{n}b_{n,L_*}^{-1}) \leq C''(\log n)^{10}\log^3(C_{\ref{app:lem:asymptotics}}n).
		\]
		Dividing by $\sqrt{np_n}$, this expression is $o(1)$ given our assumption that $q_n\log\log n\to\infty$. In essence, $np_n$ dominates arbitrarily high powers of $\log n$. The desired result now follows by replacing $t$ with $tb_{n,L_*}^{-1}\sqrt{n}$ in Proposition \ref{app:prop:incompcontrol} and absorbing all of the terms appearing in
		\[
		\log(\sqrt{n}b_{n,L_*}^{-1})a_{L_*}^{-5}K_n\sqrt{n}b_{n,L_*}^{-3/2}
		\]
		into the highest order term: $b_{n,L_*}^{-3/2}\leq (C_{\ref{app:lem:asymptotics}}n)^{3L^2}\leq (C_{\ref{app:lem:asymptotics}}n)^{0.5\log^2 n}$ for $n$ large enough.
	\end{proof}

	\appendix
	
	\section{Proof of Lemma \ref{lem:determinetheellipticlaw}}\label{sec:ellipticlawappendix}

	This appendix is dedicated to the proof of Lemma \ref{lem:determinetheellipticlaw}. For this proof, we will be following the Hermitization techniques as presented in \cite[Section 4.5]{aroundthecirc}.
	
	Before beginning the proof, we go through some of the notation we will be using. For the $n\times n$ matrix $n^{-1/2}W_n$ appearing in Lemma \ref{lem:determinetheellipticlaw}, we define the Hermitization of $n^{-1/2}W_n$ to be the matrix $H\in\cal{M}_n(\cal{M}_2(\C))$ given by
	\[
	H_{ij} = \begin{pmatrix}
		0 & n^{-1/2}w_{ij}\\
		n^{-1/2}\overline{w_{ji}}&0
	\end{pmatrix}
	\]
	for each $1\leq i,j\leq n$, where $W_n = (w_{ij})_{1\leq i,j\leq n}$.
	This matrix $H$ is equivalent up to a permutation of the entries to the more classical Hermitization
	\[
	\begin{pmatrix}
		0 & n^{-1/2}W_n\\
		n^{-1/2}W_n^* & 0
	\end{pmatrix}\in\cal{M}_2(\cal{M}_n(\C)).  
	\]
	The block form of the Hermitization $H$ will help to deal with the dependence of the atom variable appearing in Lemma \ref{lem:determinetheellipticlaw}. For the Hermitization $H$, we define the resolvent $R$ for a fixed $\eta\in\C^+:=\set{z\in\C:\Imag(z)>0}$:
	\[
	R := (H-\eta I_{2n})^{-1} = (H-q(0,\eta)\otimes I_n)^{-1}.
	\]
	To keep with the block mentality, we discuss the notation on the right. For any $z\in\C$ and $\eta\in\C^+$, we define
	\[
	q=q(z,\eta)=\begin{pmatrix}
		\eta & z\\
		\overline{z} & \eta
	\end{pmatrix}\in\cal{M}_2(\C).
	\]
	Furthermore, the matrix $q\otimes I_n\in \cal{M}_n(\cal{M}_2(\C))$ will be defined as follows:
	\[
	(q\otimes I_n)_{ii} = q
	\]
	for all $1\leq i \leq n$ and
	\[
	(q\otimes I_n)_{ij} = 0
	\]
	for all $i\leq i\neq j\leq n$. When including the $zI_n$ shift for $n^{-1/2}W_n-zI_n$, we have the resolvent form
	\[  
	R = (H-q(z,\eta)\otimes I_n)^{-1}.
	\]
	If the singular values of $n^{-1/2}W_n-zI_n$ are $\sigma_1\geq\cdots\geq\sigma_n$, then the eigenvalues of $H$ are given by $\pm\sigma_1,\ldots,\pm\sigma_n$. To capture this symmetrization at the level of empirical singular value measures, we define the symmetrization of a probability measure $\nu$ on $\R$ to be the probability measure $\hat\nu$ on $\R$ given by
	\[
	\hat{\nu}(\cdot) = (\nu(\cdot)+\nu(-\cdot))/2.
	\]
	
	Before the proof of Lemma \ref{lem:determinetheellipticlaw}, we will need concentration for the Stieltjes transform of the measures $\hat\nu_{n^{-1/2}W_n-zI_n}$, which is Lemma 7.12 in \cite{ellipticlaw}. The proof follows exactly as it does in \cite{ellipticlaw}.
	
	\begin{lem}[Concentration of Stieltjes transform]\label{lem:concentrationofstieltjes}
		\[
		\var\parens{\int\frac{1}{x-\eta}\,d\hat\nu_{n^{-1/2}W_n-zI_n}(x)} = O_{\eta,z}\parens{\frac{1}{n}}.
		\]
		Furthermore, we have that for the resolvent $R = (H-q(z,\eta)\otimes I_n)^{-1}\in \cal{M}_n(\cal{M}_2(\C))$,
		\begin{align*}
			\var\parens{\frac{1}{n}\sum_{i=1}^n(R_{ii})_{12}} &= O_{\eta,z}\parens{\frac{1}{n}},
			\\\var\parens{\frac{1}{n}\sum_{i=1}^n(R_{ii})_{21}} &= O_{\eta,z}\parens{\frac{1}{n}},
		\end{align*}
	\end{lem}
	
	We now give the proof of the Lemma \ref{lem:determinetheellipticlaw}:
	
	\begin{proof}[Proof of Lemma \ref{lem:determinetheellipticlaw}]
		
		Given the correspondence between measures on $\R$ and their symmetrizations, we will be proving that for almost all $z\in\C$, $\hat\nu_{n^{-1/2}W_n-zI_n}\to\hat\nu_z$ weakly in probability as $n\to\infty$. To this end, and by \cite[Theorem B.9]{BaiSilverstein}, we will be showing that for a fixed $z\in\C$ and a fixed $\eta \in\C^+$ that the Stieltjes transforms
		\begin{equation}\label{eq:resolvent}
			m_{n}(\eta) := \int\frac{1}{x-\eta}\,d\hat{\nu}_{n^{-1/2}W_n-zI_n}(x) = \frac{1}{2n}\tr(H-q(z,\eta)\otimes I_n)^{-1}
		\end{equation}
		converge in probability to some $m(\eta)$ as $n\to\infty$. Necessarily, $m(\eta)$ will be the Stieltjes transform of some symmetric probability measure $\hat\nu_z$ on $\R$ and so we will have $\nu_{n^{-1/2}W_n-zI_n}\to\nu_z$ weakly in probability as $n\to\infty$. 
		
		Throughout the proof, we will be summing $2\times2$ diagonal entries of our resolvent matrix $R$: if
		\[
		R_{kk} = \begin{pmatrix}
			a_k & b_k\\
			c_k & d_k
		\end{pmatrix}
		\]
		for all $1\leq k \leq n$, then as seen in the proof of \cite[Lemma 4.19]{aroundthecirc},
		\begin{equation}\label{eq:eqofa}
			a=a(q):= \frac{1}{n}\sum_{k=1}^na_k= \frac{1}{n}\sum_{k=1}^nd_k.
		\end{equation}
		To connect back to \eqref{eq:resolvent}, it follows that
		\[  
		m_n(\eta) = a(q).
		\]
		For notation's sake, we will also define
		\[
		b=b(q):=\frac{1}{n}\sum_{k=1}^nb_k\quad\text{and}\quad c=c(q) :=\frac{1}{n}\sum_{k=1}^nc_k.
		\]
		
		Throughout this proof, $a$, $b$, and $c$ will depend on $n$, but we will suppress this dependence. By Lemma \ref{lem:concentrationofstieltjes}, it is enough to show that $\E a(q)$, $\E b(q)$, and $\E c(q)$ converges to limits $\alpha$, $\beta$, and $\gamma$, respectively, as $n\to\infty$ and given certain properties of these parameters $\alpha$, $\beta$, and $\gamma$, we will have desirable properties of our limiting measure $\hat\nu_z$.
		
		To begin, we have by Schur block inversion that
		\[
		R_{nn} = \parens{n^{-1/2}\begin{pmatrix}
				0 & w_{nn}\\
				\overline{w_{nn}} & 0
			\end{pmatrix}-q(z,\eta)-Q^*\widetilde{R}Q
		}^{-1}.
		\]
		Here, $Q$ is the last `column' of $W_n$ with the $n$-th `entry' removed. That is, $Q$ is a `vector' of length $n-1$ whose entries are $2\times2$ matrices of the form
		\[
		Q_{i} = n^{-1/2}\begin{pmatrix}
			0 & w_{in}\\
			\overline{w_{ni}} & 0
		\end{pmatrix}
		\]
		for all $1\leq i \leq n-1$.
		Furthermore, $\widetilde{R}$ is the $2(n-1)\times 2(n-1)$ resolvent of the principal, upper-left minor of the matrix appearing in the original resolvent $R$. In other words,
		\[
		\widetilde{R} = (\widetilde{H}-q\otimes I_{n-1})^{-1},        
		\]
		where $\widetilde{H}$ is the Hermitization of the principal upper left $(n-1)\times(n-1)$ minor of  $n^{-1/2}W_n$.
		
		Because we are using this block matrix structure, even with the dependence in our original matrix $W_n$, we still have that $Q$ is independent of $\widetilde{R}$.
		
		Using the independence that remains, our first task is to compute $\E_n[Q^*\widetilde{R}Q]\in \cal{M}_2(\C)$, where $\E_n[\cdot]=\E[\cdot\given\cal{F}_n]$ is the conditional expectation with respect to the $\sigma$-algebra $\cal{F}_n$ generated by the variables $(w_{ij})_{1\leq i,j\leq n-1}$. As discussed above, $\widetilde{R}\in \cal{F}_n$ and $Q$ is independent of $\cal{F}_n$. Thus, using the mean zero hypothesis of the $w_{ij}$ and the independence that exists between entries of $W_n$,
		\[
		\E_n[Q^*\widetilde{R}Q] = \sum_{i,j=1}^{n-1}\E_n[Q^*_i\widetilde{R}_{ij}Q_j]=
		\sum_{i=1}^{n-1}\E_n[Q_i^*\widetilde{R}_{ii}Q_i].
		\]
		If we set
		\[
		\widetilde{R}_{ii} = \begin{pmatrix}
			\widetilde{a}_i & \widetilde{b}_i\\
			\widetilde{c}_i & \widetilde{d}_i
		\end{pmatrix},
		\]
		then for any $1\leq i\leq n-1$,
		\begin{align*}
			\E_n[Q_i^*\widetilde{R}_{ii}Q_i] &= \frac{1}{n}\E_n\left[
			\begin{pmatrix}
				0 & w_{ni}\\
				\overline{w_{in}} & 0
			\end{pmatrix}
			\begin{pmatrix}
				\widetilde{a}_i & \widetilde{b}_i\\
				\widetilde{c}_i & \widetilde{d}_i
			\end{pmatrix}
			\begin{pmatrix}
				0 & w_{in}\\
				\overline{w_{ni}} & 0
			\end{pmatrix}\right]\\&=
			\frac{1}{n}\E_n\begin{pmatrix}
				\abs{w_{ni}}^2\widetilde{d}_i & w_{ni}w_{in}\widetilde{c}_i\\
				\overline{w_{ni}w_{in}}\widetilde{b}_i & \abs{w_{in}}^2\widetilde{a}_i
			\end{pmatrix}\\&=
			\frac{1}{n}
			\begin{pmatrix}
				\widetilde{d}_i & \rho\widetilde{c}_i\\
				\rho\widetilde{b}_i & \widetilde{a}_i
			\end{pmatrix}.
		\end{align*}
		
		The third line here follows from \eqref{eq:psuedocovariancestructure} and \eqref{eq:covariancestructure}. Therefore,
		\[
		\E_n[Q^*\widetilde{R}Q] = \frac{1}{n}\sum_{i=1}^{n-1}\begin{pmatrix}
			\widetilde{d}_i & \rho\widetilde{c}_i\\
			\rho\widetilde{b}_i & \widetilde{a}_i
		\end{pmatrix}
		=
		\frac{1}{n}\sum_{i=1}^{n-1}\begin{pmatrix}
			\widetilde{a}_i & \rho\widetilde{c}_i\\
			\rho\widetilde{b}_i & \widetilde{a}_i
		\end{pmatrix},
		\]
		where the last equality holds by symmetry as it did in \eqref{eq:eqofa}. By Cauchy interlacing, we have the trace of the minor $\widetilde{R}$ and the trace of the entire resolvent $R$ are not too different: $\abs{\sum_{k=1}^n a_k - \sum_{k=1}^{n-1} \widetilde{a}_k} = O(\frac{1}{\Imag(\eta)})$. 
		
		For the off-diagonal entries, we can achieve similar control. To see this in the case of the $b_k$ and $\widetilde{b}_k$, we can realize the sum $\sum_{k=1}^n b_k$ as the trace of $PRQ$, where $P$ and $Q$ are partial isometries. After writing $\sum_{k=1}^n \widetilde{b}_k$ similarly, using the resolvent identity, the low rank of the difference, and the operator norm of the resolvents themselves gives similar $O(\frac{1}{\Imag(\eta)})$ control of the quantity $\abs{\sum_{k=1}^n b_k - \sum_{k=1}^{n-1} \widetilde{b}_k}$.   
		
		Thus, we have shown that
		\[
		\E_n[Q^*\widetilde{R}Q] = \begin{pmatrix}
			a & \rho c\\
			\rho b  & a
		\end{pmatrix} + \eps_1,
		\]
		where $\norm{\eps_1}_{HS} = o(1)$.
		
		However, the more important approximation is
		\[
		Q^*\widetilde{R}Q \approx \E_n[Q^*\widetilde{R}Q],
		\]
		which we will achieve by bounding the trace of
		\[
		\eps_2 = \E_n[(Q^*\widetilde{R}Q - \E_n[Q^*\widetilde{R}Q])^*(Q^*\widetilde{R}Q - \E_n[Q^*\widetilde{R}Q])].
		\]
		This is desirable because by Jensen's inequality,
		\[
		\left(\E_n\norm{Q^*\widetilde{R}Q-\E_n[Q^*\widetilde{R}Q]}_{HS}\right)^2\leq \E_n\norm{Q^*\widetilde{R}Q-\E_n[Q^*\widetilde{R}Q]}_{HS}^2 = \tr(\eps_2).
		\]
		A routine calculation shows that each entry of $\eps_2$ is itself a quadratic or bilinear form coming from entries of $Q$ and $\widetilde{R}$. For example, if we define $\widetilde{R}^{11}\in \cal{M}_{n-1}(\C)$ by
		\[
		\widetilde{R}^{11}_{ij} = (\widetilde{R}_{ij})_{11}
		\]
		and similarly define $Q^{21}\in \C^{n-1}$ by
		\[
		Q^{21}_i = (Q_i)_{21},
		\]
		then
		\[
		(Q^*\widetilde{R}Q)_{11} = (Q^{21})^*\widetilde{R}^{22}Q^{21}. 
		\]
		Bringing the other entries together, we conclude that
		\[
		Q^*\widetilde{R}Q = \begin{pmatrix}
			(Q^{21})^*\widetilde{R}^{22}Q^{21} & (Q^{21})^*\widetilde{R}^{21}Q^{12}\\
			(Q^{12})^*\widetilde{R}^{12}Q^{21} &
			(Q^{12})^*\widetilde{R}^{11}Q^{12}
		\end{pmatrix}=:A.
		\]
		Finally,
		\[
		\eps_2 = \begin{pmatrix}
			\var_n(A_{11})+\var_n(A_{21}) & \cov_n(A_{12},A_{11})+\cov_n(A_{22},A_{21})\\
			\cov_n(A_{11},A_{12})+\cov_n(A_{21},A_{22}) & \var_n(A_{12})+\var_n(A_{22}),
		\end{pmatrix}
		\]
		where $\var_n(Z) = \E_n\abs{Z-\E_nZ}^2$ and $\cov_n(Z,W) = \E_n[(Z-\E_nZ)\overline{(W-\E_n W)}]$. To show that $\tr(\eps_2)=o(1)$, we show that each diagonal entry of $\eps_2$ is $o(1)$.
		
		For the $A_{11}$ case,
		\[
		\abs{\E_n A_{11}}^2 = \frac{1}{n^2}\sum_{i,j=1}^{n-1}\widetilde{R}^{22}_{ii}\overline{\widetilde{R}^{22}_{jj}}
		\]
		and
		\[
		\E_n\abs{A_{11}}^2 = \frac{1}{n^2}\sum_{i,j,k,l=1}^{n-1}\widetilde{R}_{ij}^{22}\overline{\widetilde{R}_{kl}^{22}}\E_n[w_{ni}\overline{w_{nj}}\cdot\overline{w_{nk}}w_{nl}].
		\]
		Using the mean zero hypothesis, independence, and the finite fourth moment of these Gaussian variables, we have a Wick-type result (see \cite{wick}, say):
		\[
		\E_n\abs{A_{11}}^2=
		\frac{1}{n^2}\sum_{i,j=1}^{n-1}\widetilde{R}^{22}_{ii}\overline{\widetilde{R}^{22}_{jj}}+\frac{1}{n^2}\sum_{i,j=1}^{n-1}\abs{\widetilde{R}^{22}_{ij}}^2+
		\frac{\abs{\mu_{22}}^2}{n^2}\sum_{i,j=1}^{n-1}\widetilde{R}^{22}_{ij}\overline{\widetilde{R}^{22}_{ji}}.
		\]
		Therefore,
		\begin{equation}\label{eq:variancecalc}
			\var_n(A_{11}) \leq \frac{1}{n^2}\norm{\widetilde{R}^{22}}_{HS}^2+
			\frac{\abs{\mu_{22}}^2}{n^2}\abs{\sum_{i,j=1}^{n-1}\widetilde{R}^{22}_{ij}\overline{\widetilde{R}^{22}_{ji}}}.
		\end{equation}
		The first term on the right-hand side of \eqref{eq:variancecalc} is no greater than
		\[
		\frac{1}{n^2}\norm{\widetilde{R}^{22}}_{HS}^2\leq \frac{1}{n}\norm{\widetilde{R}}^2=O\parens{\frac{\Imag(\eta)^{-2}}{n}}=o(1).
		\]
		These two inequalities come from the fact that for any $n\times n$ matrix $M$, $\norm{M}_{HS}\leq \sqrt{n}\norm{M}$ and if $M$ is Hermitian, then the resolvent satisfies
		\[
		\norm{(M-\eta I_n)^{-1}}\leq \Imag(\eta)^{-1} 
		\]
		for any $\eta\in\C^+$.
		The summand on the right in \eqref{eq:variancecalc}, by Cauchy--Schwarz, is no greater than
		\[
		\frac{\abs{\mu_{22}}^2}{n^2}\abs{\sum_{i,j=1}^{n-1}\widetilde{R}^{22}_{ij}\overline{\widetilde{R}^{22}_{ji}}} \leq   
		\frac{\abs{\mu_{22}}^2}{n^2}\norm{\widetilde{R}^{22}}_{HS}^2\leq
		\frac{\abs{\mu_{22}}^2}{n}\norm{\widetilde{R}^{22}}^2=
		O\parens{\frac{\Imag(\eta)^{-2}}{n}}=o(1).
		\]

		A similar result holds for the other variances $\var_n(A_{ij})$ for $1\leq i,j \leq 2$. Summing the diagonal of $\eps_2$ gives a final estimate of $\tr(\eps_2)=o(1)$.
		Therefore, the matrix
		\begin{equation}\label{eq:dmatrix}
			D = n^{-1/2}\begin{pmatrix}
				0 & w_{nn}\\
				\overline{w_{nn}} & 0 
			\end{pmatrix}
			- Q^*\widetilde{R}Q + \E\begin{pmatrix}
				a & \rho c\\
				\rho b & a
			\end{pmatrix}
		\end{equation}
		has a Hilbert--Schmidt norm satisfying $\E\norm{D}_{HS}=o(1)$. The expected Hilbert--Schmidt norm of the left-most summand in \eqref{eq:dmatrix} is $O(n^{-1/2})$. To deal with the other two summands, we use the triangle inequality and our estimates:
		\begin{align*}
			Q^*\widetilde{R}Q &= \E_n[Q^*\widetilde{R}Q]+\eps_4,\\
			\E_n[Q^*\widetilde{R}Q] &= 
			\begin{pmatrix}
				a & \rho c\\
				\rho b & a
			\end{pmatrix}+\eps_1,\\
			\begin{pmatrix}
				a & \rho c\\
				\rho b & a
			\end{pmatrix} &= \E\begin{pmatrix}
				a & \rho c\\
				\rho b & a
			\end{pmatrix} + \eps_3.
		\end{align*}
		
		Here, $\eps_1$ is as above with $\norm{\eps_1}_{HS}=o(1)$. The quantity $\E\norm{\eps_3}_{HS}=o(1)$ is coming from Lemma \ref{lem:concentrationofstieltjes}. Finally, we have $\E\norm{\eps_4}_{HS}= \E[\E_n\norm{\eps_4}_{HS}]\leq \sqrt{\tr(\eps_2)}=o(1)$.
		By \eqref{eq:dmatrix} and a computation, we have
		\[
		R_{nn}D\parens{q+\E\begin{pmatrix}
				a & \rho c\\
				\rho b & a
		\end{pmatrix}}^{-1} = R_{nn}+\parens{q+\E\begin{pmatrix}
				a & \rho c\\
				\rho b & a
		\end{pmatrix}}^{-1}.
		\]
		Since $R_{nn}$ and $\parens{q+\E\begin{pmatrix}
				a & \rho c\\
				\rho b & a
		\end{pmatrix}}^{-1}$ both have operator norms (deterministically) bounded by $O(\Imag({\eta})^{-1})$, it must be that
		\[
		\E R_{nn} = -\parens{q+\E\begin{pmatrix}
				a & \rho c\\
				\rho b & a
		\end{pmatrix}}^{-1} + \eps_5,
		\]
		where $\norm{\eps_5}_{HS}=o(1)$. Finally, repeating this Schur complement argument for other rows and columns and averaging, we have that
		\begin{equation}\label{eq:limit}
			\frac{1}{n}\sum_{i=1}^n\E R_{ii} = \begin{pmatrix}
				\E a & \E b\\
				\E c & \E a
			\end{pmatrix} = -\parens{q+\begin{pmatrix}
					\E a & \rho \E c\\
					\rho \E b & \E a
			\end{pmatrix}}^{-1} +\eps_6,
		\end{equation}
		where $\norm{\eps_6}_{HS}=o(1)$. It follows that any three accumulation points $\alpha$, $\beta$, and $\gamma$ of the respective sequences $\E a$, $\E b$, and $\E c$ satisfy the self-consistent equation
		\begin{equation}\label{eq:selfconsistenteq1}
			\begin{pmatrix}
				\alpha & \beta\\
				\gamma & \alpha
			\end{pmatrix} = -\parens{q+\begin{pmatrix}
					\alpha & \rho \gamma\\
					\rho \beta & \alpha
			\end{pmatrix}}^{-1} = \begin{pmatrix}
				-\eta-\alpha & -z-\rho\gamma\\
				-\overline{z}-\rho\beta & -\eta-\alpha
			\end{pmatrix}^{-1}.
		\end{equation}
		This is the exact equation appearing in \cite[Lemma C.2]{largeelliptic} so that $\alpha$ must satisfy the following equation:
		\begin{equation}{\label{eq:selfconsistenteq}}
			\frac{1}{(\alpha+\eta)\alpha}+1=\frac{\Real(z)^2}{(\eta+(1+\rho)\alpha)^2}+\frac{\Imag(z)^2}{(\eta+(1-\rho)\alpha)^2}.
		\end{equation}
		Similar equations can be derived explicitly for $\beta$ and $\gamma$ as well, but we will not need them.
        
		By the uniform boundedness of $\abs{a},\abs{b},\abs{c}\leq \frac{1}{\Imag(\eta)}$ in $n$ and Jensen's inequality, we have that the sequences $\E a$, $\E b$, and $\E c$ are bounded.  Any subsequence of these sequences has a further subsequence that is convergent. Necessarily, the limits of any convergent subsequence of $\E a$, $\E b$, or $\E c$ are determined by \eqref{eq:selfconsistenteq} and the equations for $\beta$ and $\gamma$ found in \cite[Remark C.3]{largeelliptic}.  Thus, we may find, after taking further and further subsequences, a common subsequence along which each of $\E a$, $\E b$, and $\E c$ converge. Therefore, each of the sequences $\E a$, $\E b$, and $\E c$, converge to $\alpha$, $\beta$, and $\gamma$, respectively and these limits satisfy \eqref{eq:selfconsistenteq1}.

        However, $\alpha$ satisfying \eqref{eq:selfconsistenteq} does not guarantee that it is the Stieltjes transform of a probability measure on $\R$. To rectify this, we note that we have tightness for the sequence of deterministic probability measures $\overline{\nu}_{n,z} :=\E\widehat{\nu}_{n^{-1/2}W_n-zI_n}$ defined via duality. This comes from the routine calculation
        \[
            \int x^2\,d\overline{\nu}_{n,z}=\frac{1}{n}\E \norm{n^{-1/2}W_n-zI_n}_{HS}^2 = O_z(1)
        \]
        uniformly in $n$. Therefore, we may pick a further subsequence along which $\overline{\nu}_{n,z}$ converges weakly to a probability measure $\mu_z$ on $\R$. Along this final subsequence, we have that $\E a \to \alpha$, $\alpha$ satisfies $\eqref{eq:selfconsistenteq}$, and that $\alpha$ is the Stieltjes transform of some symmetric probability measure on $\R$ because each $\E a$ is, by \eqref{eq:resolvent}. By Lemma C.6 in \cite{largeelliptic}, for each $z\in\C$, there is a unique probability measure $\hat{\nu}_z$ on $\R$ such that
		\[
		\alpha = \int\frac{1}{x-\eta}\,d\hat{\nu}_z(x)
		\]
		is a solution of \eqref{eq:selfconsistenteq}. Note here and above that $\alpha$ does depend on $z\in\C$ and $\eta\in\C^+$, but this dependence is suppressed. Therefore, since we have shown that $a(q)$ converges in probability to a solution of \eqref{eq:selfconsistenteq} for each $\eta\in\C^+$ and $z\in\C$, we have again by \cite[Theorem B.9]{BaiSilverstein} that $\hat\nu_{n^{-1/2}W_n-zI_n}\to \hat\nu_z$ and thus $\nu_{n^{-1/2}W_n-zI_n}\to \nu_z$ weakly in probability as $n\to\infty$.
		
		It still remains to show that $(\nu_z)_{z\in\C}$ satisfies \eqref{eq:determineellipticspecial}. To do this, we notice that in the limit \eqref{eq:limit}, there is no dependence on the quantities $\mu_{11}$, $\mu_{12}$, or $\mu_{22}$. Because the convergence of the Stieltjes transforms completely determines the weak convergence of the empirical singular value measures and this convergence is guaranteed for any values of $\mu_{11}$, $\mu_{12}$, and $\mu_{22}$, the limit of the weak convergence will be the same given \textit{any} values of the $\mu_{ij}$. Thus, we have the freedom to examine the weak convergence under specific values of these $\mu_{ij}$.

		With $\mu_{11}=\mu_{12}=\mu_{22}=0$, the real vector $\Psi = (\Real(\psi_1),\Imag(\psi_1),\Real(\psi_2),\Imag(\psi_2))^T$ has the following covariance structure:
		\[
		\E[\Psi\Psi^T] = \begin{pmatrix}
			1/2 & 0 & \rho/2 & 0\\
			0 & 1/2 & 0 & -\rho/2\\
			\rho/2 & 0 & 1/2 & 0\\
			0 & -\rho/2 & 0 &1/2
		\end{pmatrix}.
		\]
		But this means that $(\psi_1,\psi_2)$ is coming from the $(\frac{1}{2},\rho)$-family as in Definition 1.6  of \cite{ellipticlaw}. Hence by Lemma 7.11 of \cite{ellipticlaw} and (jj) of Lemma \ref{lem:hermitization}, it must be that $(\nu_z)_{z\in\C}$ satisfies \eqref{eq:determineellipticspecial}. From the discussion of the previous paragraph, this limiting family $(\nu_z)_{z\in\C}$ will be the same for the arbitrary values of the $\mu_{ij}$ given at the beginning of the proof.
	\end{proof}

	\section{Auxiliary proofs from Sections \ref{sec:uniformintegrability} and \ref{sec:lsvappendix}} \label{sec:singvalappendix}
	
	\subsection{Largest singular value}\label{sec:app:largestsingval}
	We begin with a proof of our control of the largest singular value of our matrix $M_n-zI_n$.
	\begin{proof}[Proof of Lemma \ref{lem:largestsv}]
		By Markov's inequality,
		\begin{align*}
			\P(\sigma_1(M_n-zI_n)\geq n\sqrt{p_n}) &=
			\P(\norm{p_n^{-1/2}(X_n+F_n)-\sqrt{n}zI_n}\geq n^{3/2}\sqrt{p_n})\\&\leq
			\P(\norm{p_n^{-1/2}(X_n+F_n)-\sqrt{n}zI_n}_{HS}^2\geq n^{3}p_n)\\&\leq
			\frac{\E\norm{p_n^{-1/2}(X_n+F_n)-\sqrt{n}zI_n}_{HS}^2}{n^3p_n}.
		\end{align*}
		We compute the expectation of the squared Hilbert--Schmidt norm of the given matrix:
		\begin{align}
			\nonumber\E\norm{p_n^{-1/2}(X_n+F_n)-\sqrt{n}zI_n}_{HS}^2
			&\leq 3\E\parens{p_n^{-1}\norm{X_n}_{HS}^2
				+p_n^{-1}\norm{F_n}_{HS}^2+n^2\abs{z}^2}\\\nonumber&= 3\parens{n^2\abs{z}^2+p_n^{-1}\sum_{1\leq i,j\leq n}\E[\delta_{ij}^{(n)}\abs{x_{ij}}^2]+p_n^{-1}\norm{F_n}_{HS}^2}\\\label{eq:hilbertshcmidtbound}&\leq
			3n^2\parens{\abs{z}^2+1+\var(x_{11})+n^{-2}p_n^{-1}
				\norm{F_n}^2_{HS}}.
		\end{align}
		Note that here we used the assumption of finite variance on the i.i.d.\ diagonal elements of our matrix model. The result follows after plugging this bound for $\E\norm{p_n^{-1/2}(X_n+F_n)-\sqrt{n}zI_n}_{HS}^2$ into the above string of inequalities.
	\end{proof}
	
	\subsection{Control of compressible vectors at the first stage}\label{sec:app:firststagecompressibles}
	
	The section is dedicated to the proof of Lemma \ref{app:lem:beginningcomp}, which guarantees a starting point in the proof of Proposition \ref{prop:leastsingularvalue}. We recall that $Z_n$ is the random matrix presented in Theorem \ref{thm:andrewsmallsingval}. Also, $M$ is a deterministic matrix in $\cal{M}_n(\C)$.
	
	\begin{lem}[Distance of a random vector to a deterministic subspace] \label{lem:lsv:distance}
		Let $\Col_i$ be the $i$-th column of $Z_n+M$ and $\Col_i^{(i)}$ be the $i$-th column of $Z_n+M$ with the $i$-th entry removed. There exist constants $\eps_{\ref{lem:lsv:distance}}, C_{\ref{lem:lsv:distance}}, c_{\ref{lem:lsv:distance}}, \delta_{\ref{lem:lsv:distance}} > 0$ such that for all $1 \leq i \leq n$ and any deterministic subspace $H$ of $\mathbb{C}^{n-1}$ with $1 \leq \dim(H) \leq \delta_{\ref{lem:lsv:distance}} n$, we have 
		\[ \P \left(\dist(\Col_i^{(i)}, H) \leq \eps_{\ref{lem:lsv:distance}} \sqrt{n} \right) \leq C_{\ref{lem:lsv:distance}} \exp(-c_{\ref{lem:lsv:distance}} np_n).
		\]
	\end{lem}
	\begin{proof}
		The proof follows the same arguments as those given in the proof of \cite[Lemma 5.4]{andrewelliptic} modulo the following changes:
		
		\begin{itemize}
			\item The variables $Y_k$ defined in the proof of \cite[Lemma 5.4]{andrewelliptic} satisfy the bound $\abs{Y_k}\leq 2ap_n^{-1/2}$ and thus the upper bound guaranteed by Talagrand's inequality is of the form $C\exp(-cp_n\frac{t^2}{a^2})$.
			
			\item The lower bound of the expected squared distance appearing in the proof of \cite[Lemma 5.4]{andrewelliptic} is again lower bounded by $c''\sigma^2 n$ for some constant $c''>0$. This result uses that fact that $np_n\to\infty$.
		\end{itemize}  
	\end{proof}
	Applying a union bound, we get the following corollary of Lemma \ref{lem:lsv:distance}.
	
	\begin{cor} \label{app:cor:mindistance}
		There exist absolute constants $C_{\ref{app:lem:beginningcomp}}, c_{\ref{app:lem:beginningcomp}}> 0$ such that for any deterministic subset $\pi \subset [n]$ with $|\pi| \leq \delta_{\ref{lem:lsv:distance}} n $ and any deterministic $u \in \mathbb{C}^n$, we have
		\[ \P \left( \min_{i \in \pi} \dist(\Col_i, H_i) \leq \eps_{\ref{lem:lsv:distance}} \sqrt{n} \right) \leq C_{\ref{app:lem:beginningcomp}}\abs{\pi} \exp(-c_{\ref{app:lem:beginningcomp}}  np_n), \]
		where $H_i := \Span\left(\{ \Col_j : j \in \pi, j \neq i\}\cup\{u\} \right)$.  
	\end{cor}
	\begin{proof}
		The proof of this corollary follows as it does in \cite[Corollary 5.5]{andrewelliptic}, but will use Lemma \ref{lem:lsv:distance} everywhere \cite[Lemma 5.4]{andrewelliptic} is used in the proof of \cite[Corollary 5.5]{andrewelliptic}.
	\end{proof}

	\subsection{Results from Section \ref{sec:bootstrap}}\label{app:sec:bootstrapnew}
    
	We start with the proof of Lemma \ref{app:lem:infcomp}.
	\begin{proof}[Proof of Lemma \ref{app:lem:infcomp}]
		Fix $\delta>0$ to be chosen later and let $\cal{E}$ be the event whose probability we are bounding in the statement of Lemma \ref{app:lem:infcomp}. To utilize the pointwise bound we assume on vectors coming from $I$, we first construct a $(3b_n/2)$-net  $\cal{N}$ of $\Comp(\delta\widehat{q},b_n)$ having cardinality $\abs{\cal{N}} \leq \binom{n}{k}(\frac{C}{b_n})^{2k}$, where $k = \floor{ \delta n \widehat{q}}$ and $C>0$ is a universal constant.
		
		To bring in $I$, we define
		\[
		\cal{N}' := \set{x\in\cal{N}:B(x,3b_n/2)\cap I\cap\Comp(\delta\widehat{q},b_n)\neq\varnothing}
		\]
		where $B(x,3b_n/2)$ is the ball in $\C^n$ centered at $x$ with radius $3b_n/2$. Thus, for each $x\in \cal{N}'\subseteq\cal{N}$, we may choose a representative $v\in B(x,3b_n/2)\cap I\cap\Comp(\delta\widehat{q},b_n)\subseteq I\cap\Comp(\delta\widehat{q},b_n)$. If $\cal{N}'$ is empty, then so is $I\cap \Comp(\delta\widehat{q},b_n)$, making \eqref{app:eq:infcomp1} trivial. We finally define
		\[
		\cal{N}_{I} := \{v_x=v\in I\cap\Comp(\delta\widehat{q},b_n):x\in\cal{N}'\}.
		\]
		which is a $(3b_n)$-net of $I\cap\Comp(\delta\widehat{q},b_n)$ by construction.
		
		Now, if $\cal{E}$ holds, there exists a $v\in I\cap\Comp(\delta \widehat{q},b_n)$ such that $\norm{(Z_n+M)v-w}\leq \gamma/4$ and $\sigma_1(Z_n+M)\leq K_n$. Then, there exists some $y\in \cal{N}_I$ such that $\norm{v-y}\leq 3b_n$. Therefore, on the event $\cal{E}$, the vector $y$ satisfies
		\[
		\norm{(Z_n+M)y-w}\leq \gamma/4+K_n\norm{v-y}\leq \gamma/4+K_n(3b_n)=\gamma,
		\]
		by our definition of $b_n = \frac{\gamma}{4K_n}$. Via a union bound over the net $\cal{N}_I$ and our pointwise bound \eqref{app:eq:pointwiseboundonI} for vectors coming from $I$, we have bounded the probability in \eqref{app:eq:infcomp1} above by
		\[
		\binom{n}{\floor{\delta\widehat{q}n}}\parens{\frac{C}{b_n}}^{2\floor{\delta\widehat{q}n}}\exp(-cnq) \leq \exp(-cnq/4),
		\]
		for all $n\gg1$. For this inequality, we heavily use our assumption \eqref{eq:gammahypothesis} on $\gamma$ and choose $\delta=\delta_{\ref{app:lem:infcomp}}$ small enough depending on $c$ and $K$.
	\end{proof}
	
	As a final corollary of the results of this section, we give a result that guarantees the incompressibility of certain vectors at the final stage with high probability:
	\begin{cor}[\cite{andrewelliptic}, Corollary 5.8]\label{app:cor:incomptest}
		There exist absolute constants $C_{\ref{app:cor:incomptest}},c_{\ref{app:cor:incomptest}}>0$ such that, for any deterministic $u \in \C^{n}$ with $u \neq 0$, we have
		\[ \P \left( \frac{  (Z_n+M)^{-1} u }{ \| (Z_n+M)^{-1} u \| } \not\in \Incomp_{L_*}, \sigma_1(Z_n+M) \leq K_n \right) \leq C_{\ref{app:cor:incomptest}}L_* b_{n,L_*}^{-1} \exp(-c_{\ref{app:cor:incomptest}} n \widehat{p}_n). \]
	\end{cor}
	\begin{proof}
		Let $v = (Z_n+M)^{-1}u$ so that $(Z_n+M)v-u=0$ and $v/\norm{v}\not\in \Incomp_{L_*}$ given the event we are trying to control. Since
		\[
		\field S^{n-1} = \parens{\bigcup_{l=1}^{L_*}\Comp_l}\cup \Incomp_{L_*},
		\]
		we may union bound
		\begin{align*}
			&\P \left( \frac{  (Z_n+M)^{-1} u }{ \| (Z_n+M)^{-1} u \| } \not\in \Incomp_{L_*}, \sigma_1(Z_n+M) \leq K_n \right)\\
			&\qquad\qquad\leq\sum_{l=1}^{L_*}\P\parens{\inf_{v/\norm{v}\in \Comp_l}{\frac{\norm{(Z_n+M)v-u}}{\norm{v}}}=0,\sigma_1(Z_n+M)\leq K_n}\\
			&\qquad\qquad\leq\sum_{l=1}^{L_*}\P\parens{\inf_{v/\norm{v}\in \Comp_l}{\frac{\norm{(Z_n+M)v-u}}{\norm{v}}}\leq{\frac{\gamma_n} 8},\sigma_1(Z_n+M)\leq K_n},
		\end{align*}
		where $\gamma_n := \min(\gamma_{n,1},\gamma_{n,2},\ldots,\gamma_{n,L_*})$ and $\gamma_{n,1} := 4K_nb_{n,1}$.
		
		The obstruction to applying Lemma \ref{app:lem:bootstrapping2} for a vector $v/\norm{v}\in\Comp_l$ is that the deterministic vector in the quotient
		\[
		\frac{\norm{(Z_n+M)v-u}}{\norm{v}} = \norm{(Z_n+M)\frac{v}{\norm{v}}-\frac{u}{\norm{v}}}
		\]
		includes the norm of our compressible vector $v/\norm{v}$. We get around this by comparing $u/\norm{v}$ to $u/\norm{u}$. First, we know by the triangle inequality that
		\[
		\norm{\frac{u}{\norm{v}}}\leq \frac{\gamma_n}{8}+K_n
		\]
		via the assumption on $\sigma_1(Z_n+M)\leq K_n$. Next, we choose a $\frac{\gamma_n}{8}$-net $\cal{N}$ of the interval $[-\frac{\gamma_n}{8}-K_n,\frac{\gamma_n}{8}+K_n]\subseteq \R$ which can be chosen to have cardinality satisfying $\abs{\cal{N}}\leq C_{\ref{app:cor:incomptest}}K_n/\gamma_n$ for some absolute constant $C_{\ref{app:cor:incomptest}}>0$. Therefore, there exists an element $c_0\in\cal{N}$ such that
		\[
		\abs{\frac{\norm{u}}{\norm{v}}-c_0}\leq {\frac{\gamma_n}{8}}\Longrightarrow\norm{\frac{u}{\norm{v}}-c_0\frac{u}{\norm{u}}}\leq \frac{\gamma_n}{8}.
		\]
		Another application of the triangle inequality gives
		\[
		\norm{(Z_n+M)\frac{v}{\norm{v}}-c_0\frac{u}{\norm{u}}}\leq \frac{\gamma_n}{4}.
		\]
		Thus, we have bounded
		\begin{align*}
			&\P\parens{\inf_{v/\norm{v}\in \Comp_l}{\frac{\norm{(Z_n+M)v-u}}{\norm{v}}}\leq{\frac{\gamma_n}{8}},\sigma_1(Z_n+M)\leq K_n}\\
			&\qquad\qquad\leq\P\parens{\inf_{c_0\in\cal{N}}\inf_{v/\norm{v}\in\Comp_l}{\norm{(Z_n+M)\frac{v}{\norm{v}}-c_0\frac{u}{\norm{u}}}}\leq{\frac{\gamma_n}{4}},\sigma_1(Z_n+M)\leq K_n}.
		\end{align*}
		By union bounding over the net $\cal{N}$, using the cardinality bound on $\cal{N}$, and applying Lemma \ref{app:lem:bootstrapping2} for each $2\leq l\leq L_*$, we have
		\begin{align*}
			\P\parens{\inf_{v/\norm{v}\in\Comp_l}{\frac{\norm{(Z_n+M)v-u}}{\norm{v}}}\leq{\frac{\gamma_n}{8}},\sigma_1(Z_n+M)\leq K_n}&\leq
			C_{\ref{app:cor:incomptest}}\frac{K_n}{\gamma_n}\exp(-\tfrac{\widetilde{c}_{\ref{app:lem:bootstrap}}}{16}n\min(a_{l-1}p_{n,l},1))\\&\leq
			C_{\ref{app:cor:incomptest}}\frac{K_n}{\gamma_n}\exp(-c_{\ref{app:cor:incomptest}}n\widehat{p}_n),
		\end{align*}
		by setting $c_{\ref{app:cor:incomptest}} = \frac{\widetilde{c}_{\ref{app:lem:bootstrap}}}{16}$. A similar bound follows for the $l=1$ term $\gamma_{n,1}$, but from Lemma \ref{app:lem:beginningcomp} rather than Lemma \ref{app:lem:bootstrapping2}.
        The corollary follows after summing $l$ from 1 to $L_*$ and using that, by construction,
		\[
		\frac{K_n}{\gamma_n} = \frac14\max(b_{n,1}^{-1},\ldots,b_{n,L_*}^{-1})\leq b_{n,L_*}^{-1}.
		\]
	\end{proof}
	
	\bibliographystyle{siam}
	\bibliography{refs}
	
\end{document}